\documentclass{amsart}

\usepackage[english]{babel}
\usepackage[utf8]{inputenc}
\usepackage{xcolor}

\usepackage{enumitem}
\usepackage{amsmath}
\usepackage{amsfonts}
\usepackage{amssymb}
\usepackage{mathrsfs}
\usepackage{mathtools}
\usepackage{caption}
\usepackage[hidelinks]{hyperref}
\usepackage{amsthm}
\usepackage[norefs, ignoreunlbld, nocites]{refcheck}
\usepackage[norefs, nocites]{refcheck}
\newtheorem{theorem}{Theorem}[section]
\newtheorem{lemma}[theorem]{Lemma}
\newtheorem{proposition}[theorem]{Proposition}

\theoremstyle{definition}
\newtheorem{example}[theorem]{Example}
\newtheorem{remark}[theorem]{Remark}
\newtheorem*{acknowledgements}{Acknowledgements}

\author{Bart Michels}
\title[Periods and oscillatory integrals for maximal flat submanifolds]{The maximal growth of toric periods and oscillatory integrals for maximal flat submanifolds}
\thanks{Universit\'{e} Sorbonne Paris Nord, LAGA, CNRS, UMR 7539,  F-93430, Villetaneuse, France. Email: \texttt{michels@math.univ-paris13.fr}}
\subjclass[2010]{11F03, 11F72}

\begin{document}

\begin{abstract}We prove a new omega result for toric periods of Hecke-Maass forms on compact locally symmetric spaces associated to forms of $\mathbf{PGL}_3$. This is motivated by conjectures on the maximal growth of $L$-functions as well as by questions about the size of automorphic periods. We also prove a mean square asymptotic result for maximal flat periods on more general locally symmetric spaces of non-compact type, which takes as main input bounds for real relative orbital integrals.
\end{abstract}

\maketitle

\section{Introduction}

In this article we prove that toric periods of Hecke-Maass forms on compact locally symmetric spaces associated to forms of $\mathbf{PGL}_3$ exhibit nontrivial growth. In particular, the result implies the following.

\begin{theorem}\label{largevaluesslthreesimplified}Let $\mathbf G$ be an anisotropic $\mathbb Q$-form of $\mathbf {PGL}_{3}$ with $\mathbf G(\mathbb R)$ noncompact. Let $X$ be an associated locally symmetric space and $(f_j) \in L^2(X)$ an orthonormal basis of Hecke-Maass forms with Laplacian eigenvalues $\lambda_j \geq 0$. Let $\mathbf H \subset \mathbf G$ be a maximal torus with the same $\mathbb R$-rank $r$ as $\mathbf G$, and denote by $\mathscr P_{\mathbf H}(f_j)$ the $\mathbf H$-period of $f_j$. Then
\[ \max_{\lambda_j \leq x} \lambda_j^{r/4} \left\lvert \mathscr P_{\mathbf H}(f_j) \right\rvert \gg (\log \log x)^{1/4 + o(1)} \]
as $x \to +\infty$.
\end{theorem}

The more precise result is Theorem~\ref{largevaluesslthree}. We make two immediate remarks, on which we expand further in the remainder of this introduction. The first concerns the factor $\lambda_j^{r/4}$ in the left-hand side. Because the $\mathbf H$-period is not zero-dimensional, the quadratic mean of the periods $\mathscr P_{\mathbf H}(f_j)$ over eigenvalue intervals of size $\sqrt{\lambda}$ around $\lambda > 0$, decays by a theorem of Zelditch \cite{zelditch1992}, with rate $\lambda^{-r/4}$. In other words, the periods have size $\asymp \lambda_j^{-r/4}$ ``on average'', and the theorem says that the periods are unbounded when normalized accordingly. In fact, proving a finer spectral parameter version of the result of Zelditch is an essential part of the proof of Theorem~\ref{largevaluesslthreesimplified}, and we state it in Theorem~\ref{meansquareasymptotic}. We add here that stating the bound in terms of the Laplacian eigenvalue is an oversimplification, and in fact, has the risk of making the growth rate meaningless because the average decay rate is likely to have a smaller exponent near singular spectral parameters; see also \S\ref{Lvsperioddifferences}. We did so for the sake of exposition.

The second remark concerns the double logarithmic growth rate in the right-hand side. At first sight, this may seem disappointing when compared with results of this type for (forms of) $\mathbf{PGL}_2$ \cite{Michels2020,milicevic2010}, or with results that give polynomial growth of periods (\S\ref{Lvsperioddifferences}). But in fact we have a strong reason to believe that in this case the double logarithmic growth is best possible, up to the exponent $1/4$. And moreover, that this is related to the maximal growth of $L$-functions at the edge of their critical strip, and should be compared to the conjecture \cite{Granville2006} that
\[ \max_{1 \leq |t| \leq x} |\zeta(1+it)| \sim e^\gamma \log \log x \,.  \]
We expand on this at the end of \S\ref{introsecsubpoly}.

\subsection{The spectral growth of periods}

Let $\mathbf G$ be a semisimple algebraic $\mathbb Q$-group and $\mathbf H$ a closed algebraic subgroup that we assume to be anisotropic, meaning that the adelic points $\mathbf H(\mathbb A_{\mathbb Q})$ have compact image $[\mathbf H]$ in the automorphic quotient $[\mathbf G] = \mathbf G(\mathbb Q) \backslash \mathbf G(\mathbb A_{\mathbb Q})$. For any automorphic form $f$ on $[\mathbf G]$ one may consider the period $\mathscr P_{\mathbf H}(f)$ by integrating along a suitable translate of $[\mathbf H]$. We consider families of forms by allowing certain parameters to vary, and we also speak of aspects. Our focus is on the spectral aspect family of Hecke-Maass forms: Fix a maximal compact subgroup of $\mathbf G(\mathbb R)$ and a level structure, and let $\mathcal F$ be the family of spherical cusp forms on the locally symmetric space resulting from these choices. The aim of this article is to advance our understanding of the following question.

\begin{center}
\itshape When $\lambda(f)$ denotes the Laplacian eigenvalue of $f \in \mathcal F$, how fast does
\begin{equation}\label{maxgrowthperiodeqn}
\max_{\substack{f \in \mathcal F \\ \lambda(f) \leq x}} \left\lvert \mathscr P_{\mathbf H}(f) \right\rvert
\end{equation}
grow, if at all, as $x \to \infty$?
\end{center}
To state the question correctly one needs normalizing factors in front of the periods (as in \eqref{geodesicresult} and Theorem~\ref{largevaluesslthree} below). Further, $\lambda(f)$ is not always the correct measure of complexity for the $\mathbf H$-period (see \S\ref{Lvsperioddifferences}). But we ignore these details for now.
We will be interested in sub-polynomial growth rates, whose significance is best understood in the context of extreme values of $L$-functions. In fact, question \eqref{maxgrowthperiodeqn} is in many ways similar to the following question, and this similarity will be our basis for revealing possible arithmetic information about periods.
\begin{center}
\itshape When $\mathcal F$ is a reasonable family of $L$-functions (not necessarily known to be automorphic) for $\mathbf{GL}_n$, how fast does
\begin{equation}\label{problemLfunctiongrowth}
\max_{\substack{L \in \mathcal F \\ C(L) \leq x}} |L(\tfrac12)| 
\end{equation}
grow as $x \to \infty$?
\end{center}
Here $C(L)$ denotes the analytic conductor \cite{iwaniec2000}, and we have scaled all $L$-functions to have central value at $\frac12$. A classical example is the family of unramified unitary Hecke characters, leading to vertical shifts of the Riemann zeta function: $L(\lvert \cdot \rvert_\infty^{it}, s) = \zeta(s + it)$ (the $t$-aspect for $\zeta$), in which case $C(L) \asymp t$.

Many problems in analytic number theory are related to understanding the size of the central value $L(\frac12)$ in families. For every reasonable family we can formulate the Lindelöf hypothesis that $L(\frac12) \ll_\epsilon C(L)^\epsilon$ and the subconvexity problem which asks for a $\delta > 0$ such that $L(\frac12) \ll C(L)^{1/4 - \delta}$. In particular, the growth of $|L(\tfrac12)|$ is sub-polynomial under the Lindelöf hypothesis. It is conjectured \cite{farmer2007} that the true growth of \eqref{problemLfunctiongrowth} is
\begin{equation}\label{Lfunctionconjecture}
\max_{\substack{L \in \mathcal F \\ C(L) \leq x}} |L(\tfrac12)| = \exp \left( (C_{\mathcal F}+o(1))\sqrt{\log x \cdot \log \log x}  \right)
\end{equation}
for some explicit constant $C_{\mathcal F}$ that depends on the family. Proofs conditional on GRH \cite{montgomery1977,pankowski2013} or using the resonance method of Soundararajan \cite{soundararajan2008} typically yield lower bounds of quality
\begin{equation}\label{resonancequality}
\max_{\substack{L \in \mathcal F \\ C(L) \leq x}} |L(\tfrac12)| \gg \exp \left( C\sqrt{\frac{\log x }{\log \log x}}  \right)
\end{equation}
for some $C > 0$. Recent progress as well as lower bounds with additional restrictions on the argument of $L(\frac12)$ can be found in \cite{blomer2020}, \cite{bondarenko2017}, \cite{breteche2019}.

\subsection{Spectral resonance}

The similarity between the problems about automorphic periods and (central) $L$-values is already visible in the resonance method in \cite{soundararajan2008}, where for each family of $L$-functions under consideration the spectral ingredient is a trace formula: The Fourier inversion formula in the case of $\zeta$, and the Petersson trace formula in the case of $L$-functions of modular forms. In the seminal article \cite{iwaniec1995}, which predates this, the authors use the Selberg trace formula to prove the following result about a compact arithmetic hyperbolic surface $X$: Let $z \in X$ be a CM point. When $(f_j)$ denotes the sequence of Hecke-Maass forms on $X$, ordered by increasing eigenvalue $\lambda_j$, then
\begin{equation}\label{iwaniecsarnakresult}
\max_{\lambda_j \leq x }\left\vert f_j(z) \right\rvert \gg \sqrt{\log \log x} \qquad  \qquad (x \to \infty) \,.
\end{equation}
Their method is known as the amplification method, which bears much similarity to the resonance method of Soundararajan. It may be summarized as follows: Construct a ``resonator'' $R(f) \geq 0$ with the property that a quotient of the form
\begin{equation}
\frac{\sum_{j}R(f_j) |f_j(z)|^2}{\sum_{j}R(f_j)}
\end{equation}
is large. If this quotient is bigger than a real number $M>0$, then at least one `period' $|f_j(z)|$ must be bigger than $M$. When $R(f)$ is defined as a Hecke eigenvalue of $f$, a sum of the shape of the numerator naturally appears in the relative trace formula, and a sum of the shape of the denominator naturally appears in the trace formula.

The similarity between the two problems is even more explicit in \cite{milicevic2010}, where the lower bound \eqref{iwaniecsarnakresult} is improved to
\begin{equation}\label{milicevicresult}
\max_{\lambda_j \leq x }\left\vert f_j(z) \right\rvert \gg \exp \left( C \sqrt{\frac{\log x}{\log \log x}} \right) \qquad  \qquad (x \to \infty) \,,
\end{equation}
for some explicit $C > 0$, by employing a resonator inspired by the resonators used for $L$-functions in \cite{soundararajan2008}. That this is similar to the growth rate \eqref{resonancequality} should not come as a surprise: The point evaluation $f_j(z)$ is (at least under a class number one hypothesis) precisely the period that appears in a formula of Waldspurger \cite{waldspurger1985} that relates the period to the central value of a Rankin-Selberg $L$-function.

In \cite{Michels2020} we proved an analogue of \eqref{milicevicresult} with CM points replaced by closed geodesics: for certain compact arithmetic hyperbolic surfaces $X$ and for a closed geodesic $\ell \subset X$,
\begin{equation}\label{geodesicresult}
\max_{\lambda_j \leq x }\lambda_j^{1/4} \left\vert \int_\ell f_j \right\rvert \gg \exp \left( C \sqrt{\frac{\log x}{\log \log x}} \right)\qquad  \qquad (x \to \infty) \,,
\end{equation}
for some $C > 0$. The factor $\lambda_j^{1/4}$ is a normalizing factor (which in this case can be thought of as a gamma factor), which is there to make the geodesic periods of size $1$ on average; see also Theorem~\ref{meansquareasymptotic} below.

\subsection{Periods beyond \texorpdfstring{$\mathbf{GL}_2$}{GL(2)}}

\label{Lvsperioddifferences}

This article originates in an attempt to generalize the result \eqref{geodesicresult} about geodesic periods to higher rank groups. Now, while question \eqref{maxgrowthperiodeqn} is in some aspects similar to question \eqref{problemLfunctiongrowth}, there are important differences. The first is in the problem statement itself: Where for $L$-functions the correct measure of complexity is the analytic conductor, the situation is much more delicate for periods. The Laplacian eigenvalue $\lambda(f)$ is a naive choice, and the correct replacement is most likely an integral involving an approximate spectral projector around $f$, which determines the size of mean square asymptotics for periods over $O(1)$ spectral windows (as in Theorem~\ref{meansquareasymptotic}). For sufficiently generic spectral parameters, that integral should be of size $\lambda(f)^{(n-r-d)/2}$, with $n$ the dimension of the symmetric space, $r$ its rank, and $d$ the dimension of the projection to the symmetric space of the $\mathbf H$-orbit underlying the period. But asymptotically evaluating this integral in terms of the spectral parameter of $f$, even for generic parameters, is a difficult problem.

The second difference is the following. Whereas in the case of $L$-functions one would expect the resonance method to always produce a nontrivial result of quality \eqref{resonancequality} (provided that we know enough about the family under consideration), this not the case for periods. In fact, the spectral resonance method is sometimes fruitless, and sometimes produces extreme values of periods with power growth. Moreover, a larger variety of techniques exist to prove lower bounds for periods. We give a brief historical account of these facts.

The first example of a different technique is given in \cite{rudnick1994}, which is about discrete periods on certain hyperbolic $3$-manifolds. The authors use what is known as a distinction method. It employs a vanishing property that says that only a sparse subsequence of Hecke-Maass forms, lying in the image of a theta lift, have nonzero $\mathbf H$-period. They are ``distinguished'' by $\mathbf H$. This is then contrasted with a mean square asymptotic, which gives the full average of the periods (and which does not see the arithmetic underlying the distinction). If the periods are supported on a sparse subsequence and the average does not know about this, it follows that that the periods must attain large values on the distinguished subsequence. In fact, the sequence is polynomially sparse, and the authors obtain periods with power growth. Note that while the arithmetic properties used are crucial, the Hecke operators do not play a direct role.

The article \cite{Milicevic2011} characterizes the hyperbolic $3$-manifolds to which the proof extends, as being those of Maclachlan-Reid type. Moreover, it gives the first example of the second phenomenon mentioned above: power growth obtained from the resonance method. Further power growth results that use the distinction method include the following settings: discrete periods on hyperbolic $n$-manifolds with $n \geq 5$ \cite{Donnelly2007}, and discrete unitary periods for $\mathbf{GL}_n$ (which includes the case of hyperbolic $2$ and $3$-manifolds) \cite{Lapid2007}.

A vast generalization of power growth results in the other direction, using the resonance method, is the article \cite{brumley2020}. In certain situations, one can obtain exceptional sequences of eigenfunctions with periods of power growth using both methods, and show that they are related. We refer to \cite{brumley2020, Milicevic2011} for a discussion of the relation between the two methods. We do note the following: All $\mathbb Q$-groups in these examples are such that $\mathbf G(\mathbb R)$ is not split.

\subsection{Sub-polynomial growth}

\label{introsecsubpoly}

This article is concerned with applications of the resonance method to periods of Hecke-Maass forms in the spectral aspect, in situations that have so far remained gray zones. By ``gray zone'', we mean the following. A first condition is that no explicit connection to central $L$-values is known or conjectured. Indeed, when conjecture \eqref{Lfunctionconjecture} does not apply, without heuristics that come from applications of random matrix theory to $L$-functions, it becomes an interesting question whether there are sufficiently strong arithmetic reasons for the existence of unusually large periods, and if so, what the maximal growth should be. A second condition is that no power growth is expected, or at least that established distinction techniques or resonance techniques do not produce power growth. In the gray zone it is not clear what period growth, if present at all, should be attributed to: heuristics motivated by random behavior of Laplacian eigenfunctions in negative curvature (Berry's random wave conjecture \cite{Berry1977}), random behavior of more exotic underlying artihmetic objects, or neither of these?

In fact, our hope when exploring the gray zone is to reveal arithmetic information: If in some new situation we obtain lower bounds of quality \eqref{resonancequality}, this would be a very strong hint that there is an as of yet unknown relation with central $L$-values. If we obtain periods with power growth, this might indicate the existence of a functorial lift (although it could be attributed to other factors as well; see \cite{brumley2020}). If we obtain growth rates smaller than \eqref{resonancequality} or nothing at all, this still leaves the possibility of a relation with $L$-values further to the right of $\frac12$, or at worst in the half-plane of convergence.

An example of such an unexplored situation is that of toric periods on locally symmetric spaces of non-compact type associated to a semisimple group $\mathbf G$. Unless $\mathbf G$ is isogenous to a product of forms of $\mathbf{PGL}_2$, these fall in the gray zone.  When the locally symmetric space $X$ is viewed as a disjoint union of classical locally symmetric spaces, then the period along a maximal $\mathbb R$-split torus corresponds to the integral along a flat submanifold of $X$, which has the property that the intersection with each component is either maximal flat or empty; see \S\ref{sectorusorbitscorrespondence}.

To state the main theorems we introduce some minimal notation for spectral parameters. When $G$ is a connected semisimple Lie group with finite center, make a choice of Iwasawa decomposition $G = NAK$. Let $\mathfrak a = \operatorname{Lie}(A)$ and define the generic set $(\mathfrak a^{*})^{\operatorname{gen}} \subset \mathfrak a^*$ as the set of elements that are regular and that do not lie in a proper subspace spanned by roots. A locally symmetric space $X$ is assumed to be compatible with the choice of $K$. When $\mathbf G$ is a semisimple group over $\mathbb Q$, we may define the above notation with respect to $\mathbf G(\mathbb R) ^0$, and we again assume an associated adelic locally symmetric space $X$ to be compatible with $K$. The spectral parameters of Hecke-Maass forms on $X$ can then be viewed as elements of $(\mathfrak a_{\mathbb C})^*$. For any additional notation and terminology used in the theorems below, we refer to \S\ref{preliminariesLiegroups}, \S\ref{prelimsymmetricspacesmainsection} and \S\ref{alggrouppreliminaries}.

We can now state the following theorem.

\begin{theorem}\label{largevaluesslthree} Let $\mathbf G$ be an anisotropic $\mathbb Q$-form of $\mathbf {PGL}_{3}$ with $\mathbf G(\mathbb R)$ noncompact. Let $X$ be an associated adelic locally symmetric space  and $(f_j) \in L^2(X)$ an orthonormal basis of Hecke-Maass forms with spectral parameters $\nu_j \in (\mathfrak a_{\mathbb C}) ^*$. Let $\mathbf H \subset \mathbf G$ be a maximal torus with the same $\mathbb R$-rank $r$ as $\mathbf G$, and denote by $\mathscr P_{\mathbf H}(f_j)$ the $\mathbf H$-period of $f_j$. Let $D_{\mathfrak a^*} \subset (\mathfrak a^*)^{\operatorname{gen}}$ be compact. There exist $ C > 0$ and $ \delta > 0$ such that uniformly for $\nu \in D_{\mathfrak a^*}$ and $t \in \mathbb R$ we have
\[ \max_{\lVert \nu_j - t\nu \rVert  \leq C } (1 + t)^{r} \left\lvert \mathscr P_{\mathbf H}(f_j) \right\rvert^2 \gg  (\log \log (2+t))^{\delta + o(1)}  \,. \]
Moreover, when $E$ is the splitting field of $\mathbf H$, we may take $\delta=6/[E : \mathbb Q]$.
\end{theorem}

In \S\ref{sectionformspgl} we give the list of groups to which Theorem~\ref{largevaluesslthree} can be applied. There we also show that the associated Lie group $\mathbf G(\mathbb R)$ is either $\operatorname{PGL}_3(\mathbb R)$ or the quasi-split projective unitary group $\operatorname{PU}(2, 1)$, and in these cases the $\mathbb R$-rank equals $2$ or $1$ respectively.

We now come back to the question raised earlier: What must the nontrivial growth in Theorem~\ref{largevaluesslthree} be attributed to? We believe the reason for this growth is arithmetic, for the following reason. There is an exceptional theta-correspondence between $\mathbf{PGL}_3$ and the exceptional group $\mathbf G_2$, which is related to the maximal toric periods in the theorem. In fact, the period $\mathscr P_{\mathbf H}(f_j)$ should be, up to other arithmetic factors that are equally mysterious, related to a product of $L$-values
\[ L(\pi_{f_j}, 1) L(\widetilde {\pi}_{f_j}, 1) \,, \]
where $\pi_{f_j}$ denotes the representation generated by $f_j$ and $\widetilde {\pi}_{f_j}$ denotes its contragredient. This period relation, while not proven, explains a great deal about the lower bound in Theorem~\ref{largevaluesslthree}. First, it leads us to believe that the lower bound can, perhaps, not simply be attributed to random behaviour of Laplacian eigenfunctions on Riemannian manifolds of nonpositive curvature. But in fact, much more can be said. The period formula explains why the growth rate we obtain should be polynomial in $\log \log C(L(\pi_{f_j}, s))$, the double log of the analytic conductor! Indeed, the double logarithm reminds us of the following result of Levinson \cite{Levinson1972}: There exists a constant $C > 0$ such that for arbitrarily large $t \in \mathbb R$, one has
\[ |\zeta(1 + it)| \geq e^\gamma \log \log t - C \,. \]
This is certainly not the best known result, but the main term in the right-hand side is conjectured to be optimal \cite{Granville2006}.
We refer to \cite{Aistleitner2019} for the state of the art on the extreme values of $\zeta(1 + it)$, and results with lower order terms. The results and conjectures for $\zeta(1 + it)$ are exemplary for the more general situation. In fact, there is the following conditional statement \cite{Blomer2020d}: When $\pi$ is a unitary cuspidal tempered automorphic representation for $\mathbf{GL}_n$ whose Godement-Jacquet $L$-function $L(\pi, s)$ satisfies the generalized Riemann hypothesis, then
\[ (\log \log C(\pi))^{-n} \ll |L(\pi, 1)| \ll (\log \log C(\pi))^n\,, \]
where the implicit constants depend on $n$ only and we denote $C(\pi)$ for the analytic conductor of $L(\pi, s)$. Moreover, if $\pi$ is not tempered, one still has similar bounds but with bigger exponents.

Coming back to the interpretation of Theorem~\ref{largevaluesslthree}, we conclude the following. First, ignoring for the sake of the argument the other factors in the period formula mentioned above, it is to be expected that the maximal toric periods exhibit oscillations that are polynomial in $\log \log \lambda_j$. Indeed, this is (at least conjecturally) the largest permitted oscillation of the $L$-value $L(\pi_{f_j}, 1) L(\widetilde {\pi}_{f_j}, 1)$, and based on what we know about $\zeta(1 + it)$ we would expect this maximal oscillation to be almost realized. Second, if forced to make a conjecture about the maximal growth of the periods in Theorem~\ref{largevaluesslthree}, it would be that the exponent of $\log \log (2+t)$ in the right-hand side can be replaced by $12 + o(1)$. (Taking into account the fact the periods appear with a square in the left-hand side.)

About the exponent in the right-hand side, we remark the following: The splitting field $E$ of $\mathbf H$ is Galois, and the Galois group embeds naturally as a subgroup of $\operatorname{GL}_2(\mathbb Z)$ \cite[\S 1.7]{borel1965}. It is well known that finite subgroups of $\operatorname{GL}_2(\mathbb Z)$ have cardinality at most $12$; this is most easily seen by observing that the eigenvalues of a matrix of finite order must be roots of unity. Thus $[E : \mathbb Q] \leq 12$, meaning that $\delta = \frac12$ is admissible. In the best case, the Galois group is of cardinality $3$ and we obtain the exponent $2$; still far from the (naive) conjecture that it can be replaced by $12$.

The main arithmetic ingredient that goes into the proof of Theorem~\ref{largevaluesslthree} is the optimization problem of finding the best possible resonator sequence. This is what explains the restriction to forms of $\mathbf{PGL}_3$. The optimization problem takes as input asymptotics for local $p$-adic integrals arising in the geometric main term in a relative trace formula, and requires us to construct a suitable Hecke operator. For (forms of) $\mathbf {PGL}_n$ with $n \geq 3$, only for $n = 3$ have we found a suitable winning construction that, when plugged into the optimization problem, yields nonconstant growth. We strongly believe that for $n \geq 4$ no such construction exists, and that the resonance method cannot produce any growth of toric periods for $n \geq 4$. If indeed true, we believe that, roughly speaking, this should be attributed to the heuristic that tori inside $\mathbf{PGL}_3$ are still relatively large, while they are too small inside $\mathbf{PGL}_n$, $n \geq 4$. We refer to \S\ref{amplifierupperbounds} and Remark~\ref{remarkglnboundoptimality} for these negative statements.

\subsection{A mean square asymptotic}

\label{meanssquareintrosection}

The main analytic ingredient in the proof of Theorem~\ref{largevaluesslthree} is a (amplified) mean square asymptotic for maximal flat periods. We require averages over spectral windows of bounded size, and even a non-amplified version requires a considerable amount of work. For this ingredient there is no reason to restrict to the Lie groups associated to groups in Theorem~\ref{largevaluesslthree}, and we prove a more general result, formulated classically in terms of periods along maximal flat submanifolds.

\begin{theorem}\label{meansquareasymptotic}Let $G$ be a noncompact connected semisimple Lie group with finite center and rank $r$, and let $X$ be an associated compact Riemannian locally symmetric space. Assume that at least one of the following holds:
\begin{itemize}
\item $G$ has rank $1$;
\item $G = \operatorname{SL}_p(\mathbb R)$ or $\operatorname{SU}(k, p-k)$ with $p$ prime and $0< k < p$, and $X$ arises from a $\mathbb Q$-form of $\mathbf{SL}_p$.
\end{itemize}
Let $(f_j) \in L^2(X)$ be an orthonormal basis of Maass forms. Let $\mathscr F \subset X$ be a compact maximal flat submanifold and denote by $\mathscr P_{\mathscr F}(f_j)$ the period of $f_j$ along $\mathscr F$.
Let $D_{\mathfrak a^*} \subset (\mathfrak a^*)^{\operatorname{gen}}$ be compact. There exists $ C > 0$ such that uniformly for $\nu \in D_{\mathfrak a^*}$ and $t \in \mathbb R$ we have
\[ \sum_{\lVert \nu_j - t\nu \rVert  \leq C }\left \rvert\mathscr P_{\mathscr F} (f_j) \right\rvert^2 \asymp \beta(t\nu) \cdot (1 + t)^{-r} \,, \]
where $\beta$ denotes the Plancherel density.
\end{theorem}

The restriction to certain locally symmetric spaces in Theorem~\ref{meansquareasymptotic} comes from our bounds for orbital integrals in the relative trace formula. These are defined in terms of a parameter $g \in G$, and behave differently when $g$ centralizes a proper portion of the group underlying the period. That is, when $g$ lies in the standard Levi subgroup of a semistandard parabolic other than $G$ or the minimal one (see \eqref{definitionlevis}). We have not been able to deal with the integrals for such $g$, which is why we restrict to settings where they do not appear in the relative trace formula.

Theorem~\ref{meansquareasymptotic} tells us that we may view the factor $(1+t)^{r}$ in the left-hand side in Theorem~\ref{largevaluesslthree} as a normalizing factor that makes the squared periods of size $1$ on average. 
It is consistent with the mean square asymptotic of Zelditch \cite{zelditch1992} over eigenvalue intervals, but there is no obvious implication between the two results. The relation to Zelditch's result is analogous to the relation of the spectral parameter Weyl law \cite{duistermaat1979} to the classical Weyl law for compact Riemannian manifolds.

The stationary phase analysis in the proof of Theorem~\ref{meansquareasymptotic} draws inspiration from the proof of bounds for spherical functions in \cite{duistermaat1983} and generalizes work of Marshall \cite{Marshall2016} for $\operatorname{PGL}_2$. Where for $\operatorname{PGL}_2$ one relies on classical facts about the geometry of geodesics in the Poincaré upper half-plane model, the analogues of those facts were not available in the generality needed here and are established in the companion article \cite{Michels2022cartan}.

The quantification over $\nu$ and $t$ in Theorem~\ref{largevaluesslthree} and Theorem~\ref{meansquareasymptotic} may be visualized as follows. Fix a closed cone in the interior of $(\mathfrak a^*)^{\operatorname{gen}}$. Then $t \nu$ tends to $\infty$ inside the cone, and the maximum is taken over a ball of bounded radius around $t \nu$. Figure~\ref{conegenericset} gives a picture when $\mathfrak g = \mathfrak{sl}_3(\mathbb R)$, and the simple roots are denoted by $\alpha$ and $\beta$. When $\mathfrak g = \mathfrak{sl}_3(\mathbb R)$, the condition on tempered spectral parameters to not lie in a proper subspace spanned by roots, is equivalent to not being self-dual, and it also appears in \cite{Blomer2015}.

\begin{figure}[h!]
\begin{center}
\captionsetup{justification=centering}
\includegraphics[scale=1,angle=0]{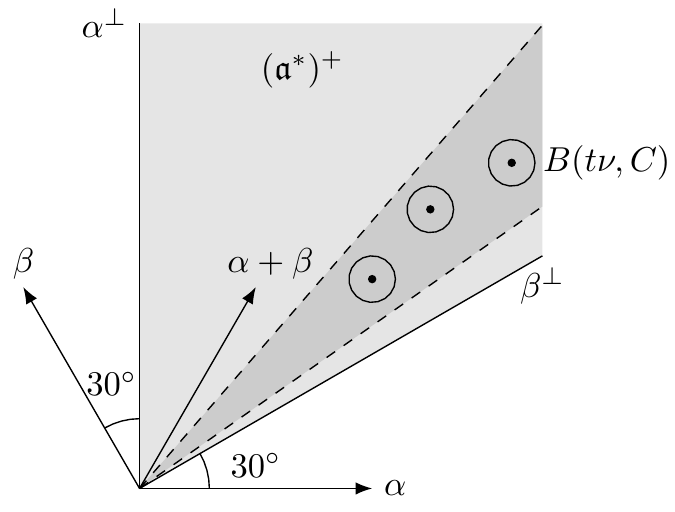}
\caption{A cone in $(\mathfrak a^*)^{\operatorname{gen}}$ when $G = \operatorname{SL}_3(\mathbb R)$, with balls $B(t\nu, C)$ as in Theorem~\ref{meansquareasymptotic}.}
\label{conegenericset}
\end{center}
\end{figure}

\section{Preliminaries}

\subsection{Lie groups and Lie algebras}

\label{preliminariesLiegroups}
\label{notationliegroups}

Let $G$ be a (possibly disconnected) semisimple Lie group which is reductive in the sense of Harish-Chandra \cite{harishchandra1975}; see also \cite[Chapter VII]{knapp2002}. That is, $G$ has finite center and acts by inner automorphisms on its Lie algebra $\mathfrak g$. The group $G$ will almost always be connected semisimple with finite center and we will then simply say $G$ is semisimple; only in \S\ref{alggrouppreliminaries} and \S\ref{extremevaluesthreeproof} it will be possibly disconnected. Let $K \subset G$ be a maximal compact subgroup and $\theta$ a Cartan involution of $G$ whose fixed point set is $K$. It induces an involution $\theta$ of $\mathfrak g$, whose $+1$- and $-1$-eigenspaces we denote by $\mathfrak k$ and $\mathfrak p$. We denote the exponential of $X \in \mathfrak g$ by $\exp(X)$. Denote by $\langle \cdot, \cdot \rangle$ the Killing form on $\mathfrak{g}$; it is positive definite on $\mathfrak{k}$ and negative definite on $\mathfrak{p}$. Define $\langle \cdot, \cdot \rangle_{\theta} = \langle \cdot, - \theta(\cdot) \rangle$, a positive definite symmetric bilinear form. All statements on $\mathfrak{g}$ involving norms, orthogonality and adjoints will be with respect to $\langle \cdot, \cdot \rangle_{\theta}$. Let $\mathfrak a \subset \mathfrak p$ be a maximal abelian subalgebra and $A = \exp(\mathfrak a)$. The choices of $\mathfrak a$ are all conjugate under $K$. Define $P = \exp(\mathfrak p)$. Multiplication $P  \times K \to G$ is a diffeomorphism, known as the Cartan decomposition. In particular, $K$ meets all components of $G$.

\subsection{Symmetric spaces}
\label{prelimsymmetricspacesmainsection}

\label{defdistancegeneralgroup}

References for the following facts about symmetric spaces are \cite{Eberlein1996, Helgason1978}. 

Assume here that $G$ is (connected) semisimple. The quotient $S = G/K$ carries a left-$G$-invariant Riemannian metric induced by the Killing form on $\mathfrak p$. It is a symmetric space of non-compact type, and every such space arises in this way.

The maximal flat submanifolds of $S$ are of the form $gAK$ with $g\in G$. Such $g$ is uniquely determined by the submanifold up to multiplication on the right by $N_G(A)$. When $\dim(A) = 1$, the maximal flats are precisely the geodesics. The rank of $G$ is defined to be $\dim(A)$.

We equip $G$ with any Riemannian metric and the associated distance function $d(\cdot, \cdot)$. We will care only about distances on compact sets and up to constant factors, so we do not need to impose any invariance properties of $d(\cdot, \cdot)$.

\subsection{Iwasawa decomposition}

\label{secroots}

Let $\Sigma$ be the set of restricted roots of $\mathfrak{a}$ in $\mathfrak{g}$. By convention, $0  \notin \Sigma$. We denote by $\mathfrak{g}_{\alpha}$ the root space of a root $\alpha \in \Sigma$ and by $H_\alpha \in \mathfrak a$ the element corresponding to $\alpha$ under the isomorphism $\mathfrak a \cong \mathfrak a^*$ given by $\langle \cdot, \cdot \rangle$. Fix a set of positive roots $\Sigma^{+}$ with basis $\Pi$. Let $\mathfrak n \oplus \mathfrak a \oplus \mathfrak k$ and $N \times A \times K$ be the corresponding Iwasawa decompositions of $\mathfrak g$ and $G$. Define $M = Z_{K}(A)$ and $M' = N_{K}(A)$ and denote by $\mathfrak{m}$ the Lie algebra of $M$.

Denote the Lie algebra Iwasawa projections by $E_{\mathfrak{n}}$, $E_{\mathfrak{a}}$ and $E_{\mathfrak{k}}$. We have the orthogonal restricted root space decomposition
\begin{equation}\label{rootspacedecomp}
\mathfrak{g} = \mathfrak{a} \oplus \mathfrak{m} \oplus \bigoplus_{\alpha \in \Sigma} \mathfrak{g}_{\alpha}  \,.
\end{equation}

Denote the Iwasawa projections from $G$ onto $N$, $A$ and $K$ by $n$, $a$ and $\kappa$. Define the height $H(g) = \log(a(g)) \in \mathfrak{a}$, the logarithm being the Lie logarithm on $A$.

\subsection{Centralizers}
\label{defgenericset}

\label{definitionlevis}
Denote by $\mathcal{L}$ the set of centralizers in $G$ of subgroups of $A$. They are the standard Levi subgroups of semistandard parabolic subgroups of $G$. We will denote such a centralizer typically by $L$. When $L \in \mathcal{L}$ with Lie algebra $\mathfrak{l}$, define $\mathfrak{a}_{L} = \mathfrak{z}(\mathfrak{l}) \cap \mathfrak{a}$ and $\mathfrak{a}^{L}$ its orthogonal complement in $\mathfrak a$. 
The set $\mathcal L$ contains $A$ and $G$, and $\mathfrak a^A = \mathfrak a_G = 0$ and $\mathfrak a_A = \mathfrak a^G = \mathfrak a$.

Define the positive Weyl chamber $\mathfrak{a}^{+} = \{ H \in \mathfrak{a} : \forall \alpha \in \Sigma^{+} : \alpha(H) > 0 \}$ and the regular set
\[ \mathfrak{a}^{\operatorname{reg}} = \mathfrak{a} - \bigcup_{L \neq M} \mathfrak{a}_{L} = \mathfrak{a} - \bigcup_{\alpha \in \Sigma} \ker(\alpha) \,. \]
We have that $H \in \mathfrak a^{\operatorname{reg}}$ if and only if its centralizer equals the centralizer of $\mathfrak a$.
Define also the generic set
\[ \mathfrak{a}^{\operatorname{gen}} = \mathfrak{a}^{\operatorname{reg}} - \bigcup_{L \in \mathcal L - \{ G\}} \mathfrak{a}^{L} \,. \]
Combined superscripts correspond to intersections: $\mathfrak{a}^{\operatorname{gen}, +} = \mathfrak{a}^{\operatorname{gen}} \cap \mathfrak{a}^{+}$.

We also define $(\mathfrak a^*)^{\operatorname{reg}}$, $(\mathfrak a^*)^{\operatorname{gen}}$ and $(\mathfrak a^*)^{+}$ to be the corresponding subsets under the isomorphism  $\mathfrak a \cong \mathfrak a^*$ defined by the inner product $\langle \cdot, \cdot \rangle$. When $H \in \mathfrak a$ corresponds to $\lambda \in \mathfrak a^*$ under the isomorphism given by the Killing form, then $H \in \mathfrak a^{\operatorname{reg}}$ if and only if $\lambda$ is not orthogonal to any roots, and $H  \in \mathfrak a^{\operatorname{gen}}$ if and only if $\lambda$ is in addition not contained in a proper subspace spanned by roots.

We will frequently use the facts that $Z_G(A) = MA$ and $N_G(A) = M' A$, as well as the following lemma.

\begin{lemma} \label{truelemma}Let $g \in G$ and $H \in \mathfrak{a}$. If $\operatorname{Ad}_{g}(H) \in \mathfrak{m} \oplus \mathfrak{a}$, then $g \in M' Z_{G}(H)$.\end{lemma}

\begin{proof}See for example \cite[Lemma~2.2]{Michels2022cartan}.
\end{proof}

\subsection{Derivatives}

\label{sectionderivatives}

When $G$ is any Lie group with Lie algebra $\mathfrak g$ and $b$ is an element of the universal enveloping algebra $U(\mathfrak{g})$, we denote by $L_{b}$ the corresponding left invariant differential operator on $C^{\infty}(G)$
. When $X \in \mathfrak g \subset U(\mathfrak g)$, by definition
\[ (L_Xf)(g) = \left.\frac d{dt}\right\rvert_{t=0} f( ge^{tX}) \,.\]
When $f : M \to N$ is a differentiable map between differentiable manifolds, denote its differential at $m \in M$ by $(Df)_m$.
Using left translation we identify all tangent spaces $T_{g} G$ with $\mathfrak{g}$. When $g \in G$, denote by $L_{g}$ and $R_{g}$ the left and right multiplication by $g$ on $G$. With our convention on tangent spaces, we then have for all $g, h \in G$ that
\begin{align}
(D L_{g})_{h} & = \operatorname{id} \,,  \nonumber \\
(D R_{g})_{h} & = \operatorname{Ad}_{g^{-1}} \,. 
\end{align}
When $X, Y \in \mathfrak{g}$ we have
\begin{align}
L_{X} \operatorname{Ad}_{g}(Y) & = \operatorname{Ad}_{g}([X, Y]) \,, \label{derivativeAd} \\
L_{X} \operatorname{Ad}_{g^{-1}}(Y) & = - [X, \operatorname{Ad}_{g^{-1}}(Y)] \label{derivativeAdinverse} \,.
\end{align}

\begin{lemma} \label{computationAfirstderivative}Let $G$ be semisimple as in \S \ref{notationliegroups}. The differential of the Iwasawa projection $H : G \to \mathfrak a$ at $g \in G$ is as follows:
\begin{align*}
(D H)_{g} & = E_{\mathfrak{a}} \circ \operatorname{Ad}_{\kappa(g)} \,.
\end{align*}
\end{lemma}

\begin{proof}
See for example \cite[Lemma~2.3]{Michels2022cartan}.
\end{proof}

\subsection{Measures and convolution}

\label{measuresliegroupcase}

We fix any Haar measure $dg$ on $G$ and let $dk$ be the Haar measure on $K$ for which $\operatorname{Vol}(K) = 1$. We fix the Haar measure on $A$ that coincides with the measure induced from the left-invariant Riemannian metric on the symmetric space $G/K$ on the submanifold $A \subset G/K$. If $H$ is a subgroup between $A$ and $MA$, this then defines a unique Haar measure on $H$.

The universal enveloping algebra $U(\mathfrak g)$ acts on $C^\infty(G)$ as in \S\ref{sectionderivatives}. The convolution algebra $C_c^\infty(K \backslash G /K)$ acts on $C^\infty(G / K)$ by right translation, and it is commutative \cite[Theorem IV.3.1]{Helgason1984}.

\subsection{Maass forms}

\label{classicallocallysymmetricdefs}

\label{classicalmaasformdefs}

Let $G$ be semisimple and $\Gamma \subset G$ a co-compact and torsion-free lattice. The quotient $X = \Gamma \backslash G/K$ is a compact locally symmetric space, and $C^\infty(\Gamma \backslash G)$ carries an action of $U(\mathfrak g)$. The algebra $Z(U(\mathfrak g))$ is commutative, and we may find an orthonormal basis $(f_j)_{j \geq 0}$ of $L^2 (X)$ consisting of simultaneous eigenfunctions for $Z(U(\mathfrak g))$, the Maass forms.

For $\nu \in (\mathfrak a^*)_{\mathbb C}$, the complexified dual of $\mathfrak a$, let $ \varphi_{\nu} : G \to \mathbb C$ be the spherical function of parameter $\nu$ \cite[Chapter IV]{Helgason1984}. It is given explicitly by Harish-Chandra's integral representation as
\begin{equation} \label{sphericalfunctionformula}
\varphi_{\nu}(g) = \int_{K} \exp \left( (\rho + i\nu)(H(kg)) \right) dk \,,
\end{equation}
where $H : G \to \mathfrak a$ is the Iwasawa projection and $\rho \in \mathfrak a^*$ the half-sum of positive roots (\S\ref{notationliegroups}) \cite[Theorem IV.4.3]{Helgason1984}. We say $f_j$ has spectral parameter $\nu_j \in (\mathfrak a^*)_{\mathbb C}$ if it has the same $Z(U(\mathfrak g))$-eigenvalues as $\varphi_{\nu_j}$; such $\nu_j$ is uniquely determined up to the action of the Weyl group.

The $f_j$ are eigenfunctions for the algebra $C_c^\infty(K \backslash G /K)$. For any compactly supported smooth kernel $k \in C_c^\infty(K \backslash G /K)$, define the Harish-Chandra transform
\begin{align}\label{harishchandratransformgeneral}
\begin{split}
\widehat k : (\mathfrak a^*)_{\mathbb C} & \to \mathbb C \\
\nu & \mapsto \int_G k(g) \varphi_{-\nu}(g) dg \,.
\end{split}
\end{align}
Then $k * f_j = \widehat k (\nu_j) f_j$ (the convolution being on $G$). We recover $k$ from $\widehat k$ by
\begin{equation}\label{harishchandrainversion}
k(g) = \frac 1{|W|}\int_{\mathfrak a^*} \varphi_\nu(g) \widehat k( \nu) \beta(\nu) d \nu
\end{equation}
where $\beta(\nu)$ is the Plancherel density \cite[\S IV.7]{Helgason1984}, provided that the Haar measure on $G$ in \eqref{harishchandratransformgeneral} is appropriately normalized (see \cite[Exercise IV.C.4]{Helgason1984}).

\subsection{Periods}

\label{propsmaximalflatssection}

When $\mathscr F \subset X$ is a compact maximal flat submanifold, we may lift it to a maximal flat submanifold of the symmetric space $G/K$. The lift is the image of a subset of $G$ of the form $gA$. Such $g$ is uniquely determined by the choice of the lift, up to multiplication on the right by $N_G(A)$. Because $\mathscr F$ is compact, $\Gamma$ intersects $gAg^{-1}$ in a lattice. We define the period of $f_j$ along $\mathscr F$ to be the integral
\begin{equation}
\label{classicalperiodgroupwise}
\mathscr P_{\mathscr F}(f_j) = \int_{(\Gamma \cap gAg^{-1}) \backslash gA} f_j \,.
\end{equation}

\begin{remark}The definition \eqref{classicalperiodgroupwise} does not depend on the choice of $g$ or even on the choice of Iwasawa $A$-component $A$. But do note that the inclusion $gA \subset G$ induces a closed embedding of $N_\Gamma(gAg^{-1}) \backslash gA$ in $\Gamma \backslash G  / K$ with image $\mathscr F$, and not (always) of the quotient of $gA$ by $\Gamma \cap gAg^{-1}$. That is, one may also consider the integral $\int_{\mathscr F} f_j$ with respect to the induced metric on the submanifold  $\mathscr F \subset G/K$. By our choice of Haar measure on $A$, the latter integral equals $\mathscr P_{\mathscr F}(f_j)$ divided by the finite number $|N_\Gamma(gAg^{-1}) /(\Gamma \cap gAg^{-1}) |$. \end{remark}

\section{Mean square asymptotics for maximal flat periods}

\label{sectionmeansquareasymptotics}

This section contains the skeleton of the proof of Theorem~\ref{meansquareasymptotic}. We set up a relative pre-trace formula in \S\ref{asymptoticproofsetup} and prove the theorem in \S\ref{asymptotictheoremproof}, using the analytic results from the later sections \S\ref{sectionmaintermintegral} and \S\ref{orbitalintegralsection}. For most of this section, $G$ can be any semisimple Lie group and $X$ any associated compact locally symmetric space. Only in \S\ref{asymptotictheoremproof} will we use that $G$ and $X$ as in Theorem~\ref{meansquareasymptotic}.

\subsection{Classical setup}
\label{asymptoticproofsetup}

Let $k \in C_c^\infty(K \backslash G /K)$ be a compactly supported smooth kernel. The automorphic kernel
\[ k^{\operatorname{aut}}(x, y) = \sum_{\gamma \in \Gamma} k(x^{-1}\gamma y) \]
acts on $L^2(\Gamma \backslash G)$, and its image lies in the space of right $K$-invariant functions. If this operator is positive on $K$-invariants (which will always be the case for us), then it has spectral expansion
\begin{equation}
\label{pretraceclassical}
\sum_{\gamma \in \Gamma} k(x^{-1}\gamma y) = \sum_{j} \widehat k(\nu_j) f_j(x) \overline{f_j(y)} \,,
\end{equation}
and both sums are uniformly convergent on $X \times X$. When $\Gamma$ is torsion free and orientation-preserving on the Riemannian manifold $G/K$, this is a theorem of Mercer \cite[p.\ 445]{mercer1909} (see also \cite[Satz VI.4.2.]{werner2018}). For the general case, see for example \cite{Tamagawa1960}.

As in \S\ref{propsmaximalflatssection}, let $g \in G$ be such that $gA$ projects to $\mathscr F$ in the quotient $\Gamma \backslash G / K$. Define $H = gAg^{-1}$ and $\Gamma_{H} = \Gamma \cap H$.

\begin{lemma}[Partitions of unity] \label{classicalpartitionuunity}There exists a nonnegative $b \in C_c^\infty(A)$ satisfying $\sum_{\gamma \in \Gamma_{H}} b(g^{-1}\gamma g a) = 1$ for all $a \in A$.\end{lemma}

\begin{proof}We may find $b_0 \in C_c^\infty(\mathbb R)$ with the property that $\sum_{n \in \mathbb Z} b_0(x+n) = 1$ for all $x \in \mathbb R$. Because $\log(\Gamma_{H})$ is a lattice of full rank in $\log(H) = \operatorname{Lie}(H)$, out of $b_0$ we may construct $b \in C_c^\infty(A)$ as in the statement.
\end{proof}

The group $N_G(H)/H$ is compact, which implies that the discrete subgroup $N_\Gamma(H) / \Gamma_{H} \subset N_G(H) / H$ is finite.

\begin{lemma}\label{classicalmainplusoffdiag}There exists a nonngative $b \in C_c^\infty(A)$ with the property that for all $k \in C_c^\infty(K \backslash G / K)$ we have 
\begin{align*}
\sum_{j} \widehat k(\nu_j) \lvert \mathscr P_{\mathscr F} (f_j) \rvert^2&= \operatorname{Vol}(\Gamma_H \backslash H) \cdot \left\lvert N_\Gamma( H) / \Gamma_{H} \right\rvert \int_A k(a) da \\
& + \sum_{\gamma \in \Gamma - N_\Gamma( H)} \int_{A \times A} b(a_1)b(a_2) k(a_1^{-1} g^{-1}\gamma g a_2) da_1da_2 \,.
\end{align*}
\end{lemma}

\begin{proof}This follows from fairly standard manipulations. Take $b$ to be a cutoff function given by Lemma~\ref{classicalpartitionuunity}. Integrating \eqref{pretraceclassical} over $(\Gamma_H \backslash H)^2$ and using \eqref{classicalperiodgroupwise} gives
\begin{align*}
\sum_{j} \widehat k(\nu_j) \lvert \mathscr P_{\mathscr F} (f_j) \rvert^2 &= \int_{(\Gamma_H \backslash H)^2}\sum_{\gamma \in \Gamma} k(x^{-1} \gamma y) dx dy\\
&= \int_{((g^{-1}\Gamma_{H}g) \backslash A)^2}  \sum_{\gamma \in \Gamma} k(a_1^{-1} g^{-1} \gamma ga_2) da_1 da_2\,.
\end{align*}
We split the $\gamma$-sum into sums over the disjoint subsets $N_\Gamma(H)$ and $\Gamma - N_\Gamma(H)$. Because these sets are stable on both sides under $\Gamma_{H}$ we may distribute the integral over the two terms. Concretely, the integral is the sum of the ``diagonal term'' 
\begin{equation}\label{classicalmainterm}
\int_{((g^{-1}\Gamma_{H}g) \backslash A)^2} \sum_{\gamma \in N_\Gamma(H)} k(a_1^{-1} g^{-1} \gamma ga_2) da_1 da_2 
\end{equation}
and an ``off-diagonal term'' where the sum goes over $\Gamma - N_\Gamma(H)$. For \eqref{classicalmainterm}, unfolding in $a_2$ gives
\begin{align*}
&= \int_{((g^{-1}\Gamma_{H}g) \backslash A) \times A} \sum_{\gamma \in N_\Gamma(H) / \Gamma_{H}} k(a_1^{-1} g^{-1} \gamma ga_2) da_1 da_2 \,.
\end{align*}
The $\gamma$-sum is now finite. We may change the order of summation and make for fixed $\gamma$ and $a_1$ the change of variables in the $a_2$-integral that makes the argument of $k$ equal to $a_2$, keeping in mind that $g^{-1}N_G(H)g  = N_G(A) \subset N_K(A) A$ and that $k$ is bi-$K$-invariant. This then gives the main term in the statement.

For the off-diagonal term, inserting the partitions of unity from Lemma~\ref{classicalpartitionuunity}, making a change of variables in the $\gamma$-sum and unfolding gives that the integral equals
\begin{align*}
&= \int_{(g^{-1} \Gamma_{H}g \backslash A)^2}\sum_{\gamma_1, \gamma_2 \in \Gamma_{H}} b(g^{-1}\gamma_1 g a_1)b(g^{-1}\gamma_2 g a_2) \\
& \qquad \cdot \sum_{\gamma \in \Gamma - N_\Gamma(H)} k(a_1^{-1} g^{-1}\gamma_1^{-1}\gamma \gamma_2 g a_2) da_1 d a_2 \\
&= \int_{A \times A} b(a_1) b(a_2) \sum_{\gamma \in \Gamma - N_\Gamma(H)} k(a_1^{-1} g^{-1}\gamma g a_2) da_1 d a_2 \,.
\end{align*}
Note that this $A \times A \times \Gamma$-integral is indeed absolutely convergent, even compactly supported thanks to the compact support of $b$ and $k$. Applying Fubini gives us the sum in the statement.
\end{proof}

\subsection{Test functions}

To prove Theorem~\ref{meansquareasymptotic} we make a choice of $k \in C_c^\infty(K \backslash G / K)$ to filter out $O(1)$ sums on the spectral side of the pre-trace formula \eqref{pretraceclassical}. For all spectral parameters $\nu_j$ we either have $\nu \in \mathfrak a^*$ or $ \operatorname{Re}(\nu)$ is singular and $\left\lVert \operatorname{Im}(\nu) \right\rVert \leq \left\lVert \rho \right\rVert$ (see \cite[\S IV.8, Theorem 8.2]{Helgason1984} for the bound on $\left\lVert \operatorname{Im}(\nu) \right\rVert$ and \cite[Theorem 16.6]{Knapp1989} for the statement about singularity).

\begin{lemma} \label{archimedeantestfunction} Let $R > 0$. We may find for every $\nu \in \mathfrak a^*$ a $k_\nu \in C_c^\infty(K \backslash G / K)$ with the following properties:
\begin{enumerate}
\item $\widehat k_\nu(\mu) \geq 0$ for $\mu \in \mathfrak a^*$;
\item $\widehat k_\nu(\mu) \geq 1$ for $\mu  \in \mathfrak a^*$ with $\lVert \mu - \nu \rVert \leq R$;
\item $\widehat k_\nu(\mu)\ll_N (1 + \left\lVert \nu - \mu \right\rVert)^{-N}$ uniformly in $\nu$ and $\mu  \in ( \mathfrak a_{\mathbb C})^*$ with $\left\lVert  \operatorname{Im}(\mu) \right\rVert\leq R$;
\item $k_\nu$ has support bounded independently of $\nu$;
\item If $G$ is simple modulo its center, $k_\nu(g) \ll \beta(\nu) (1 + \lVert \nu \rVert d(g, K))^{-1/2}$ uniformly in $\nu$ and $g$.
\end{enumerate}
\end{lemma}

\begin{proof}
Such $k_\nu$ can be constructed using the Paley-Wiener theorem of Gangolli \cite[Theorem 3.5]{gangolli1971}; see for example \cite[\S2.2]{Marshall2016lp} for a construction, and \cite[\S4.1]{brumley2020} for a proof of the last condition.
\end{proof}

We will apply Lemma~\ref{archimedeantestfunction} with any fixed $R > 
\left\lVert \rho \right\rVert$ in order to get decay on the spectrum of $L^2(X)$.

\begin{remark}The last condition in Lemma~\ref{archimedeantestfunction} is only used in the proof of Theorem~\ref{largevaluesslthree}. Specifically, in Proposition~\ref{brumarweyllaw}.\end{remark}

\subsection{Diagonal and off-diagonal estimates}

\label{asymptotictheoremproof}

Assume now that $G$ and the locally symmetric space are as in Theorem~\ref{meansquareasymptotic}. The theorem will be proven by estimating the different terms that appear in the right-hand side in Lemma~\ref{classicalmainplusoffdiag}. Using the inversion formula \eqref{harishchandrainversion} it will suffice to bound the analogous terms with the test function $k$ replaced by the spherical function. The results in question are proven in \S\ref{sectionmaintermintegral} and \S\ref{orbitalintegralsection}.

\begin{lemma}\label{groupshavenolevi}If $\gamma \in \Gamma$ is such that $g_\infty^{-1}\gamma g_\infty \notin M' A$, then $g_\infty^{-1}\gamma g_\infty \notin \bigcup_{L \in \mathcal L - \{G\}} M' L$.\end{lemma}

\begin{proof}When $G$ has rank $1$ there is nothing to prove because $\mathcal L = \{MA, G\}$. Assume now that $G = \operatorname{SL}_p(\mathbb R)$ or $\operatorname{SU}(k,p-k)$ with $p$ prime, and that $\Gamma$ comes from a (necessarily anisotropic) $\mathbb Q$-form $\mathbf G$ of the algebraic group $\mathbf{SL}_p$. By the discussion in \S\ref{sectorusorbitscorrespondence}, the maximal flat submanifold comes from a maximal torus $\mathbf H \subset \mathbf G$. If $g_\infty^{-1}\gamma g_\infty \notin M' A$ but it lies in some $M' L$, then the subvariety of $\mathbf H$ that is sent to $\mathbf H$ by the conjugacy action of $\gamma$ is a proper subtorus defined over $\mathbb Q$. This is not possible because $p$ is prime, in fact, the algebraic group $\mathbf G$ has no nontrivial proper closed connected subgroups other than its maximal tori \cite[Proposition 4.1, Corollary 4.2]{Garibaldi2009}.
\end{proof}

\begin{remark}In general, even when a locally symmetric space $X = \Gamma \backslash G / K$ is compact, it can happen that $g^{-1}\Gamma g$ intersects the groups $L \in \mathcal L -  \{G, MA\}$ (see \S\ref{definitionlevis}) nontrivially. When for example $G = \operatorname{PGL}_4(\mathbb R)$ and $\Gamma$ is an arithmetic lattice coming from a degree $4$ division algebra $D$ over $\mathbb Q$, then $\mathscr F$ corresponds to a quartic totally real field $F \subset D$. Because $F$ is quartic, it has a quadratic subfield $E$. It gives rise to a $1$-dimensional subgroup $H' \subset H$ and the centralizer $Z_D(E)$ gives rise to a hyperbolic plane $\mathcal H \subset G/K$ and an infinite discrete subgroup of $\Gamma \cap Z_G(H')$ that acts co-compactly on $\mathcal H$. \end{remark}

\begin{lemma}\label{asymptoticwithweightfunction}Let $D_{\mathfrak a^*} \subset (\mathfrak a^*)^{\operatorname{gen}}$ be compact and let $k_\nu$ for $\nu \in \mathfrak a^*$ be as in Lemma~\ref{archimedeantestfunction}. Uniformly for $\nu \in D_{\mathfrak a*}$ and $t \geq 1$ we have
\[ \sum_{j} \widehat k_{t\nu}(\nu_j) \left\lvert \mathscr P_{\mathscr F} (f_j) \right\rvert^2 \asymp \beta(t\nu) \cdot (1 + t)^{-r} \,. \]
\end{lemma}

\begin{proof}
With the analytic results from \S\ref{sectionmaintermintegral} and \S\ref{orbitalintegralsection}, this follows from standard properties of the Plancherel density and our assumptions on $k_{\nu}$. We apply Lemma~\ref{classicalmainplusoffdiag} to $k = k_{t \nu}$. Let $b_0 \in C_c^\infty(K \backslash G/K)$ with the property that $b_0(g) = 1$ when $g \in \bigcup_{\nu \in \mathfrak a^*} \operatorname{supp}(k_\nu)$; this exists by Lemma~\ref{archimedeantestfunction}. Using \eqref{harishchandrainversion} and Fubini the first term in the right-hand side of Lemma~\ref{classicalmainplusoffdiag} can be written as
\[ \operatorname{Vol}(\mathscr F) \cdot \left\lvert N_\Gamma( H) / \Gamma_{H} \right\rvert \int_{\mathfrak a^* } \widehat k_{t\nu}(\mu) \beta(\mu) \int_A b_0(a) \varphi_{\mu}(a) da d\mu \,. \]
The contribution of $\mu \in \mathfrak a^*$ with $\lVert \mu - t\nu \rVert \geq \sqrt t$ (say) can be bounded trivially using the rapid decay of $\widehat k_{t \nu}$ (Lemma~\ref{archimedeantestfunction}), the polynomial growth of $\beta(\mu)$ \cite[\S IV.7, Proposition 7.2]{Helgason1984} and the bound $\varphi_{\mu}(g) \ll 1$ which follows for example from the fact that $\varphi_\mu$ is positive definite \cite[Exercise IV.B.9]{Helgason1984}. For the remaining $\mu \in \mathfrak a^*$ we have $\mu \in (\mathfrak a^*)^{\operatorname{gen}}$ when $t$ is sufficiently large, and we may apply the asymptotic from Proposition~\ref{mainintegralsphericalfunction} to the inner integral. Note that for such $\mu$ we have $\beta(\mu) \asymp \beta(t \nu)$ by the almost-polynomial behavior of $\beta$ \cite[Lemma 3.11]{duistermaat1979}. Finally, using the positivity of $\widehat k_{t\nu}$ and the lower bound from Lemma~\ref{archimedeantestfunction} we get that the double integral is $\asymp \beta(t\nu) (1+t)^{-r}$.

The other terms in Lemma~\ref{classicalmainplusoffdiag} are finite in number by the support condition on $k_{t\nu}$. By our assumptions on $G$ and $X$, Lemma~\ref{groupshavenolevi} says that $g^{-1}\gamma g \notin \bigcup_{L \in \mathcal L - \{G\}} M' L$ for such terms. Using Proposition~\ref{orbitalintegralboundclosetolevi} and entirely similar arguments as for the diagonal term, one shows that they are $\ll \beta(t\nu) (1+t)^{-r - \delta}$ for some $\delta > 0$.
\end{proof}

\begin{proof}[Proof of Theorem~\ref{meansquareasymptotic}]We wish to replace the weight $\widehat k_{t\nu}$ in Lemma~\ref{asymptoticwithweightfunction} by a sharp cutoff. This can be done by applying the lemma to various $t\nu \in \mathfrak a^*$; see for example \cite[Lemma 4.5]{brumley2020}.
\end{proof}

\section{Archimedean model integrals}

\label{sectionmaintermintegral}

The main result of this section is the following proposition, whose proof is in \S\ref{maintermfinalproofsection}. The notations for Lie groups and Lie algebras are as in \S \ref{notationliegroups}.

\begin{proposition} \label{mainintegralsphericalfunction}Let $G$ be a semisimple Lie group and $b \in C_{c}^{\infty}(A)$ with $b(1) > 0$. Let $D_{\mathfrak{a}^*} \subset (\mathfrak{a}^*)^{\operatorname{gen}}$ be a compact set. Then
\[ \int_{A} \varphi_{it \nu}(a) b(a) da \asymp (1 + t)^{-r} \,, \]
uniformly for $t \in \mathbb{R}$ and $\nu \in D_{\mathfrak{a}^*}$.
\end{proposition}

\subsection{Setup}

We fix $b \in C_{c}^{\infty}(A)$ with $b(1) > 0$ and we will not indicate the dependence on $b$ in our notations. Take $\nu \in \mathfrak a^*$ and let $H_0 \in \mathfrak a$ be the corresponding element under the isomorphism given by the Killing form. Define
\[ I(H_{0}) = \int_{A} \varphi_{i\nu}(a) b(a) da \,. \]

Inserting Harish-Chandra's formula for the spherical function \eqref{sphericalfunctionformula} yields the oscillatory integral
\begin{equation}\label{maintermintegralwithphase}
I(H_{0}) = \int_{A} \int_{K} \exp \left( i \phi_{H_{0}}(a, k) \right) b' (a, k) dk da \,,
\end{equation}
with phase function
\begin{equation} \label{definitionphasemainterm}
\phi_{H_{0}}(a, k) = \langle H_{0}, H(k a) \rangle
\end{equation}
and with an amplitude function $b'\in C_{c}^{\infty}(A \times K)$ satisfying $b' (1, k) = b(1) > 0$.
We will determine the critical points of $\phi_{H_0}$ and obtain Proposition~\ref{mainintegralsphericalfunction} as an application of the stationary phase method.

\subsection{Structure of the critical set}

By Lemma~\ref{computationAfirstderivative} we have for $H \in \mathfrak{a}$ that
\begin{equation}\label{computationphiAderivative}
(D \phi_{H_{0}} (\cdot, k))_{a} (H) = \langle H_{0}, \operatorname{Ad}_{\kappa(ka)}(H) \rangle \,.
\end{equation}
Define the set
\begin{equation} \label{definitionCreductive}
\mathcal{C}(G, H_{0}) = \{ k \in K : \operatorname{Ad}_{k^{-1}} (H_{0}) \perp \mathfrak{a} \} \,.
\end{equation}

\begin{lemma}\label{criticalsetmainterm}When $H_{0} \in \mathfrak{a}^{\operatorname{gen}}$, the set of critical points of $\phi_{H_{0}}$ is $\{1\} \times \mathcal C(G, H_0)$.
\end{lemma}

\begin{proof}
Assume $(a, k)$ is a critical point of $\phi_{H_{0}}$. By \cite[Proposition 5.4]{duistermaat1983}, criticality in $k$ is equivalent to
\[ k \in Z_{K}(H_{0}) M' Z_{K}(a) = M' Z_{K}(a) \,, \]
where we have used that $H_{0} \in \mathfrak{a}^{\operatorname{reg}}$, so that $Z_{K}(H_{0}) = M$. By \eqref{computationphiAderivative} and the definition \eqref{definitionCreductive}, criticality in $a$ is equivalent to $\kappa(ka) \in \mathcal C(G, H_0)$. Because $k \in M' Z_{K}(a)$ we have $\kappa(ka) = k$, so that $k \in \mathcal C(G, H_0)$. It remains to show that $a = 1$. If not, we would have $k \in M' L$ with $L = Z_G(a) \in \mathcal L$ a Levi subgroup different from $G$. Writing $k = m\ell$, the condition that $\operatorname{Ad}_{k^{-1}}(H_0) \perp \mathfrak a_L$ implies $H_0 \perp \operatorname{Ad}_{m^{-1}}(\mathfrak a_L)$ and so $H_0 \in \mathfrak a^{m^{-1}L m}$. This contradicts the hypothesis that $H_0 \in \mathfrak a^{\operatorname{gen}}$.\end{proof}

\begin{remark}\label{criticalsetmaintermmoreprecise}When more generally $H_{0} \in \mathfrak{a}^{\operatorname{reg}}$, by being more careful in the proof of Lemma~\ref{criticalsetmainterm} one shows that the critical set of $\phi_{H_0}$ is
\[ \bigcup_{L \in \mathcal L} \operatorname{Ad}_{M'}(A_{L}) \times \mathcal{C}(L, H_{0}) \,, \]
where the sets $\mathcal{C}(L, H_{0})$ are defined analogously by
\[ \mathcal{C}(L, H_{0}) = \{k \in K \cap L : \operatorname{Ad}_{k^{-1}}(H_0) \perp \mathfrak a \}\,, \]
and only the sets $\mathcal{C}(L, H_{0})$ with $H_0 \in \mathfrak a^L$ are nonempty. \end{remark}

Define the map 
\begin{align}\label{definitionfmapC}
\begin{split}
f_{H_0} : K & \to \mathfrak a \\
k & \mapsto E_{\mathfrak a}(\operatorname{Ad}_{k^{-1}} (H_{0})) \,,
\end{split}
\end{align}
which has the properties that $\mathcal{C}(G, H_{0}) = f_{H_0}^{-1}(0)$ and
\begin{equation}\label{phifrelation}
(D \phi_{H_0}(\cdot, k))_a(H) = \langle f_{H_0}(\kappa(ka)), H \rangle \,.
\end{equation}

\begin{remark}The set $\mathcal{C}(G, H_{0})$ has the geometric interpretation that it consists of the $k \in K$ for which the maximal flat $kA \subset G/K$ is orthogonal to $H_0$ at the base point $k\cdot 1$. The geometric picture will come more fully into its own right in \S\ref{orbitalintegralsection}. \end{remark} 

The main properties of $\mathcal C(G, H_0)$ are established in \cite{Michels2022cartan}.

\begin{lemma}\label{genericcriticalsetstructure} \label{maintermgenericsubmersionatcriticalset} \label{CdoesnotmeetMLarticle} When $H_0 \in \mathfrak a^{\operatorname{gen}}$, the function $f_{H_0}$ is a submersion at the points of $\mathcal C(G, H_0)$. The set $\mathcal C(G, H_0)$ is a nonempty smooth submanifold of $K$ of codimension 
$\dim(A)$ that varies smoothly with $H_0 \in \mathfrak a^{\operatorname{gen}}$. It is disjoint from the sets $M'L$ with $L \in \mathcal L - \{G\}$.
\end{lemma}

\begin{proof}
That $f_{H_0}$ is a submersion at $\mathcal C(G, H_0)$ is \cite[Corollary 4.9]{Michels2022cartan}. The second sentence is \cite[Proposition 4.19]{Michels2022cartan}.  The third sentence is \cite[Lemma 4.7]{Michels2022cartan}, although the (short) argument for this already appeared in the proof of Lemma~\ref{criticalsetmainterm} above.
\end{proof}

Note that the smoothness and the codimension of $\mathcal C(G, H_0)$ are immediate consequences of the fact that $f_{H_0}$ is a submersion, but the fact that it is nonempty, is a major result established in \cite{Michels2022cartan}.

\subsection{Stationary phase}

Recall the map $f_{H_0} : K \to \mathfrak a$ defined in \eqref{definitionfmapC}.

\begin{lemma}\label{tangentspaceC}For $k \in \mathcal{C}(G, H_{0})$ we have
\[ T_k(\mathcal{C}(G, H_{0})) = \ker((D f_{H_{0}})_{k}) = \{ X \in \mathfrak k : \operatorname{Ad}_k(X) \perp [H_0, \operatorname{Ad}_k(\mathfrak a)]\}  \,.\]\end{lemma}

\begin{proof}
The first statement holds because the defining map $f_{H_0}$ for $\mathcal{C}(G, H_{0})$ is a submersion on $\mathcal C(G, H_0)$ (Lemma~\ref{maintermgenericsubmersionatcriticalset}). Indeed, it is a general fact that when a submanifold is the level set of a submersion, then its tangent spaces are the kernels of the differential of the defining map \cite[Proposition 5.38]{lee2013}. The differential may be computed using \eqref{derivativeAdinverse}:
\[ (Df_{H_0})_k (X) = - E_{\mathfrak a}([X, \operatorname{Ad}_{k}^{-1}(H_0)]) \,.\]
Therefore $X \in \ker((D f_{H_{0}})_{k})$ if and only if $[X, \operatorname{Ad}_{k}^{-1}(H_0)] \perp \mathfrak a$. Using associativity and $\operatorname{Ad}_k$-invariance of the Killing form, this is equivalent to $\operatorname{Ad}_k(X) \perp [H_0, \operatorname{Ad}_k(\mathfrak a)]$.
\end{proof}

The first equality in Lemma~\ref{tangentspaceC} will be used in the following lemma. The second will be used in the proof of Lemma~\ref{orthcomplementCHnought}.

When $M \subset N$ are Riemannian manifolds and $m \in M$, a symmetric bilinear form on $T_m N$ is called transversely nondegenerate to $M$ if its radical is contained in $T_m M$. Recall the phase function $\phi_{H_{0}}$ defined by \eqref{definitionphasemainterm}.

\begin{lemma} \label{maintermhessiannondeg}When $H_{0} \in \mathfrak{a}^{\operatorname{gen}}$, the Hessian of $\phi_{H_{0}}$ is transversely nondegenerate to the critical set of $\phi_{H_{0}}$. Moreover, its signature on the critical set equals $(n_{0}, n_{+}, n_{-}) = (\dim(K) - r, r, r)$, where $r = \dim(A)$.
\end{lemma}

\begin{proof} The Hessian of $\phi_{H_0}$ at a critical point $(a, k)$ is given simply by
\[ (X, Y) \mapsto L_X L_Y \phi_{H_0}(a, k) \,. \]
By Lemma~\ref{criticalsetmainterm} the critical set of $\phi_{H_{0}}$ is $\{ 1 \} \times \mathcal{C}(G, H_{0})$. Let $(1, k)$ be a critical point of $\phi_{H_{0}}$. Let $L : \mathfrak{a} \oplus \mathfrak{k} \to \mathfrak{a} \oplus \mathfrak{k}$ be the unique and self-adjoint linear map such that
\[ (\operatorname{Hess}_{(1, k)} \phi_{H_{0}})(X, Y) = \langle L X, Y \rangle_{\theta} \]
for all $X, Y \in \mathfrak{a} \oplus \mathfrak{k}$. We must show that $\ker L \subset \{ 0 \} \oplus T_{k} \mathcal{C}(G, H_{0})$. Write
\[ L = \begin{pmatrix}
L_{\mathfrak{a} \mathfrak{a}} & L_{\mathfrak{a} \mathfrak{k}} \\
L_{\mathfrak{a} \mathfrak{k}}^{*} & L_{\mathfrak{k} \mathfrak{k}}
\end{pmatrix} \]
relative to the natural decomposition of $\mathfrak{a} \oplus \mathfrak{k}$. Because $\phi_{H_{0}}(1, k) = 0$ for all $k \in K$, it is clear that $L_{\mathfrak{k} \mathfrak{k}} = 0$.
To compute $L_{\mathfrak{a} \mathfrak{k}} : \mathfrak{k} \to \mathfrak{a}$, let $X \in \mathfrak{k}$ and $H\in \mathfrak{a}$. From \eqref{phifrelation} we have $L_{H} \phi_{H_{0}}(1, k) = \langle f_{H_{0}}(k), H \rangle$. Therefore
\begin{align*}
\operatorname{Hess}_{(1, k)} \phi_{H_{0}} (X, H) & = L_{XH} \phi_{H_{0}}(a, k) \\
& = L_{X} \langle f_{H_{0}}(k), H\rangle \\
& = \langle (D f_{H_{0}})_{k}(X), H\rangle \,,
\end{align*}
so that $L_{\mathfrak{a} \mathfrak{k}} = (D f_{H_{0}})_{k}$. From Lemma~\ref{maintermgenericsubmersionatcriticalset} it follows that $L_{\mathfrak{a} \mathfrak{k}}$ is surjective, so that the adjoint $L_{\mathfrak{a} \mathfrak{k}}^{*}$ is injective. This implies
\[ \ker L = \{ 0 \} \oplus \ker L_{\mathfrak{a} \mathfrak{k}} =  \{ 0 \} \oplus \ker ((D f_{H_{0}})_{k}) = \{ 0 \} \oplus T_{k} \mathcal{C}(G, H_{0}) \,, \]
where the last equality is Lemma~\ref{tangentspaceC}. This proves the first part of the statement.

We can compute the signature at a critical point $(1, k)$ as follows. Let $V$ be a complement of $T_{k} \mathcal{C}(G, H_{0})$ in $\mathfrak{k}$. Relative to the decomposition $\mathfrak{a} \oplus V \oplus T_{k} \mathcal{C}(G, H_{0})$, the map $L$ has the form
\[ L = \begin{pmatrix}
L_{\mathfrak{a} \mathfrak{a}} & L_{\mathfrak{a} V} & 0 \\
L_{\mathfrak{a} V}^{*} & 0 & 0 \\
0 & 0 & 0 
\end{pmatrix} \]
with $L_{\mathfrak{a} V}$ invertible, because $\ker(L_{\mathfrak ak}) = T_{k} \mathcal{C}(G, H_{0})$. A self-adjoint map this form has signature $(n_0, n_+, n_-) = (\dim(\mathcal{C}(G, H_{0})), r, r)$. Indeed, to show this we must show that a real symmetric $2r \times 2r$ block matrix of the form
\[ M = \begin{pmatrix}
A & B \\
B^T & 0
\end{pmatrix} \]
with $B$ invertible, has signature $(r, r)$. Doing as if the blocks are numbers, define
\[ \Lambda_{\pm} =  \frac12 (A  \pm \sqrt{A^2+ 4BB^T})\,,\]
where the square root is the symmetric positive definite one. Because ``square roots are monotone on symmetric positive definite matrices'', $\Lambda_{+}$ is positive definite and $\Lambda_-$ is negative definite.
Now one can either write down explicit matrices to show that $M$ is congruent to $\operatorname{diag}(\Lambda_+, \Lambda_-)$, or observe that if $(v^\pm_i)$ are bases of eigenvectors for $\Lambda_\pm$, then the vectors $(\Lambda_+ v^+_i, B^T v^+_i)$ and $(\Lambda_- v^-_i, B^T v^-_i)$ form a basis of eigenvectors for $M$, with the same eigenvalues as those of $\Lambda_{\pm}$.
\end{proof}

We are ready to prove Proposition~\ref{mainintegralsphericalfunction}. We use the stationary phase theorem with critical manifolds of positive dimension, which we now recall.

\begin{theorem}\cite[Th\'{e}or\`{e}me 4.1]{colindeverdiere1973} Let $M$ be a smooth Riemannian manifold, $b \in C_c^\infty(M)$ and $\phi \in C^\infty(X)$ real-valued. Assume that the set of critical points of $\phi$ intersects the support of $b$ in a smooth closed submanifold $W \subset M$ of codimension $d$, and that the Hessian of $\phi$ is transversely nondegenerate to $W$. Denote for $w \in W$ by $\operatorname{Hess}_{w, \perp} \phi$ the restriction of the Hessian to the orthogonal complement of $T_w W$. Assume that $\phi$ is constant on $W$ and that $\operatorname{Hess}_\perp \phi$ has constant signature $(n_+, n_-)$ on $W$. Then as $t \to +\infty$,
\begin{align*}
\int_M e^{it\phi(x)} b(x) dx &= \left( \frac{2\pi}{t} \right)^{d/2} e^{itf(W) + \pi i (n_{+} - n_{-}) /4} \cdot \\
& \qquad \int_W \frac{b(w)}{|\det(\operatorname{Hess}_{w, \perp} \phi)|^{1/2}} dw + O(t^{-d/2 -1})\,,
\end{align*}
where the integral over $W$ is with respect to the induced metric.
\end{theorem}

\subsection{Proof of Proposition~\ref{mainintegralsphericalfunction}}

\label{maintermfinalproofsection}

\begin{proof}[Proof of Proposition~\ref{mainintegralsphericalfunction}] Take $\nu_0 \in(\mathfrak a^*)^{\operatorname{gen}}$ and let $H_0 \in\mathfrak a^{\operatorname{gen}}$ correspond to it via the identification $\mathfrak a \cong \mathfrak a^*$ given by the Killing form. Via the reduction in the beginning of \S\ref{sectionmaintermintegral} we must bound the oscillatory integrals $I(tH_0)$ given by \eqref{maintermintegralwithphase}:
\[ I(tH_{0}) = \int_{A} \int_{K} \exp \left( i t \phi_{H_{0}}(a, k) \right) b' (a, k) dk da \,. \]
By Lemma~\ref{maintermhessiannondeg} the Hessian of $\phi_{H_{0}}$ has signature $(n_+, n_-) = (r, r)$ transversely to its critical set, which has codimension $2 \dim(A) = 2r$ by Lemma~\ref{criticalsetmainterm} and Lemma~\ref{genericcriticalsetstructure}, and $\phi_{H_{0}}$ takes the value $0$ there because $H(K) = 0$. The stationary phase theorem implies
\begin{align*}
I(t H_{0}) =  t^{-r} \cdot (2 \pi)^{r} e^{\pi i (n_{+} - n_{-}) /4} \int_{M' \mathcal{C}(G, H_{0})} \frac{b' (1, k)}{\sqrt{|\det(\mathbf H_k)|}} dk + O(t^{-r- 1})
\end{align*}
as $t \to + \infty$, where $\mathbf H_k =\operatorname{Hess}_{(1, k), \perp} \phi_{H_{0}}$ is the Hessian of $\phi_{H_{0}}$ restricted to the orthogonal complement of the tangent space $T_{(1, k)}(A \times M' \mathcal{C}(G, H_{0}))$, the determinant is taken in an orthonormal basis and the integral is with respect to the induced metric. Moreover, this bound is uniform for $H_0$ in compact subsets of $\mathfrak{a}^{\operatorname{gen}}$ because of the smooth dependence in Lemma~\ref{genericcriticalsetstructure}. We have $b' (1, k) > 0$ (see the beginning of \S \ref{sectionmaintermintegral}), so that the constant in the main term is strictly positive.
\end{proof}

\begin{remark}A possible way to remove the condition $\nu \in \mathfrak a^{\operatorname{gen}}$ in Proposition~\ref{mainintegralsphericalfunction} (or at least to replace $\mathfrak a^{\operatorname{gen}}$ by $\mathfrak a^{\operatorname{reg}}$) would be to use the more precise statement about the critical set in Remark~\ref{criticalsetmaintermmoreprecise} and to resolve the singularities of $f_{H_0}$ and lift the integration in \eqref{maintermintegralwithphase} to a blowup of $\mathfrak a \times A \times K$ (the $\mathfrak a$-factor corresponding to $H_0$).
\end{remark}

\section{Bounds for orbital integrals}

\label{orbitalintegralsection}

In this section we prove the following proposition. The notations for Lie groups and Lie algebras are as in \S \ref{notationliegroups}.

\begin{proposition} \label{orbitalintegralboundclosetolevi} Let $G$ be semisimple and $b \in C_{c}^{\infty}(A \times A)$. Let $D_{\mathfrak{a}^*} \subset (\mathfrak{a}^*)^{\operatorname{gen}}$ be compact and $D_{G} \subset G$ be compact. Then there exist $\delta > 0$ and $N > 0$ such that
\begin{align*}
&\int_{A \times A} \varphi_{i t \nu}(a_{1}^{-1} g a_{2}) b(a_{1}, a_{2}) da_{1} d a_{2} \\
&\qquad\ll (1 + t)^{-r} \cdot \left( 1 + t \cdot d \left( g, \bigcup_{L \in \mathcal L - \{ G\}} M' L \right)^{N} \right)^{-\delta} \,,
\end{align*}
uniformly for $t \in \mathbb{R}$, $\nu \in D_{\mathfrak{a}^*}$ and $g \in D_{G}$.
\end{proposition}

\begin{remark}The dependence on $g$ in the above result is not needed in the proof of Theorem~\ref{meansquareasymptotic}, because only finitely many $\gamma \in \Gamma$ contribute to the sum in Lemma~\ref{classicalmainplusoffdiag}. We will need this dependence in the proof of Theorem~\ref{largevaluesslthree}. \end{remark}

\subsection{Setup and phase functions}

We fix $b \in C_{c}^{\infty}(A \times A)$ and we will not incorporate it in the notations. Take $\nu \in \mathfrak a^*$ and let $H_0 \in \mathfrak a$ be the corresponding element under the isomorphism given by the Killing form. For $g \in G$ define
\begin{equation} \label{definitionJ}
J(H_{0}, g) = \int_{A \times A} \varphi_{i \nu}(a_{1}^{-1}g a_{2}) b(a_{1}, a_{2}) da_{1} da_{2} \,.
\end{equation}
By invariance of $\varphi_{i\nu}$ under the action of the Weyl group on $\nu$, we have for $e \in M'$ that
\begin{equation} \label{JintegralWeylinvariance}
J(H_{0}, g) = J(\operatorname{Ad}_{e}(H_{0}), g) \,.
\end{equation}
It is therefore no restriction to assume that $H_{0}$ lies in a Weyl chamber of our choice.

Inserting Harish-Chandra's formula for the spherical function \eqref{sphericalfunctionformula} yields the oscillatory integral
\[ J(H_{0}, g) = \int_{A \times A} \int_{K} \exp \left( i \widetilde{\phi}_{H_{0}, g}(a_{1}, a_{2}, k) \right) b_{g} (a_{1}, a_{2}, k) d k da_{1} da_{2} \,, \]
with phase function
\begin{equation} \label{definitionphinought}
\widetilde{\phi}_{H_{0}, g}(a_{1}, a_{2}, k) = \langle H_{0}, H(k a_{1}^{-1} g a_{2}) \rangle
\end{equation} and with amplitude $b_{g} \in C_{c}^{\infty}(A \times A \times K)$ depending smoothly on $g$ and with support bounded independently of $g$, which incorporates the real exponential factor in \eqref{sphericalfunctionformula}. Following \cite{Marshall2016}, we now bring this in a form that makes the $A$-derivatives more manageable.

For $h \in G$ define the map
\begin{align}\label{definitionTheta}
\begin{split}
\Theta_h : K & \to K \\
k & \mapsto \kappa(kh) \,.
\end{split}
\end{align}
By smoothness of the Iwasawa decomposition it is a smooth map, with smooth inverse $\Theta_{h^{-1}}$, and therefore a diffeomorphism. For $k \in K$ and $y, z \in G$ we have that (see \cite[Lemma~6.2]{Marshall2016})
\[ H(k y^{-1} z) = H(\Theta_{y^{-1}}(k) z) - H(\Theta_{y^{-1}}(k) y) \,. \]
Applying this with $y = a_{1}$ and $z = ga_{2}$ and making the change of variables $k \leftarrow  \Theta_{a_{1}}(k)$ gives
\begin{equation}
\label{expressionJrgphi}
J(H_{0}, g)  = \int_{K} \int_{A \times A} \exp( i \phi_{H_{0}, g}(a_{1}, a_{2}, k) ) b'_{g} (a_{1}, a_{2}, k) da_{1} da_{2} d k
\end{equation}
with phase function
\begin{align} \label{definitionphi}
\begin{split}
\phi_{H_{0}, g}(a_{1}, a_{2}, k) &:= \widetilde{\phi}_{H_{0}, g}(a_{1}, a_{2}, \Theta_{a_{1}}(k)) \\
& = \left\langle H_{0}, H(k g a_{2}) \right\rangle - \left\langle H_{0}, H(k a_{1}) \right\rangle
\end{split}
\end{align}
and some amplitude $b'_{g} \in C_{c}^{\infty}(A \times A \times K)$ depending smoothly on $g$.

\begin{remark}The expression \eqref{definitionphinought} will be useful when computing derivatives of $\phi_{H_{0}, g}$ with respect to $k$, and \eqref{definitionphi} will be useful when computing derivatives with respect to $a_{1}$ and $a_{2}$.\end{remark}

The expression \eqref{definitionphi} separates the variables $a_{1}$ and $a_{2}$. Our strategy, inspired by \cite{Marshall2016}, is to first apply the stationary phase theorem in the variables $a_{1}$ and $a_{2}$, leaving us with an oscillatory integral over $K$, and then to apply the van der Corput lemma to this integral.

\subsection{Extremal points on maximal flats}

\label{extremalpointssummarysection}

In view of equation \eqref{definitionphi}, we are naturally led to study the critical points of the `height' functions
\begin{align*}
h_{H_{0}, g} : A & \to \mathbb{R} \\
a & \mapsto \langle H_{0}, H(g a) \rangle
\end{align*}
with $g \in G$, which allow us to write
\begin{equation}
\label{phiintermsofheightfunctions}
\phi_{H_{0}, g}(a_{1}, a_{2}, k) = h_{H_{0}, kg}(a_{2}) - h_{H_{0}, k}(a_{1}) \,.
\end{equation}
The critical points of $h_{H_{0}, g}$ are studied in \cite{Michels2022cartan}. Many of the results concerning them require that $H_0 \in \mathfrak{a}^{\operatorname{gen}, +}$, so that this assumption is propagated throughout most of the analysis in this section. We summarize the results as follows. Recall the set $\mathcal{C}(G, H_{0}) \subset K$ defined in \eqref{definitionCreductive}.

\begin{lemma}\label{extremalpointssummary}\label{heightfunctionhessiannondeg}Let $H_0 \in \mathfrak a^{\operatorname{gen}, +}$. Then $h_{H_{0}, g}$ has at most one critical point. A critical point $a \in A$ is characterized by the condition that $\kappa(ga) \in \mathcal C(G, H_0)$. The Hessian at a critical point is negative definite and (as a quadratic form) given by
\begin{align*}
\mathfrak a & \to \mathbb R \\
H & \mapsto \left\langle [H_{0}, \operatorname{Ad}_{c}(H)], E_{\mathfrak{n}} (\operatorname{Ad}_{c}(H)) \right\rangle \,,
\end{align*}
where $c = \kappa(ga)$.
\end{lemma}

\begin{proof}
The uniqueness and the fact that the Hessian is negative definite are contained in \cite[Theorem~1.2]{Michels2022cartan}. The characterization is \cite[Lemma~4.1]{Michels2022cartan}. The Hessian is computed in \cite[Proposition~4.16]{Michels2022cartan} when $g \in K$, but it does not depend on the triangular part of $g$ by \cite[equation (12)]{Michels2022cartan}.
\end{proof}

We require some additional facts about the dependence of the critical point of $h_{H_{0}, g}$ on the parameters $H_0$ and $g$. Take $H_{0} \in \mathfrak{a}^{\operatorname{gen}, +}$. Define $\mathcal{R}_{H_{0}} \subset G$ to be the set of elements $g$ for which the function $h_{H_{0}, g}$ has a critical point. Define also
\begin{equation} \label{definitionR}
\mathcal{R} = \bigcup_{H_{0} \in \mathfrak{a}^{\operatorname{gen}, +}} \{ H_{0} \} \times \mathcal{R}_{H_{0}} \,.
\end{equation}
When $g \in \mathcal{R}_{H_{0}}$, by Lemma~\ref{extremalpointssummary} there is a unique critical point of $h_{H_{0}, g}$. Define the function
\begin{equation} \label{definitionxiH}
\xi_{H_{0}} : \mathcal{R}_{H_{0}} \to A
\end{equation}
that sends $g$ to the critical point of $h_{H_{0}, g}$, and define
\begin{align*}
\xi : \mathcal{R} & \to A \\
(H_{0}, g) & \mapsto \xi_{H_{0}}(g) \,.
\end{align*}
Define also the $\mathcal{C}$-projection
\begin{align*}
c_{H_0} : \mathcal{R}_{H_0} & \to \mathcal C(G, H_0) \\
g & \mapsto \kappa(g \xi_{H_{0}}(g))
\end{align*}
which is guaranteed to take values in $\mathcal{C}(G, H_{0})$ by Lemma~\ref{extremalpointssummary}.

\begin{lemma} \label{xismoothdependence}The set $\mathcal{R} \subset \mathfrak{a} \times G$ is open, and $\xi_{H_{0}}(g)$ and $c_{H_{0}}(g)$ are real analytic in $(H_{0}, g) \in \mathcal{R}$. \end{lemma}

\begin{proof} Take $(H_{0}, g) \in \mathcal{R}$. By Lemma~\ref{extremalpointssummary} the critical points of $h_{H_{0}, g}$ are nondegenerate, which we may reformulate by saying that the map
\begin{align*}
A & \to \mathfrak{a}^{*} \\
a & \mapsto (D h_{H_{0}, g})_{a}
\end{align*}
has invertible differential at the level set above $0$, which is the singleton $\{\xi_{H_{0}}(g) \}$. By the implicit function theorem applied to this real analytic map with parameter $(H_{0}, g) \in \mathfrak{a}^{\operatorname{gen}, +} \times G$, it follows that $\mathcal{R}$ is open in $\mathfrak{a}^{\operatorname{gen}, +} \times G$, and that $\xi_{H_{0}}(g)$ is real analytic in $(H_{0}, g)$. Consequently, $c_{H_{0}}(g)$ is also real analytic in $(H_{0}, g)$.
\end{proof}

The following coordinate system for $\mathcal R_{H_0} \cap K$ will be useful when dealing with expressions involving the $\xi_{H_0}$.

\begin{lemma} \label{regularsetKintermsofC}
Let $H_{0} \in \mathfrak{a}^{\operatorname{gen}, +}$. The map
\begin{align*}
\mathcal{C}(G, H_{0}) \times A & \to \mathcal{R}_{H_{0}} \cap K \\
(c, a) & \mapsto \kappa(ca^{-1})
\end{align*}
is a real analytic isomorphism whose inverse is
\[ k \mapsto (c_{H_{0}}(k), \xi_{H_{0}}(k)) \,.  \]
\end{lemma}

\begin{proof}
The two maps are real analytic by Lemma~\ref{xismoothdependence}, and the fact that they are mutual inverses follows from the definitions of $\xi_{H_{0}}(k)$ and $c_{H_{0}}(k)$.\end{proof}

\begin{example}When $G = \operatorname{PSL}(2, \mathbb{R})$ with the standard choice of Iwasawa decomposition, Lemma~\ref{regularsetKintermsofC} can be visualized as follows. Identify $G$ with the unit tangent bundle of $\mathbb H$ and $K$-projections with unit tangent vectors. We have $\mathcal{R}_{H_{0}} \cap K = K - M'$, which has two connected components, corresponding to the directions pointing east or west. (The north and south directions are excluded as they come from $M'$.) Accordingly, the set $\mathcal{C}(G, H_{0})$ has two points:
\begin{align*}
c_{\pm} = \begin{pmatrix}
\cos(\pi /4) & \pm \sin(\pi /4) \\
\mp \sin(\pi /4) & \cos(\pi /4)
\end{pmatrix} \,.
\end{align*}
When $a$ runs through $A$, the direction of $c_{+} a$ runs through all directions pointing east, and the direction of $c_{-} a$ runs trough all directions pointing west.\end{example}

\subsection{Reduction to an integral over \texorpdfstring{$K$}{K}}

We seek to apply the stationary phase theorem to evaluate the inner integral in \eqref{expressionJrgphi}. The main result of this subsection is Proposition~\ref{reductiontokintegral}. Given the results of \S\ref{extremalpointssummarysection}, for which most of the work was done in \cite{Michels2022cartan}, the proof is parallel to that of \cite[Lemma~7.10]{Marshall2016} when $G = \operatorname{PSL}_2(\mathbb R)$. Some care must be taken to obtain uniformity in $H_{0}$.

Define the ``parameter space''
\[ \mathcal{P} = \mathfrak{a}^{\operatorname{gen}, +} \times G \times K \]
and let $\mathcal{R}' \subset \mathcal{P}$ be the set of parameters $(H_{0}, g, k)$ for which the phase function $\phi_{H_{0}, g}(\cdot, \cdot, k)$ has a critical point. Denote by $\mathcal{R}'_{H_{0}, g} \subset K$ the fiber of $\mathcal{R}'$ above $(H_{0}, g)$.

\begin{lemma} \label{acriticalityapplied}The set $\mathcal{R}'$ is open in $\mathcal{P}$. When $(H_{0}, g, k) \in \mathcal{R}'$, the function $\phi_{H_{0}, g}(\cdot, \cdot, k)$ has a unique critical point $(a_{1}, a_{2})$ given by $(\xi_{H_{0}}(k), \xi_{H_{0}}(kg))$.
\end{lemma}

\begin{proof}Recall from \eqref{phiintermsofheightfunctions} that
\[ \phi_{H_{0}, g}(a_{1}, a_{2}, k) = h_{H_{0}, kg}(a_{2}) - h_{H_{0}, k}(a_{1}) \,, \]
so that $(a_{1}, a_{2})$ is a critical point of $\phi_{H_{0}, g}(\cdot, \cdot, k)$ if and only if $a_{1}$ is a critical point of $h_{H_{0}, k}$ and $a_{2}$ is a critical point of $h_{H_{0}, kg}$. Lemma~\ref{extremalpointssummary} gives the uniqueness, and the last part is the definition of $\xi_{H_0}$ (see \eqref{definitionxiH}).
We show that $\mathcal{R}'$ is open. Let $\mathcal{R}$ be as in \eqref{definitionR}; it is open in $\mathfrak{a} \times G$ by Lemma~\ref{xismoothdependence}, and it suffices to note that $\mathcal{R}'$ is the preimage of $\mathcal{R} \times \mathcal{R}$ under the continuous map
\begin{align*}
\mathcal{P} & \to (\mathfrak{a} \times G) \times (\mathfrak{a} \times G) \\
(H_{0}, g, k) & \mapsto ((H_{0}, k), (H_{0}, kg)) \,. \qedhere
\end{align*}
\end{proof}

Define a function on $K$ by
\begin{equation} \label{definitionpsi}
\psi_{H_{0}, g}(k) = \begin{cases}
\phi_{H_{0}, g}(\xi_{H_{0}}(k), \xi_{H_{0}}(kg), k) & \; \text{when} \; k \in \mathcal{R}'_{H_{0}, g} \,, \\
0 & \; \text{otherwise.}
\end{cases}
\end{equation}
When $(H_{0}, g, k) \in \mathcal{R}'$, denote by $d_{H_{0}, g}(k)$ the Hessian determinant of the function $\phi_{H_{0}, g}(\cdot, \cdot, k)$ (for fixed $k$) at its unique critical point. Let $b_g'$ be as in \eqref{expressionJrgphi} and define
\begin{align*}
b''_{H_{0}, g} (k) = \begin{cases}
b'_{g}(\xi_{H_{0}}(k), \xi_{H_{0}}(kg), k) \frac{(2 \pi)^{r}}{ \sqrt{ \lvert d_{H_{0}, g}(k) \rvert}} & \; \text{when} \; k \in \mathcal{R}'_{H_{0}, g} \,, \\
0 & \; \text{otherwise,}
\end{cases}
\end{align*}
and call $b'' : \mathcal{P} \to \mathbb{C}$ the corresponding function of $(H_{0}, g, k)$.

\begin{lemma} \label{cutoffbsmooth}The function $b''$ is smooth and $\operatorname{supp}(b'') \subset \mathcal R'$.
\end{lemma}

\begin{proof} Roughly speaking, this is because $\xi_{H_0}(g)$ diverges as $(H_0, g)$ approaches the boundary of $\mathcal R$. We must only show that every point of $\mathcal P - \mathcal R'$ has a neighborhood on which $b''$ is zero. Suppose not, then there exists a sequence $(H_{0, n}, g_n, k_n) \in \mathcal R'$ on which $b''$ is nonzero and which converges to a point $(H_0, g, k)$ of $\mathcal P - \mathcal R'$. Because $b''(H_{0, n}, g_n, k_n) \neq 0$ and $b_g'$ has support bounded independently of $g$ (even independent of $g$ altogether) the sequences $\xi_{H_{0, n}}(k_n)$ and $\xi_{H_{0, n}}(k_ng_n)$ are bounded. We may then extract a subsequence on which $\xi_{H_{0, n}}(k_n)$ and $\xi_{H_{0, n}}(k_ng_n)$ converge, say to $\xi_1, \xi_2 \in A$. By continuity we then have
\[ (D h_{H_{0}, k})_{\xi_1} = \lim_{n \to \infty} (D h_{H_{0, n}, k_{n}})_{\xi_{H_{0, n}}(k_{n})} = 0 \,, \]
and similarly for $\xi_2$, so that $(H_{0}, g, k) \in \mathcal{R}'$. This is a contradiction.
\end{proof}

\begin{proposition} \label{reductiontokintegral}Let $D_{\mathfrak{a}} \subset \mathfrak{a}^{\operatorname{gen}, +}$ be compact and $D_{G} \subset G$ be compact. Define $J(H_{0}, g)$ by \eqref{definitionJ}. Then
\[
J(tH_{0}, g) = t^{-r} \int_{K} e^{it \psi_{H_{0}, g}(k)}  b''_{H_{0}, g}(k) dk + O(t^{-r- 1}) 
\]
as $t \to + \infty$, uniformly for $H_{0} \in D_{\mathfrak{a}}$ and $g \in D_{G}$.
\end{proposition}

\begin{proof}
We prove a uniform asymptotic for the double $A$-integral in \eqref{expressionJrgphi}, which we then integrate over $K$. Let $\mathcal{P}_{0} = \operatorname{supp}(b'') \cap (D_{\mathfrak{a}} \times D_{G} \times K)$. It is closed in $\mathcal{P}$ and contained in $\mathcal{R}'$ by Lemma~\ref{cutoffbsmooth}. We distinguish two cases depending on where $(H_{0}, g, k)$ lies.
When $(H_{0}, g, k) \in \mathcal{R}'$, the phase function $\phi_{H_{0}, g}(\cdot, \cdot, k)$ has a unique and nondegenerate critical point $(\xi_{H_{0}}(k), \xi_{H_{0}}(kg))$, where the Hessian has signature $(r, r)$ by \eqref{phiintermsofheightfunctions} and Lemma~\ref{extremalpointssummary}. The stationary phase theorem \cite[\S VIII.2]{stein1993} implies that
\begin{align} \label{Hintegralboundstatphase}
\begin{split}
\int_{A \times A} e^{it \phi_{H_{0}, g}(a_{1}, a_{2}, k)} b'_{H_{0}, g}(a_{1}, a_{2}, k) da_{1} da_{2}
 & = t^{-r} e^{it \psi_{H_{0}, g}(k)} b''_{H_{0}, g}(k) \\
 & \quad + O(t^{-r- 1})
\end{split}
\end{align}
as $t \to + \infty$, uniformly for $(H_{0}, g, k)$ in compact subsets of $\mathcal{R}'$.
When $(H_{0}, g, k) \in \mathcal{P} - \mathcal{P}_{0}$, the phase function $\phi_{H_{0}, g}(\cdot, \cdot, k)$ has no critical points in the support of $b'_{H_{0}, g}$, and the Van der Corput lemma \cite[\S VIII.2]{stein1993} implies that
\[
\int_{A \times A} e^{it \phi_{H_{0}, g}(a_{1}, a_{2}, k)} b'_{H_{0}, g}(a_{1}, a_{2}, k) da_{1} da_{2} \ll_{N} t^{-N}
\]
as $t \to + \infty$, uniformly for $(H_{0}, g, k)$ in compact subsets of $\mathcal{P} - \mathcal{P}_{0}$. Because $b''$ is by definition zero on $\mathcal{P} - \mathcal{P}_{0}$, the estimate \eqref{Hintegralboundstatphase} also holds in this case, uniformly in compact subsets of $\mathcal{P} - \mathcal{P}_{0}$.
Because $\mathcal{P} - \mathcal{P}_{0}$ and $\mathcal{R}'$ are open in $\mathcal{P}$ (Lemma~\ref{acriticalityapplied}) and they cover $\mathcal{P}$, we may find compact subsets of $\mathcal{R}'$ and $\mathcal{P} - \mathcal{P}_{0}$ that cover the compact set $\mathcal{P} \cap (D_{\mathfrak{a}} \times D_{G} \times K)$. Therefore \eqref{Hintegralboundstatphase} holds uniformly for all $(H_{0}, g, k) \in \mathcal{P} \cap (D_{\mathfrak{a}} \times D_{G} \times K)$. The statement follows then by integrating \eqref{Hintegralboundstatphase} over $K$.
\end{proof}

\subsection{Critical points of \texorpdfstring{$\psi_{H_{0}, g}$}{psi H0, g}}

This subsection is concerned with the critical points of $\psi_{H_{0}, g}$ (defined in \eqref{definitionpsi}) when $g$ lies in the dense open set $G - \bigcup_{L \in \mathcal{L} - \{ G \}} M' L$. The main result is Proposition~\ref{psinotlocallyconstantk}, which will imply Proposition~\ref{orbitalintegralboundclosetolevi}.

\begin{proposition} \label{psinotlocallyconstantk}Let $H_{0} \in \mathfrak{a}^{\operatorname{gen}, +}$ and $g \in G - \bigcup_{L \in \mathcal L - \{ G\}} M' L$. Then $\psi_{H_{0}, g}$ is nowhere locally constant in $\mathcal{R}'_{H_{0}, g}$.
\end{proposition}

The proof of Proposition~\ref{psinotlocallyconstantk} is at the end of this subsection.

\begin{remark}The proof differs from the proof in \cite{Marshall2016} for $\operatorname{PSL}_2(\mathbb R)$. Namely, we do not prove that the Hessian of $\psi_{H_0, g}$ is nondegenerate at critical points. In fact, to obtain a useful expression for the Hessian of $\psi_{H_0, g}$ in order to generalize to proof for $\operatorname{PSL}_2(\mathbb R)$, we must know something about the derivatives of the $\xi_{H_0}$ (see \eqref{definitionxiH}), which are determined implicitly by their image under an injective linear map $\mathfrak a \to \mathfrak k$ involving the Hessian of $h_{H_0, g}$. When $\dim(\mathfrak a) = \dim(\mathfrak k) = 1$, all of this can be made relatively explicit, but in general it seems very hard to say anything about the $D\xi_{H_0}$. As a replacement for this computation, we will use the change of variables from Lemma~\ref{regularsetKintermsofC}, through which one of the occurrences of $\xi_{H_0}$ simplifies greatly. The second step in the proof is to eliminate the other occurrence of $\xi_{H_0}(g) \in A$, by using the fact that its adjoint action on $\mathfrak a$ is trivial; see Lemma~\ref{criticalityelimxi}.
\end{remark}

We first deduce Proposition~\ref{orbitalintegralboundclosetolevi} using very general principles.

\begin{proof}[Proof of Proposition~\ref{orbitalintegralboundclosetolevi} assuming Proposition~\ref{psinotlocallyconstantk}]
By \eqref{JintegralWeylinvariance} we may assume that $H_{0} \in \mathfrak{a}^{+}$. Proposition~\ref{reductiontokintegral} then reduces the statement to showing that for some $N > 0$ and $\delta > 0$,
\begin{equation}
\label{leviboundreductiontoK}
\int_{K} e^{it \psi_{H_{0}, g}(k)} b''_{H_{0}, g}(k) dk \ll \left( 1 + t \cdot d \left( g, \bigcup_{L \in \mathcal L - \{ G\}} M' L \right)^{N} \right)^{-\delta} \,,
\end{equation}
with uniformity in $H_0$.
When $g$ lies at positive distance from $\bigcup_{L \in \mathcal L - \{ G\}} M' L$, this follows directly from the Van der Corput lemma \cite[\S VIII.2, Proposition 5]{stein1993}, because by Proposition~\ref{psinotlocallyconstantk} there is no $k \in \operatorname{supp}(b''_{H_{0}, g})$ at which $\psi_{H_{0}, g}$ is somewhere constant, meaning that some higher $k$-derivative is bounded away from zero around every $k \in \operatorname{supp}(b''_{H_{0}, g})$.

To get control over the dependence on $g$, consider for every $n \geq 1$ the real analytic map
\begin{align*}
D_n : \mathcal{R}' & \to (\mathfrak k^*)^{\otimes n} \\
(H_{0}, g, k) & \mapsto ((D^{(n)} \psi_{H_{0}, g})_{k}
\end{align*}
where we denote $D^{(n)}$ for the higher derivatives. Call $Z \subset \mathcal R'$ the joint zero locus of the $D_n$.

Proposition~\ref{psinotlocallyconstantk} says that
\[ Z \subset \mathfrak{a} \times \left( \bigcup_{L \in \mathcal L - \{G\}}M' L \right) \times K \,. \]
Locally around a fixed $(H_0, g, k)$ we may embed $\mathcal R'$ as a real submanifold of a complex manifold $\Omega$ on which $\psi_{H_0, g}(k)$ is complex analytic in all three variables. The maps $D_n$ are derivatives of $\psi_{H_0, g}(k)$ and therefore also extend analytically to $\Omega$. By \cite[\S 3.9 Theorem 9 C]{whitney1972}, their joint zero locus is locally the zero locus of only finitely many $D_n$, say of $D_1, \ldots, D_J$, on a neighborhood $U \subset \mathcal R'$ of $(H_0, g, k)$. Consider now the Taylor polynomial map
\[
T_J : U  \to (\mathfrak k^*)^{\otimes 1} \oplus \cdots \oplus (\mathfrak k^*)^{\otimes J}
\]
obtained by pairing $D_1, \ldots, D_J$, and equip the right-hand side with a fixed norm. By Lojasiewicz's inequality, for every compact subset $V \subset U$ there exists $N > 0$ such that for $(H_{0}', g', k') \in V$ we have
\[ \left\lVert T_J(H_{0}', g', k') \right\rVert \gg d \left( (H_{0}, g, k), Z \right)^{N} \gg d \left( g, \bigcup_{L \in \mathcal L - \{G\}}M' L \right)^{N} \,. \]
Because such a bound holds locally in $\mathcal R'$, it holds on compact subsets of $\mathcal R'$ but with possibly bigger values of $J$ and $N$. This gives a lower bound for a $k$-derivative order at most $J$ for $\psi_{H_0, h}$. To apply this, take $D \subset \mathcal R'$ to be any compact neighborhood of $\operatorname{supp}(b'') \cap (D_{\mathfrak{a}} \times D_{G} \times K)$. This is possible by Lemma~\ref{cutoffbsmooth} and because $\mathcal{R}'$ is open in $\mathcal{P}$. The bound \eqref{leviboundreductiontoK} then follows from the Van der Corput lemma \cite[\S VIII.2, Proposition 5]{stein1993}, with the corresponding value of $N$ and with $\delta = 1/J$.
\end{proof}

The following lemma is an immediate application of the chain rule.

\begin{lemma} \label{criticalsetpsiphisame}Let $H_{0} \in \mathfrak{a}^{\operatorname{gen}, +}$ and $g \in G$. When $k \in \mathcal{R}'_{H_{0}, g}$, we have that $k$ is a critical point of $\psi_{H_{0}, g}$ if and only if $(\xi_{H_{0}}(k), \xi_{H_{0}}(kg), k)$ is a critical point of $\phi_{H_{0}, g}$. \hfill $\square$ \end{lemma}

For $g \in G$, define $\Theta_g : K \to K$ by \eqref{definitionTheta}.

\begin{lemma} \label{criticalpointsphi}Let $H_{0} \in \mathfrak{a}^{\operatorname{reg}}$ and $g \in G$. A point $(a_{1}, a_{2}, k)$ is a critical point of $\phi_{H_{0}, g}$ if and only if
\begin{align}
& \kappa(ka_{1}) \in \mathcal{C}(G, H_{0}) \,, \label{orbitalintegralcriticalityHone} \\
& \kappa(kga_{2}) \in \mathcal{C}(G, H_{0}) \,, \label{orbitalintegralcriticalityHtwo} \\ 
& n(\Theta_{a_{1}}(k) a_{1}^{-1} g a_{2}) = 1 \,.\label{orbitalintegralcriticalityK} 
\end{align}
\end{lemma}

\begin{proof}Let $(a_{1}, a_{2}, k)$ be a critical point of $\phi_{H_{0}, g}$. By \eqref{phiintermsofheightfunctions} and Lemma~\ref{extremalpointssummary}, criticality in $a_{1}$ and $a_{2}$ is equivalent to \eqref{orbitalintegralcriticalityHone} and \eqref{orbitalintegralcriticalityHtwo}. In view of \eqref{definitionphi}, $k$ is a critical point of $\phi_{H_{0}, g}(a_{1}, a_{2}, \cdot)$ if and only if $\Theta_{a_{1}}(k)$ is a critical point of $\widetilde{\phi}_{H_{0}, g}(a_{1}, a_{2}, \cdot)$. 
Using the expression \eqref{definitionphinought}, by \cite[Lemma~5.3]{duistermaat1983} this in turn is equivalent to
\[ \Theta_{a_{1}}(k) a_{1}^{-1} g a_{2} \in Z_{N}(H_{0}) A K = AK \,, \]
where we have used that $H_{0} \in \mathfrak{a}^{\operatorname{reg}}$. This is \eqref{orbitalintegralcriticalityK}.
\end{proof}

\begin{remark}The condition \eqref{orbitalintegralcriticalityK} can also be written as $n(ka_{1}) = n(kga_{2})$. Indeed, writing $ka_{1} = n'a'k'$, condition \eqref{orbitalintegralcriticalityK} says that
\[ k'a_{1}^{-1} g a_{2} = a'^{-1} n'^{-1} k g a_{2} \in AK \,. \]
We may write this as $kga_{2} \in n' AK$, which is what we claimed. This is analogous to the criticality condition in \cite[Lemma 7.1]{Marshall2016}, where it can be interpreted geometrically as saying that two geodesics in the upper half plane are concentric.\end{remark}

At this point it is helpful to make the change of variables given by the diffeomorphism in Proposition~\ref{regularsetKintermsofC}. Define therefore
\begin{align}
\label{definitionpsitilde}
\begin{split}
\widetilde{\psi}_{H_{0}, g} : \mathcal{C}(G, H_{0}) \times A & \to \mathbb{R} \\
(c, a) & \mapsto \psi_{H_{0}, g}(\kappa(ca^{-1})) \,.
\end{split}
\end{align}
For $H_{0} \in \mathfrak{a}$ and $g \in G$, define an open subset of $\mathcal{C}(G, H_{0}) \times A$ by
\[
\mathcal{R}''_{H_{0}, g} = \{ (c, a) \in \mathcal{C}(G, H_{0}) \times A : (H_{0}, g, \kappa(ca^{-1})) \in \mathcal{R}' \} \,.
\]
Thanks to this change of variables, $\widetilde{\psi}_{H_{0}, g}$ has a more manageable expression.
Using that $\xi_{H_{0}}(\kappa(ca^{-1})) = a$ and using the definitions \eqref{definitionpsitilde}, \eqref{definitionpsi}, \eqref{definitionphi}, \eqref{definitionphinought} we have for $(c, a) \in \mathcal{R}''_{H_{0}, g}$ that
\begin{align*}
\widetilde{\psi}_{H_{0}, g}(c, a) & = \psi_{H_{0}, g}(\kappa(c a^{-1})) \\
& = \phi_{H_{0}, g}(a, \xi_{H_{0}}(\kappa(c a^{-1})g), \kappa(c a^{-1})) \\
& = \widetilde{\phi}_{H_{0}, g}(a, \xi_{H_{0}}( \kappa( c a^{-1} ) g ) , c) \\
& = \langle H_{0} , H( c a^{-1} g \xi_{ H_{0} }( \kappa( c a^{-1} ) g ) ) \rangle \\
& = \langle H_{0} , H( c a^{-1} g \xi_{ H_{0} }( c a^{-1} g ) ) \rangle \,.
\end{align*}
In the last equality we have used that $\xi_{H_{0}}$ is invariant under left multiplication by $NA$.
For $(c, a) \in \mathcal{R}''_{H_{0}, g}$ define
\begin{equation}
\gamma_{H_{0}, g}(c, a) = c a^{-1} g \xi_{H_{0}}( c a^{-1} g ) \,,
\end{equation}
so that $\widetilde{\psi}_{H_{0}, g}(c, a) = \langle H_{0}, H(\gamma_{H_{0}, g}(c, a)) \rangle$. By Lemma~\ref{extremalpointssummary} and the definition of $\xi_{H_{0}}$ \eqref{definitionxiH} we have for all $(c, a) \in \mathcal{R}''_{H_{0}, g}$ that
\begin{equation} \label{kpartofgammainC}
\kappa(\gamma_{H_{0}, g}(c, a)) \in \mathcal{C}(G, H_{0}) \,.
\end{equation}

\begin{lemma}\label{criticalitonpartofgamma}Let $H_{0} \in \mathfrak{a}^{\operatorname{gen}, +}$ and $g \in G$. Then $(c, a) \in \mathcal{R}''_{H_{0}, g}$ is a critical point of $\widetilde{\psi}_{H_{0}, g}$ if and only if
\[ n(\gamma_{H_{0}, g}(c, a)) = 1 \,. \]
\end{lemma}

\begin{proof}A point $(c, a) \in \mathcal{R}''_{H_{0}, g}$ is a critical point of $\widetilde \psi_{H_0, g}$ if and only if $\kappa(ca^{-1})$ is a critical point of $\psi_{H_0, g}$. Using Lemma~\ref{criticalsetpsiphisame} and Lemma~\ref{criticalpointsphi}, this is seen to be equivalent to $n(\gamma_{H_{0}, g}(c, a)) = 1$.
\end{proof}

\begin{lemma}\label{criticalityelimxi}If $\widetilde \psi_{H_0, g}$ is critical at $(c, a) \in \mathcal R''_{H_0, g}$, then
\[ \langle H_0 ,\operatorname{Ad}_{ca^{-1}g}(\mathfrak a) \rangle = 0 \,.\]
\end{lemma}

\begin{proof}
By Lemma~\ref{criticalitonpartofgamma} and equation \eqref{kpartofgammainC}, criticality implies that $\gamma_{H_0, g}(c, a) \in A \mathcal C(G, H_0)$, so that
\[ ca^{-1}g \in A \mathcal C(G, H_0) A \,. \]
For $k \in \mathcal C(G, H_0)$ we have by definition that $\langle H_0, \operatorname{Ad}_k(\mathfrak a) \rangle = 0$ (see \eqref{definitionCreductive}). By $\operatorname{Ad}_G$-invariance of the Killing form, this gives
\[  \langle H_0, \operatorname{Ad}_{ca^{-1}g}(\mathfrak a) \rangle = 0 \,, \]
as claimed.
\end{proof}

\begin{remark}
\begin{enumerate}
\item In Lemma~\ref{criticalityelimxi} we write $\langle H_0, \cdot \rangle$ and not $H_0 \perp \cdot$, because $\operatorname{Ad}_{ca^{-1}g}(\mathfrak a)$ is not in $\mathfrak p$ in general.
\item Lemma~\ref{criticalityelimxi} gets rid of the (complicated) function $\xi_{H_0, g}$ by considering the adjoint action on $\mathfrak a$. This makes differentiating much easier, at the cost of a (possible) loss of information about $\gamma_{H_0, g}(c, a)$.
\end{enumerate}
\end{remark}

\begin{lemma}\label{lemmaextraorthcomplementC}If $\widetilde \psi_{H_0, g}$ is constant around a point $(c, a) \in \mathcal R''_{H_0, g}$, then
\[ [H_0,  \operatorname{Ad}_c(T_c \mathcal C(G, H_0)) ] \perp E_{\mathfrak p}(\operatorname{Ad}_{c A g}(\mathfrak a)) \,,\]
where $E_{\mathfrak p} = \frac12(\operatorname{id} - \theta)$ denotes the orthogonal projection onto $\mathfrak p$.
\end{lemma}

\begin{proof}
By the locally constant hypothesis implies, the conclusion of Lemma~\ref{criticalityelimxi} is true for all nearby $c$ and nearby $a$. In fact, by analytic continuation it is valid for all $a' \in A$. Differentiating it with respect to $c$ gives
\[ \langle H_0, [\operatorname{Ad}_c(T_c \mathcal C(G, H_0)), \operatorname{Ad}_{ca'g}(\mathfrak a)] \rangle =  0 \,, \]
for all $a' \in A$. Now we use associativity of the Killing form:
\[ \langle [H_0, \operatorname{Ad}_c(T_c \mathcal C(G, H_0))],  \operatorname{Ad}_{ca'g}(\mathfrak a) \rangle =  0 \,. \]
And finally, we use that the left member lies in $[\mathfrak p, \mathfrak k] = \mathfrak p$.
\end{proof}

\begin{lemma}\label{tangentspaceCinvariance}For $H_0 \in \mathfrak a^{\operatorname{gen}}$ we have that
\[\mathfrak m + \operatorname{Ad}_c(\mathfrak m) \subset \operatorname{Ad}_c(T_c\mathcal C(G, H_0)) \,.\]
\end{lemma}

\begin{proof}
This is easily seen from Lemma~\ref{tangentspaceC}. Associativity of the Killing form gives that $\mathfrak m \perp [H_0, \mathfrak g]$, which implies that $\mathfrak m \subset \operatorname{Ad}_c(T_c\mathcal C(G, H_0))$. Associativity also gives that $\operatorname{Ad}_c(\mathfrak m) \perp [\operatorname{Ad}_c(\mathfrak a), \mathfrak g]$, which implies $\operatorname{Ad}_c(\mathfrak m) \subset \operatorname{Ad}_c(T_c\mathcal C(G, H_0))$.
\end{proof}

Define $\mathfrak k^{\perp \mathfrak m} = \mathfrak k \cap \bigoplus_\alpha \mathfrak g_\alpha$ and $\mathfrak p^{\perp \mathfrak a} = \mathfrak p \cap \bigoplus_\alpha \mathfrak g_\alpha$. They are the ``root space parts'' of $\mathfrak k$ and $\mathfrak p$ respectively. Define $E_{\mathfrak k^{\perp \mathfrak m}}$ and $E_{\mathfrak p^{\perp \mathfrak a}}$ to be the orthogonal projections onto these spaces.

The following is nothing but a useful reformulation of Lemma~\ref{tangentspaceC}.

\begin{lemma}\label{orthcomplementCHnought}Let $H_0 \in \mathfrak a^{\operatorname{gen}}$ and $c \in \mathcal C(G, H_0)$. The orthogonal complement of the subspace $[H_0, \operatorname{Ad}_c(T_c \mathcal C(G, H_0))]$ in $\mathfrak p$ equals $\mathfrak a \oplus \operatorname{Ad}_c(\mathfrak a)$.\end{lemma}

\begin{proof}
Call $V$ the orthogonal complement of $[H_0, \operatorname{Ad}_c(T_c \mathcal C(G, H_0))]$ in $\mathfrak p$. Certainly $\mathfrak a \subset V$, because $\mathfrak a \perp [\mathfrak a, \mathfrak k]$.
By Lemma~\ref{tangentspaceC}, the space
\[ \operatorname{Ad}_c(T_c \mathcal C(G, H_0)) \subset \mathfrak k \]
is the orthogonal complement of $[H_0, \operatorname{Ad}_c(\mathfrak a)]$. It follows that $\operatorname{Ad}_c(\mathfrak a) \subset V$. Therefore
\[ \mathfrak a \oplus \operatorname{Ad}_c(\mathfrak a) \subset V \,. \]
That this is an equality, follows from dimension considerations. Consider the projection $E_{\mathfrak k^{\perp \mathfrak m}}(\operatorname{Ad}_c(T_c \mathcal C(G, H_0)))$ to $\mathfrak k^{\perp \mathfrak m}$. Because $\operatorname{Ad}_c(T_c \mathcal C(G, H_0))$ has codimension $\dim(\mathfrak a)$ in $\mathfrak k$ by Lemma~\ref{genericcriticalsetstructure}, and because $\operatorname{Ad}_c(T_c \mathcal C(G, H_0)) \supset \mathfrak m$ by Lemma~\ref{tangentspaceCinvariance}, this projection has codimension $\dim(\mathfrak a)$ in $\mathfrak k^{\perp \mathfrak m}$. Now the action of $[H_0, \cdot]$ on $\mathfrak k^{\perp \mathfrak m}$ gives a bijection with $\mathfrak p^{\perp \mathfrak a}$ because $H_0$ is regular. Therefore $[H_0, \operatorname{Ad}_c(T_c \mathcal C(G, H_0))]$ has codimension $\dim(\mathfrak a)$ in $\mathfrak p^{\perp \mathfrak a}$, and the orthogonal complement in this space must be $E_{\mathfrak p^{\perp \mathfrak a}}(\operatorname{Ad}_c(\mathfrak a))$, which has the correct dimension because $\operatorname{Ad}_c(\mathfrak a)$ intersects trivially with $\mathfrak a$ by Lemma~\ref{CdoesnotmeetMLarticle} and Lemma~\ref{truelemma}.
\end{proof}

The idea of the argument is the following. Lemma~\ref{orthcomplementCHnought} gives the precise orthogonal complement of $[H_0, \operatorname{Ad}_c(T_c \mathcal C(G, H_0))]$ in $\mathfrak p$, and Lemma~\ref{lemmaextraorthcomplementC} also gives elements in $\mathfrak p$ orthogonal to it, under the assumption that $\widetilde \psi_{H_0, g}$ is constant around a point $(c, a) \in \mathcal R''_{H_0, g}$. As soon as the latter lemma gives one new element, we have a contradiction. We may also formulate this in the following way.

\begin{lemma}\label{inclusionorthcomplementconclusion}If $\widetilde \psi_{H_0, g}$ is constant around a point $(c, a) \in \mathcal R''_{H_0, g}$, then
\[ E_{\mathfrak p}(\operatorname{Ad}_{c A g}(\mathfrak a)) \subset \mathfrak a \oplus \operatorname{Ad}_c(\mathfrak a) \,. \]
\end{lemma}

\begin{proof}
Immediate from Lemma~\ref{lemmaextraorthcomplementC} and Lemma~\ref{orthcomplementCHnought}.
\end{proof}

\begin{lemma}\label{corinclusionridiculouslystrong}If $\widetilde \psi_{H_0, g}$ is constant around a point $(c, a) \in \mathcal R''_{H_0, g}$, then
\[ E_{\mathfrak p^{\perp \mathfrak a}}(\operatorname{Ad}_{A g}(\mathfrak a)) \subset E_{\mathfrak p^{\perp \mathfrak a}}( \operatorname{Ad}_{c^{-1}}(\mathfrak a) )\,. \]
\end{lemma}

\begin{proof}We use Lemma~\ref{inclusionorthcomplementconclusion} and apply $\operatorname{Ad}_{c^{-1}}$, which gives that
\[  E_{\mathfrak p}(\operatorname{Ad}_{A g}(\mathfrak a)) \subset \operatorname{Ad}_{c^{-1}}(\mathfrak a) \oplus \mathfrak a\,. \]
Here we have used that $\operatorname{Ad}_K$ commutes with $E_{\mathfrak p}$. Projecting further to $\mathfrak p^{\perp \mathfrak a}$ we may discard the factor $\mathfrak a$ in the right-hand side, and the claim follows.
\end{proof}

If we can show that the left-hand side in Lemma~\ref{corinclusionridiculouslystrong} is of dimension $\dim(\mathfrak a)$, then the inclusion becomes an equality; in particular, an inclusion in the other direction. The following lemmas show that it is indeed of that dimension, provided that $g$ is generic.

\begin{lemma}\label{AdAspansrootspaces}Let $X = (X_\alpha) \in \bigoplus_{\alpha \in \Sigma}\mathfrak g_\alpha$. The span of $\operatorname{Ad}_A(X)$ is $\bigoplus_{\alpha \in \Sigma} \mathbb R X_\alpha$.
\end{lemma}

\begin{proof}Clearly $\operatorname{Ad}_A(X) \subset \bigoplus_{\alpha \in \Sigma} \mathbb R X_\alpha$.

We must show that the vectors $v_a = (\alpha (a))_{\alpha \in \Sigma} \in \mathbb R^{\Sigma}$ with $a \in A$ span the entire space $\mathbb R^{\Sigma}$. This is equivalent to the statement that the roots of $A$ are linearly independent, and it is a general fact that the characters of an abelian group are linearly independent. More precisely, suppose they don't span everything, then they lie on a hyperplane and there exist $c_\alpha \in \mathbb R$, not all zero such that
\[ \forall a \in A : \sum_{\alpha \in \Sigma} c_\alpha \alpha(a) = 0 \,. \]
This is saying that the $\alpha : A \to \mathbb R^\times$ are linearly dependent.
\end{proof}

\begin{lemma}\label{pperpaprojectionbiggeneric}If $g \notin \bigcup_{L \in \mathcal L - \{ G\}}M' L$, then
\[\dim(\operatorname{span}(E_{\mathfrak p^{\perp \mathfrak a}}(\operatorname{Ad}_{A g}(\mathfrak a)))) \geq \dim(\mathfrak a) \,.\]
\end{lemma}

\begin{proof}
Let $V$ be the span of $\operatorname{Ad}_{Ag}(\mathfrak a)$. By Lemma~\ref{AdAspansrootspaces}, it is of the form $V_0 \oplus \bigoplus_{\alpha \in \Sigma} V_\alpha$ with $V_0 \subset \mathfrak a + \mathfrak m$ and $V_\alpha \subset \mathfrak g_{\alpha}$. Let $S = \{ \alpha : V_\alpha \neq 0 \}$. We claim that $S$ spans $\mathfrak a^*$. Suppose not, then there exists a nonzero $H \in \bigcap_{\alpha \in S} \ker(\alpha)$, meaning that $[H, V] = 0$. Then $[H, \operatorname{Ad}_g(\mathfrak a)] = 0$, or equivalently, $[\operatorname{Ad}_{g^{-1}}(H), \mathfrak a] = 0$. Thus $\operatorname{Ad}_{g^{-1}}(H) \in \mathfrak m \oplus \mathfrak a$. By Lemma~\ref{truelemma} this implies that $g \in Z_G(H) M'$, contradicting our hypothesis on $g$. Therefore $S$ spans $\mathfrak a^*$.

In particular, there are at least $\dim(\mathfrak a)$ positive roots $\alpha \in \Sigma^+$ with $V_\alpha \neq 0$, and in particular with $E_{\mathfrak p^{\perp \mathfrak a}}(V_\alpha) \neq 0$. Because the projections $E_{\mathfrak p^{\perp \mathfrak a}}(\mathfrak g_\alpha)$ for different $\alpha \in \Sigma^+$ are orthogonal, the projection $E_{\mathfrak p^{\perp \mathfrak a}}(V)$ has dimension at least $\dim(\mathfrak a)$.
\end{proof}

\begin{lemma}\label{lemmarootspacerestriction}Suppose $\widetilde \psi_{H_0, g}$ is constant around a point $(c, a) \in \mathcal R''_{H_0, g}$, and suppose $g \notin \bigcup_{L \in \mathcal L - \{ G\}}M' L$. Then
\[ \operatorname{span}(E_{\mathfrak p^{\perp \mathfrak a}}(\operatorname{Ad}_{Ag}(\mathfrak a))) = E_{\mathfrak p^{\perp \mathfrak a}}( \operatorname{Ad}_{c^{-1}}(\mathfrak a) ) \,. \]
\end{lemma}

\begin{proof}Lemma~\ref{corinclusionridiculouslystrong} says that the left-hand side is contained in the right-hand side. The left-hand side is of dimension at least $\dim(\mathfrak a)$ by Lemma~\ref{pperpaprojectionbiggeneric}, and the right-hand side is of dimension $\dim(\mathfrak a)$ because $\operatorname{Ad}_c(\mathfrak a)$ intersects trivially with $\mathfrak a$ by Lemma~\ref{CdoesnotmeetMLarticle} and Lemma~\ref{truelemma}. Therefore, equality must hold.
\end{proof}

\begin{proof}[Proof of Proposition~\ref{psinotlocallyconstantk}] Assume $\psi_{H_0, g}$ is constant around some point $k \in \mathcal R'_{H_0, g}$. Then $\widetilde \psi_{H_0, g}$ (defined in \eqref{definitionpsitilde}) is constant around some point $(c, a)\in \mathcal R''_{H_0, g}$.

Because $\operatorname{Ad}_{c^{-1}}(H_0) \perp \mathfrak a$ (by definition of $\mathcal C(G, H_0)$), Lemma~\ref{lemmarootspacerestriction} gives that
\[ \operatorname{Ad}_{c^{-1}}(H_0) \in \operatorname{span}(E_{\mathfrak p^{\perp \mathfrak a}}(\operatorname{Ad}_{Ag}(\mathfrak a))) \,. \]
On the other hand, by the hypothesis of being locally constant and by analytic continuation in $a$, Lemma~\ref{criticalityelimxi} says that
\[ \langle \operatorname{Ad}_{c^{-1}}(H_0) , \operatorname{Ad}_{Ag}(\mathfrak a) \rangle = 0\,,  \]
so that
\[ \operatorname{Ad}_{c^{-1}}(H_0) \perp \operatorname{span}(E_{\mathfrak p^{\perp \mathfrak a}}(\operatorname{Ad}_{Ag}(\mathfrak a))) \,.\]
Therefore $\operatorname{Ad}_{c^{-1}}(H_0) = 0$, which is a contradiction.
\end{proof}

\begin{remark}When $g \in M' L$ with $L \in \mathcal L$ chosen to be minimal, an argument analogous to that used to prove Lemma~\ref{pperpaprojectionbiggeneric} shows that
\[ \dim(\operatorname{span}(E_{\mathfrak p^{\perp \mathfrak a}}(\operatorname{Ad}_{A g}(\mathfrak a)))) \geq \dim(\mathfrak a^L) \,. \]
Moreover, the span in the left-hand side is a direct sum of subspaces of the spaces $\mathfrak p \cap (\mathfrak g_{\alpha} + \mathfrak g_{-\alpha})$. If $\widetilde \psi_{H_0, g}$ is constant around a point $(c, a)\in \mathcal R''_{H_0, g}$, this gives a very strong restriction on the space $E_{\mathfrak p^{\perp \mathfrak a}}( \operatorname{Ad}_{c^{-1}}(\mathfrak a) )$ (which contains this span), but we have not been able to obtain a contradiction from this, even when exploiting the fact that the inclusion is true for all nearby $c' \in \mathcal C(G, H_0)$.
\end{remark}

\section{Preliminaries on algebraic groups}

\label{alggrouppreliminaries}

\subsection{Algebraic groups}

\label{algebraicgroups}

Let $\mathbf G / \mathbb Q$ be a linear connected anisotropic algebraic group which is almost simple over $\mathbb R$. Let $\mathbf H \subset \mathbf G$ be a maximal torus. The base fields of groups are indicated by subscripts if they are not clear from the context. We use the same subscript notation for base change and when $E/F$ is a finite field extension we denote $\operatorname{Res}_{E/F}$ for Weil restriction.

Fix a closed embedding $\rho : \mathbf G \to \mathbf{SL}_d$ for some $d > 0$. Equipping $\mathbf{SL}_d$ with the standard schematic structure and taking the schematic closures of $\rho(\mathbf G)$ and $\rho(\mathbf H)$ we obtain integral models that we continue to denote by $\mathbf G$ and $\mathbf H$.

\begin{remark}Eventually $\mathbf G$ and $\mathbf H$ will be as in Theorem~\ref{largevaluesslthree}. That is, $\mathbf G$ an anisotropic form of $\mathbf{PGL}_{3}$ and $\mathbf H$ maximal split over $\mathbb R$. But we allow some generality in order to be able to make remarks about other groups, and because for most results there is no reason to be restrictive. The condition that $\mathbf G$ is almost simple over $\mathbb R$ will only be used to apply a Weyl law-type result from \cite{brumley2020} (see Proposition~\ref{brumarweyllaw}) but is of course automatic for forms of $\mathbf{PGL}_{3, \overline{\mathbb Q}}$.\end{remark}

\subsection{Forms of \texorpdfstring{$\mathbf{PGL}_3$}{PGL3}}

\label{sectionformspgl}

By a form of a group $\mathbf G$ over a field $E$ we shall mean a form of the base change to the algebraic closure, $\mathbf G_{\overline E}$.
The following proposition gives a concrete list of all groups $\mathbf G$ that are allowed in Theorem~\ref{largevaluesslthree}. Proposition~\ref{innerformsrealpoints} below says what the corresponding Lie groups are. 

\begin{proposition}\label{localglthreeexamples} The $\mathbb Q$-forms $\mathbf G$ of $\mathbf{PGL}_{3}$, and the anisotropic ones among them, are given by the following constructions.
\begin{enumerate}
\item (Inner forms) Let $D/\mathbb Q$ be a central simple algebra of degree $3$ and define $\mathbf G = \mathbf{GL}_{1, D} / \mathbb G_m$. Then $\mathbf G$ is anisotropic if and only if $D$ is a division algebra.
\item (Outer forms) Let $F/\mathbb Q$ be a quadratic field, $(V,q)$ a $3$-dimensional nondegenerate Hermitian space over $F$ and define $\mathbf G = \mathbf{PU}(V)$. Then $\mathbf G$ is anisotropic if and only if $(V, q)$ is. 
\item (Outer forms) Let $F/\mathbb Q$ be a quadratic field, $D/F$ be a central division algebra of degree $3$ over $F$ with an anti-involution $\tau \in \operatorname{End}_{\mathbb Q}(D)$ that acts nontrivially on $F$, $(V, q)$ a $1$-dimensional Hermitian space over $D$ with $q \neq 0$, and define $\mathbf G = \mathbf{PU}(V)$. Concretely, there is a nonzero $d \in D$ fixed by $\tau$ such that $\mathbf G(\mathbb Q) = \{g \in D^\times : g d \tau(g) = d \}/F^1$, where $F^1 \subset F^\times$ denotes the norm $1$ subgroup. Then $\mathbf G$ is anisotropic. 
\end{enumerate}
\end{proposition}

\begin{proof}
That this list is exhaustive can be proven as in \cite[\S18.4]{morris2015}, or can be deduced from \cite[Theorem 18.4.1]{morris2015} by using that $\mathbf{PGL}_{n, \overline{\mathbb Q}}$ and $\mathbf{SL}_{n, \overline{\mathbb Q}}$ have the same automorphism groups. The statements about anisotropy can be seen as follows:
\begin{enumerate}
\item By Wedderburn's theorem, $D$ is either a division algebra or isomorphic to $M_3(\mathbb Q)$. When $D$ is a matrix algebra, clearly $\mathbf G$ is split. Conversely, let $D$ be a division algebra, $\mathbf T \subset \mathbf G$ be a maximal torus and $\chi$ a character of $\mathbf T$. We may lift $\mathbf T$ to a maximal torus $\mathbf T' \subset \mathbf{GL}_{1, D}$ \cite[Proposition 10.6]{borel1965}. It is well known that such tori are obtained from the multiplicative groups of étale subalgebras of $D$. Because $D$ is division, $\mathbf T' = \operatorname{Res}_{F/\mathbb Q} \mathbb G_{m, F}$ for a field extension $F/ \mathbb Q$, and it follows that $\chi$ is a power of the norm character. Because $\chi$ is trivial on $\mathbb G_{m, \mathbb Q}$, it is trivial altogether. Therefore $\mathbf T$ is anisotropic. 
\item If $q$ represents $0$, then $V$ contains a hyperbolic plane $H$, and we find a rank $1$ split torus in $\mathbf U(V)$, necessarily noncentral, that acts diagonally on $H$ in an isotropic basis and acts trivially on $H^\perp$. Conversely, if $\mathbf T \subset \mathbf G$ is a split torus then we may lift it to a torus of $\mathbf U(V)$ by \cite[Proposition 10.6]{borel1965}, which contains a nontrivial split subtorus $\mathbf T'$ by \cite[\S II.8.14]{borel1991}. A vector $v \in V$ in a weight space with respect to $\mathbf T'$ for a nonzero weight then satisfies $q(v) = 0$. (Such a vector exists because the split torus $\mathbf T' \subset \operatorname{Res}_{F/\mathbb Q} \mathbf{GL}(V) \subset \mathbf{GL}_{6}$ is diagonalizable.)
\item This can be proven as in the second case: A split torus in $\mathbf G$ would give rise to a nonzero vector $v \in V \cong D$ satisfying $vd\tau(v) = 0$. Because $D$ is a division algebra, this is impossible. \qedhere
\end{enumerate}
\end{proof}

\begin{proposition}\label{innerformsrealpoints}Let $\mathbf G$ be as in Proposition~\ref{localglthreeexamples}. With the same numbering, its group of real points is as follows:
\begin{enumerate}
\item In the first case, $\mathbf G(\mathbb R) \cong \operatorname{PGL}_3(\mathbb R)$.
\item In the second case: If $F$ is real quadratic we have $\mathbf G(\mathbb R) \cong \operatorname{PGL}_3(\mathbb R)$. If $F$ is imaginary quadratic we have $\mathbf G(\mathbb R) \cong \operatorname{PU}(3)$ or $\operatorname{PU}(2, 1)$ depending on the signature of $(V, q)$ over $\mathbb R$.
\item In the third case, $\mathbf G(\mathbb R) \cong \operatorname{PGL}_3(\mathbb R)$ when $F$ is real quadratic and $\mathbf G(\mathbb R) \cong\operatorname{PU}(3)$ or $\operatorname{PU}(2, 1)$ otherwise.
\end{enumerate}
\end{proposition}

\begin{proof}
\begin{enumerate}
\item By a classical theorem of Frobenius there is no degree $3$ division algebra over $\mathbb R$, so that $D \otimes_{\mathbb Q} \mathbb R \cong M_3(\mathbb R)$ by Wedderburn's theorem. Therefore $\mathbf G(\mathbb R) \cong \operatorname{PGL}_3(\mathbb R)$.
\item When $F$ is real quadratic we have $F \otimes_{\mathbb Q} \mathbb R \cong \mathbb R \times \mathbb R$, and the induced automorphism permutes the two factors. Using this isomorphism we have $\mathbf{GL}(V)_{\mathbb R} \cong \operatorname{GL}_3(\mathbb R) \times \operatorname{GL}_3(\mathbb R)$. The hermitian matrix defining $q$ takes the form $(A, A^T)$ under this isomorphism, and $A \in \operatorname{GL}_3(\mathbb R)$ can be assumed diagonal. Now $\mathbf U(V)(\mathbb R) \cong \{(B, C) \in \operatorname{GL}_3(\mathbb R)^2 : BAC^T = CA^TB^T = A\}$. It is clear that for every $B \in \operatorname{GL}_3(\mathbb R)$ there is a unique element $(B, C) \in\mathbf U(V)(\mathbb R)$, and this gives an isomorphism $\mathbf U(V)(\mathbb R) \cong \operatorname{GL}_3(\mathbb R)$. The conclusion follows.

When $F$ is imaginary we have $F \otimes_{\mathbb Q} \mathbb R \cong \mathbb C$ and it follows that $\mathbf U(V)(\mathbb R)$ is a classical unitary group, necessarily of signature $(3, 0)$ or $(2, 1)$.
\item This can be shown using similar arguments as in the first and second case. When $F$ is real quadratic, using the isomorphism $F \otimes_{\mathbb Q} \mathbb R \cong \mathbb R \times \mathbb R$ we have $D \otimes_{\mathbb Q} \mathbb R \cong M_3(\mathbb R)^2$. Call $\sigma$ the anti-involution of $M_3(\mathbb R)^2$ that swaps the two factors and transposes them. Then $\tau \circ \sigma$ is an automorphism of $M_3(\mathbb R)^2$. (In fact, it preserves the center $\mathbb R  \times \mathbb R$, therefore preserves both factors and is inner by Skolem-Noether.) Applying an automorphism of $M_3(\mathbb R)^3$ we may assume that $\sigma = \tau$. A similar argument as in the second case now gives an isomorphism $\mathbf U(V)(\mathbb R) \cong \operatorname{GL}_3(\mathbb R)$.

When $F$  is imaginary we have $F \otimes_{\mathbb Q} \mathbb R \cong \mathbb C$ and $D \otimes_{\mathbb Q} \mathbb R \cong M_3(\mathbb C)$. By Skolem-Noether, we may assume that the automorphism induced by $\tau$ is conjugate transpose, subsequently that $d$ is diagonal, and it follows that $\mathbf U(V)(\mathbb R)$ is a classical unitary group. \qedhere
\end{enumerate}
\end{proof}

We could state a result similar to Proposition~\ref{innerformsrealpoints} with $\mathbb R$ replaced by any $\mathbb Q_p$: We would still distinguish cases based on whether $p$ is split or inert in $F$, together with the ramified case. The main difference is that there exist degree $3$ division algebras over $\mathbb Q_p$ \cite[\S17.10]{Pierce1982} (and exactly two of them), which would be reflected in the classification. However, the only result that we will require is the following.

\begin{lemma}\label{splitimpliesglthree}Let $\mathbf G$ be a $\mathbb Q$-form of $\mathbf{PGL}_{3}$ and suppose $p$ is a prime such that $G_{\mathbb Q_p}$ is split. Then $G_{\mathbb Q_p} \cong \mathbf{PGL}_{3, \mathbb Q_p}$.\end{lemma}

\begin{proof}Both $G_{\mathbb Q_p}$ and $\mathbf{PGL}_{3, \mathbb Q_p}$ are split forms of $\mathbf{PGL}_{3, \overline{\mathbb Q_p}}$. Because a reductive group has only one split form over any field \cite[Théorème 2.13]{borel1965}, they are isomorphic.
\end{proof}

\subsection{Locally symmetric spaces}

\label{arithmeticlocsymmspaces}

Because $\mathbf G$ is connected, the Lie group $G = \mathbf G(\mathbb R)$ is as in \S\ref{notationliegroups}, and the identity component $G^0$ is semisimple. (Recall that we require a semisimple group by definition to be connected.) Let $K_\infty \subset G$ be a maximal compact subgroup. To $G$ we associate data as in \S\ref{preliminariesLiegroups} with respect to the maximal compact $K_\infty$: The Cartan involution is denoted by $\theta$, the Iwasawa decomposition by $NAK_\infty$ and the restricted root system by $(\mathfrak a^*, \Sigma)$. (In \S\ref{sectruncatedhecke} we shall use different notation for the roots of a maximal split torus over an algebraic closure, but that will be only relevant to define Hecke algebras.) We let $d(\cdot, \cdot)$ be a distance function as in \S\ref{defdistancegeneralgroup}.

\begin{remark}\label{remarkcartannotrational}Contrary to \cite{brumley2020} we do not, and cannot, assume that the involution $\theta$ is defined over $\mathbb Q$. Indeed, when $\mathbf G$ arises from a degree $3$ division algebra over $\mathbb Q$ (see \S\ref{sectionformspgl}), such $\theta$ would have to come from an anti-involution of the division algebra, and that does not exist \cite[Corollary 2.8]{Knus1998}. This remark will only play a role in Proposition~\ref{brumarweyllaw}.\end{remark}

Denote by $\mathbb A_{\mathbb Q}$ the ring of adeles over $\mathbb Q$ and by $\mathbb A_f$ the finite adeles. When $S$ is any set of places of $\mathbb Q$, define $\mathbb A_{S}$ to be the restricted product $\prod_{v \in S} \mathbb A_{\mathbb Q}$. Equip the groups $\mathbf G(\mathbb Q_p)$, $\mathbf G(\mathbb A_{\mathbb Q})$, $\mathbf G(\mathbb A_{S})$ with their natural topologies, and similarly for $\mathbf H$. Let $\rho : \mathbf G \to \mathbf{SL}_d$ and the integral models be as in \S\ref{algebraicgroups}. Choose compact subgroups $K_p \subset \mathbf G(\mathbb Q_p)$ for all primes $p$ with the following properties:
\begin{itemize}
\item $\rho(K_p) \subset \mathbf {SL}_d(\mathbb Z_p)$;
\item the compact subgroup
\[ K_f := \prod_p K_p \]
is open in $\mathbf G(\mathbb A_f)$;
\item the center $\mathbf Z(\mathbf G)(\mathbb Q)$ intersects $K_f$ trivially;
\end{itemize}
The first two conditions can be satisfied for example by taking $K_p = \mathbf G(\mathbb Z_p)$, and the third can be obtained by replacing any single $K_p$ by a smaller subgroup.
Define $K = K_\infty \times K_f \subset \mathbf G(\mathbb A_{\mathbb Q})$, and for any set $S$ of places of $\mathbb Q$ define $K_S = \prod_{v \in S} K_v$. Define the automorphic spaces
\begin{align*}
[\mathbf G] = \mathbf G(\mathbb Q) \backslash \mathbf G(\mathbb A_{\mathbb Q}) \quad \text{and}\quad
[\mathbf H] = \mathbf H(\mathbb Q) \backslash \mathbf H(\mathbb A_{\mathbb Q}) \subset [\mathbf G]
\end{align*}
and the locally symmetric space
\[ X = [\mathbf G] / K \,. \]
It is a finite disjoint union of locally symmetric spaces as in \S\ref{classicallocallysymmetricdefs} \cite[Theorem 5.1]{Platonov1991} and is compact because $\mathbf G$ is anisotropic \cite{Borel1962}. Explicitly, let $(g_i)_{i \in I} \subset \mathbf G(\mathbb A_f)$ be a system of representatives for the double quotient $[\mathbf G] / (G \times K_f)$ and define $\Gamma_i = \mathbf G(\mathbb Q) \cap g_i K_f g_i^{-1}$. Then
\begin{equation}
\label{locsymspacedisjointunion}
[\mathbf G] = \bigsqcup_{i \in I} \Gamma_i \backslash (G \times g_i K_f)
\end{equation}
and $X$ is the disjoint union of the compact locally symmetric spaces
\begin{equation}
\label{adeliclocsymspaceconcrete}
X_i = \Gamma_i \backslash G / K_\infty\,.
\end{equation}
(Here we view $\mathbf G(\mathbb Q)$ inside $\mathbf G(\mathbb A_{f})$ when taking the intersection with $g_i K_f g_i^{-1}$, inside $\mathbf G(A_{\mathbb Q})$ in \eqref{locsymspacedisjointunion} and inside $G$ in \eqref{adeliclocsymspaceconcrete}. Notice that the lattices $\Gamma_i$ do not depend on the choice of representatives $g_i$.

\subsection{Compact torus orbits}

\label{sectorusorbitscorrespondence}

We explain now how to relate tori to maximal flat submanifolds of $X$. Let $A$ as in \S\ref{arithmeticlocsymmspaces} be the group given by the choice of Iwasawa decomposition.


Assume $\mathbf H \subset \mathbf G$ has maximal split rank over $\mathbb R$. Then there exists $g_\infty \in G$ such that $\mathbf H(\mathbb R)$ contains $g_\infty A g_\infty^{-1}$, so that $\mathbf H(\mathbb R) \subset Z_G(g_\infty A g_\infty^{-1}) = g_\infty MA g_\infty^{-1}$. To $\mathbf H$ we then associate the compact maximal flat submanifold $\mathscr F \subset X$ that is the image of $\mathbf H(\mathbb A_{\mathbb Q}) g_\infty$. Relative to the decomposition \eqref{adeliclocsymspaceconcrete}, it is the disjoint union of some (but not necessarily all) of the maximal flats $ g_\infty A \subset X_i$. More precisely, $\mathscr F$ intersects $X_i$ nontrivially if and only if $\mathbf H(\mathbb A_{f})$ intersects $g_i K_f$ nontrivially. Note that $\mathscr F$ does not depend on the choice of $g_\infty$, because $g_\infty$ is determined up to multiplication on the right by $N_G(A) = M'A$. Likewise, the definition does not depend on the choice of $A$, which can be seen by using that two choices of $A$ are conjugate by an element of $K_\infty$.

Conversely, one can show that every compact maximal flat in $X_i$ is obtained in this way. Assume $\mathscr F \subset X_i$ is a compact maximal flat submanifold. As in \S\ref{propsmaximalflatssection} there exists $g_\infty \in G$ such that $\mathscr F$ is the image of $g_\infty A$ in $X_i$. The lattice $\Gamma_i$ intersects $g_\infty A g_\infty^{-1}$ in a lattice, and define $\mathbf T$ to be the Zariski closure of $\Gamma_i \cap g_\infty A g_\infty^{-1}$. Then $\mathbf T$ is a torus that is maximal split over $\mathbb R$, and we may take $\mathbf H \subset \mathbf G$ to be any maximal torus containing $\mathbf T$.

\begin{remark}The correspondence between tori and maximal flats that we outlined above, is slightly awkward because $X$ can have different components and because $\mathbf H$ (in our notation) is always a maximal torus, not a maximal $\mathbb R$-split torus. When $X$ is connected, one does obtain a clean bijection between compact maximal flats $\mathscr F \subset X$ and maximal $\mathbb R$-split tori $\mathbf T \subset \mathbf G$. Compare also \cite[\S2]{Einsiedler2009}.
\end{remark}

\subsection{Integration}
\label{sectionadelicmeasures}
For every prime $p$ we choose the Haar measures $d \mu_{\mathbf G, p}$ and $d \mu_{\mathbf H, p}$ on $\mathbf G(\mathbb Q_p)$ and $\mathbf H(\mathbb Q_p)$ respectively, for which $K_p$ and $K_p \cap \mathbf H(\mathbb Q_p)$ have volume $1$. We choose Haar measures on Lie groups as in \S\ref{measuresliegroupcase}. The abelian Lie group $\mathbf H(\mathbb R)$ factors as a conjugate of $A$ times a compact group, and we equip it with the measure that is the product of the measure coming from $A$ with the volume $1$ measure on the compact group. We form the product measures $d\mu_{\mathbf G} = \prod_{v} d \mu_{\mathbf G, v}$ on $\mathbf G(\mathbb A_{\mathbb Q})$ and likewise the product measure $d\mu_{\mathbf H}$ on $\mathbf H(\mathbb A_{\mathbb Q})$. When $S$ is any set of places of $\mathbb Q$, we similarly define measures $\mu_{\mathbf G, S}$ and $\mu_{\mathbf H, S}$ on $\mathbf G(\mathbb A_S)$ and $\mathbf H(\mathbb A_S)$ respectively.

\subsection{Hecke algebras}

\label{secheckealgebra}

A smooth function on $\mathbf G(\mathbb Q_p)$ is defined to be a locally constant one. When $S$ is a set of places of $\mathbb Q$, a compactly supported Schwartz-Bruhat function on $\mathbf G(\mathbb A_{S})$ is a finite linear combination of functions of the form
\[ \prod_{v \in S} f_v\,, \]
where $f_v \in C_c^\infty(\mathbf G(\mathbb Q_v))$ and $f_p = 1_{K_p}$ for almost all primes $p \in S$. We denote by $C_c^\infty(\mathbf G(\mathbb A_{S}))$ the set of such (complex-valued) functions. If $f \in C_c^\infty(\mathbf G(\mathbb A_{S}))$, we define an operator $\pi(f)$ on $L^2([\mathbf G])$ by the rule
\[  (\pi(f)\phi)(x) = \int_{\mathbf G(\mathbb A_{S})} \phi(xg) f(g) d \mu_{\mathbf G, S}(g) \,. \]
If we define $f^*(g) = \overline{f(g^{-1})}$, then $\pi(f)$ and $\pi(f^*)$ are adjoints.  For every set $S$ of places of $\mathbb Q$, define the algebra $\mathcal H_S$ of compactly supported smooth bi-$K_S$ invariant functions on $\mathbf G(\mathbb A_S)$, with convolution given by
\[ (f_1 * f_2)(g) = \int_{\mathbf G(\mathbb A_S)} f(gh^{-1}) f_2(h) d\mu_{\mathbf G, S}(h)\,. \]
The algebras $\mathcal H_S$ act on $L^2(X)$ by the same rule as above. When $p$ is a prime, the element $1_{K_p} \in \mathcal H_p$ is the identity. 

By \cite[\S3.9]{tits1979}, there exists an integer $D > 0$ with the property that for $p \nmid D$, $K_p$ is the compact subgroup associated with a hyperspecial point of the building of $\mathbf G_{\mathbb Q_p}$. By \cite[\S3.3.3]{tits1979} this implies that the Hecke algebra $\mathcal H_p$ is commutative when $p \nmid D$.

Define $\mathcal H_f$ to be the algebra $\mathcal H_S$ when $S$ consists of all finite places that do not divide $D$, and $\mathcal H = \mathcal H_\infty \otimes \mathcal H_f$. It is commutative by the above fact for finite places, and by \cite[Theorem IV.3.1]{Helgason1984} for the infinite place.

\subsection{Truncated Hecke algebras}

\label{sectruncatedhecke}

We will use the notion of truncated Hecke algebras that is also used in \cite{brumley2020}, which makes it more convenient to formulate bounds for orbital integrals. While we will mostly be interested in primes at which all data is split, we do provide the general definition. Let $\mathbf T$ be any maximal split torus of $\mathbf G_{\overline{\mathbb Q}}$, $X^*(\mathbf T)$ and $X_*(\mathbf T)$ be the groups of characters and cocharacters of $\mathbf T$. Let $\Delta$ the set of roots of $\mathbf T$ in $\mathbf G_{\overline{\mathbb Q}}$, $\Delta^+$ be a choice of positive roots and $W$ the Weyl group with respect to $\mathbf T$. Let $\rho \in X^*(\mathbf T) \otimes_{\mathbb Z} \mathbb Q$ be the half-sum of positive roots and denote the natural paring $X^*(\mathbf T) \times X_*(\mathbf T) \to \mathbb Z$ by $\langle \cdot, \cdot \rangle$. Define a norm on $X_*(\mathbf T)$ by
\[ \lVert \mu \rVert = \max_{w \in W } \langle w \mu, \rho \rangle \in \frac12 \mathbb Z \,.\]
Now let $p \nmid D$ be a prime. As noted above there is a maximal split torus $\mathbf A \subset \mathbf G_{\mathbb Q_p}$ such that $K_p$ corresponds to a hyperspecial point of the apartment of $\mathbf A$, and by the Cartan decomposition \cite[\S3.3.3]{tits1979} the double cosets $K_p \backslash \mathbf G(\mathbb Q_p) / K_p$ are represented uniquely by the elements $\mu(p)$ when $\mu \in X_*(\mathbf A)$ runs through the cocharacters in the closure of the positive chamber. Conjugating $\mathbf A$ to $\mathbf T$ over $\overline{\mathbb Q_p}$ yields a norm $\lVert \cdot \rVert$ on $X_*(\mathbf A)$ induced from the one on $X_*(\mathbf T)$, which does not depend on the choice of $\mathbf A$ \cite[\S2.6]{brumley2020}. For $\kappa \geq 0$ we define now the ``truncated algebra'' (which is not an algebra)
\[ \mathcal H_p^{\leq \kappa} = \operatorname{span}_{\mathbb C} \{ 1_{K_p \mu(p) K_p} : \mu \in X_*(\mathbf A), \lVert \mu \rVert \leq \kappa \} \,. \]
When $S$ is a set of primes not dividing $D$, we define $\mathcal H_S^{\leq \kappa}$ to be the restricted tensor product $\bigotimes_{p \in S} \mathcal H_p^{\leq \kappa} \subset \mathcal H_S$.

Finally, we will use yet another notion of truncated Hecke algebra to accomodate the Hecke operators that we will use to prove Theorem~\ref{largevaluesslthree}. When $S$ is a set of primes not dividing $D$, $\kappa \geq 0$ and $M \geq 1$, define the truncated algebra
\[ \mathcal H_{S, M}^{\leq \kappa} = \bigoplus_{\substack{n \leq M \\ \text{squarefree}}} \bigotimes_{\substack{p \in S \\p \mid n}} \mathcal H_p^{\leq \kappa} \,. \]

\subsection{Hecke-Maass forms}

The group $G$ acts on $[\mathbf G]$ by translation on the right, and this induces an action of the universal enveloping algebra $U(\mathfrak g)$ on $C^\infty([\mathbf G])$, which using the decomposition \eqref{locsymspacedisjointunion} we may view as given by left-invariant differential operators on $G$ (\S\ref{sectionderivatives}). By a Hecke-Maass form on $X$ we mean a smooth function on $X$ that is a joint eigenfunction for the center $Z(U(\mathfrak g))$ and for the Hecke algebra $\mathcal H_f$ from \S\ref{secheckealgebra}.

Let $(f_j)_{j \geq 0}$ be an orthonormal basis of $L^2(X)$ consisting of (complex-valued) Hecke-Maass forms. It may be obtained by decomposing $L^2(X)$ into eigenspaces for the Laplace-Beltrami operator, and finding in every eigenspace a basis of eigenfunctions for the commutative algebras $Z(U(\mathfrak g))$ and $\mathcal H_f$.
When $k \in \mathcal H$, $\mathcal H_\infty$ or $\mathcal H_f$, denote by $\widehat k(f_j)$ the eigenvalue of $f_j$ under $k$, and define the spectral parameter $\nu_j$ as in \S\ref{classicalmaasformdefs}.

Let $g_\infty \in G$ be as in \S\ref{sectorusorbitscorrespondence} with the property that $\mathbf H(\mathbb R)$ contains $g_\infty A g_\infty^{-1}$. Define the $\mathbf  H$-period of $f_j$ by
\begin{equation}
\label{definitionHadelicperiod}
\mathscr P_{\mathbf H}(f_j) = \int_{[\mathbf H]} f_j(hg_\infty) d \mu_{\mathbf H}(h) \,,
\end{equation}
where the automorphic quotient $[\mathbf H]$ is as in \S\ref{arithmeticlocsymmspaces}. As in \S\ref{sectorusorbitscorrespondence}, this definition does not depend on the choice of $A$ or $g_\infty$.

\section{Extreme values of toric periods}

\label{extremevaluesthreeproof}

In this section we prove Theorem~\ref{largevaluesslthree}. We set up a relative pre-trace formula in \S\ref{adelicsetuppretrace} and prove the theorem in \S\ref{largevaluesectionproof}. Our notations for locally symmetric spaces, operator algebras and Hecke-Maass forms are as in \S\ref{alggrouppreliminaries}.

\subsection{Adelic setup}

\label{adelicsetuppretrace}

Let $\mathbf G / \mathbb Q$ be semisimple anisotropic as in \S\ref{algebraicgroups}, and let $\mathbf H \subset \mathbf G$ be a maximal torus of maximal split rank over $\mathbb R$.

The pre-trace formula for $\mathbf G$ states that for any $k \in \mathcal H$ (with archimedean component satisfying a positivity condition as in \S\ref{asymptoticproofsetup}) one has
\begin{align*}
 \sum_j \widehat k(f_j) f_j(x) \overline{f_j(y) }= \sum_{\gamma \in \mathbf G(\mathbb Q)} k(x^{-1} \gamma y) \,,
\end{align*}
uniformly for $x, y \in [\mathbf G]$. As in \S\ref{sectorusorbitscorrespondence} let $g_\infty \in \mathbf G(\mathbb R)$ be such that $\mathbf H(\mathbb R) \supset g_\infty A g_\infty^{-1}$, define the $\mathbf H$-periods $\mathscr P_{\mathbf H}$ as in \eqref{definitionHadelicperiod} and define $K_{\mathbf H, f} = K_f \cap \mathbf H(\mathbb A_{f})$.

We will require the following analogue of Lemma~\ref{classicalpartitionuunity}.

\begin{lemma}[Partitions of unity]\label{partitionunityadelic}
There exists a nonnegative Schwartz-Bruhat function $b \in C_c^\infty(\mathbf H(\mathbb A_{\mathbb Q}) / K_{\mathbf H, f})$ such that $\sum_{\gamma \in \mathbf H(\mathbb Q)} b(\gamma h) = 1$ for all $h \in \mathbf H(\mathbb A_{\mathbb Q})$.\end{lemma}

\begin{proof}
The quotient $\mathbf H(\mathbb A_{\mathbb Q}) / (\mathbf H(\mathbb Q) \mathbf H(\mathbb R) K_{\mathbf H, f})$ is finite by \cite[Theorem 5.1]{Platonov1991}. Take coset representatives $h_1, \ldots, h_N \in \mathbf H(\mathbb A_{f})$ and define $b_f \in C_c^\infty(\mathbf H(\mathbb A_{f}))$ by
\[ b_f = \sum_{j = 1}^N 1_{h_j K_{\mathbf H,f}} \,. \]
The group $\mathbf H(\mathbb R)$ factors uniquely as $V \times T$ with $V$ a Euclidean space and $T$ compact (possibly disconnected). The discrete subgroup $\Gamma_{\mathbf H} := \mathbf H(\mathbb Q) \cap K_{\mathbf H, f}$ is a lattice in $\mathbf H(\mathbb R)$ and therefore $\Gamma_V := \Gamma_{ \mathbf H} \cap V$ has finite index in $\Gamma_{\mathbf H}$. We may construct $b_V \in C_c^\infty(V)$ as in Lemma~\ref{classicalpartitionuunity} relative to the lattice $\Gamma_V \subset V$ and define $b_T(h) = [\Gamma_{\mathbf H} : \Gamma_V]^{-1}$ for $h \in T$. Define $b_\infty \in C_c^\infty(\mathbf H(\mathbb R))$ by
\[ b_\infty(h) = b_V(h_V) \cdot b_T(h_T) \,, \]
where we denote $h = (h_V, h_T) \in V \times T \cong \mathbf H(\mathbb R)$.
Define now $b \in C_c^\infty(\mathbf H(\mathbb A_{\mathbb Q}))$ by
\[ b = b_\infty b_f \,. \]
That this $b$ satisfies the requirements can be checked by writing
\[ \sum_{\gamma \in \mathbf H(\mathbb Q)} b(\gamma h) = \sum_{\gamma \in \mathbf H(\mathbb Q)/(\mathbf H(\mathbb Q) \cap K_{H, f})} \sum_{\mu \in \mathbf H(\mathbb Q) \cap K_{H, f} }b(\gamma \mu h) \,. \qedhere\]
\end{proof}

\begin{remark}Lemma~\ref{partitionunityadelic} is a special case of a much more general statement about continuous functions on locally compact homogeneous spaces; see \cite[\S III.4]{Nachbin1965}.\end{remark}

The Weyl group $N_{\mathbf G}(\mathbf H)/\mathbf H$ is finite \cite[\S IV.11.19]{borel1991} and therefore so is its group of rational points.

\begin{lemma}\label{adelicmainplusoffdiag}Let $k \in C_c^\infty(K \backslash \mathbf G(\mathbb A_{\mathbb Q}) / K)$ be a Schwartz-Bruhat function. There exists $b \in C_c^\infty(\mathbf H(\mathbb A_{\mathbb Q}))$ such that
\begin{align*}
&\sum_{j} \widehat k(f_j)  \lvert \mathscr P_{\mathbf H}( f_j) \rvert^2 \\
&\quad= \operatorname{Vol}([\mathbf H]) \sum_{\gamma \in N_{\mathbf G}( \mathbf H)(\mathbb Q) / \mathbf H(\mathbb Q)} \int_{\mathbf H(\mathbb A_{\mathbb Q})} k(g_\infty^{-1}\gamma h g_{\infty}) d\mu_{\mathbf H}(h) \\
& \qquad+ \sum_{\gamma \in \mathbf G(\mathbb Q) - N_{\mathbf G}(\mathbf H)(\mathbb Q)} \int_{\mathbf H(\mathbb A_{\mathbb Q})^2} b(h_1)b(h_2) k(g_\infty^{-1}h_1^{-1} \gamma h_2 g_\infty ) d\mu_{\mathbf H}(h_1)d\mu_{\mathbf H}(h_2) \,.
\end{align*}
\end{lemma}

\begin{proof}
Using Lemma~\ref{partitionunityadelic}, this is entirely analogous to the proof of Lemma~\ref{classicalmainplusoffdiag} by an unfolding argument and introduction of cutoff functions.
\end{proof}

\begin{remark}The integrals in the diagonal term in Lemma~\ref{adelicmainplusoffdiag} can depend on (the finite part of) $\gamma$, because it is not always true that $N_{\mathbf G}(\mathbf H)(\mathbb Q) \subset \mathbf H(\mathbb Q)\cdot K_f$. When $G$ is $\mathbf{PGL}_2$ or an inner form, this inclusion is related to the notion of ``reciprocal geodesics'' in \cite{sarnak2007}.
\end{remark}

\subsection{Comparison of trace formulas}

We now specialize to the groups $\mathbf G$ in Theorem~\ref{largevaluesslthree}, although $\mathbf G$ can still be as general as in Theorem~\ref{meansquareasymptotic} in this subsection.
We will prove Theorem~\ref{largevaluesslthree} by comparing asymptotics for an amplified trace formula and an amplified relative trace formula. 

Let $\rho : \mathbf G \to \mathbf{SL}_d$ be as in \S\ref{algebraicgroups}, and when $p$ is prime define $\lVert g \rVert_p$ for $g \in \mathbf G(\mathbb Q_p)$ as the maximal absolute value of the entries of $\rho(g)$. When $g \in \mathbf G(\mathbb Q)$ we have for almost all primes $p$ that $\rho(g) \in \mathbf{SL}_d(\mathbb Z_p)$ and consequently $\lVert g \rVert_p = 1$. We may therefore define $\lVert g \rVert_f = \prod_p \lVert g \rVert_p$.

\begin{lemma}\label{diophantinedistancebound}There exists $A > 0$ such that when $\gamma \in \mathbf G(\mathbb Q)$ is such that $g_\infty^{-1}\gamma g_\infty \notin \bigcup_{L \in \mathcal L - \{G\}} M' L$, then
\[ d \left(g_\infty^{-1}\gamma g_\infty, \bigcup_{L \in \mathcal L - \{G\}} M' L \right) \gg \lVert \gamma \rVert_f^{-A} \,. \]\end{lemma}

\begin{proof}This can be shown as in \cite[Lemma 5.1]{brumley2020} as a consequence of the product formula for global fields. The difference is that we must deal with varieties that are not defined over our base field $\mathbb Q$.

It suffices to prove a bound as in the statement for every individual Levi $L \in \mathcal L - \{G\}$. Fix such $L$.  

Let $E \subset \mathbb R$ be a number field such that the extension of scalars $\mathbf H_E$ has the same split rank as $\mathbf H_{\mathbb R}$. Let $\mathbf T \subset \mathbf H_E$ be the maximal split subtorus, defined over $E$. Because all roots of $\mathbf H_{\mathbb R}$ are defined over $E$, we may find an element $h \in \mathbf T(E)$ whose centralizer in $\mathbf G$ has real points $g_\infty L g_\infty^{-1}$. Consider now the subvariety $\mathbf V$ of $\mathbf G$ defined over $E$ by the equation $\operatorname{Ad}_g(h) \in \mathbf H$. Replacing $h$ by a finite power if necessary, we may assume that $h \in \mathbf T(\mathbb R)^0 = g_\infty A g_\infty^{-1}$. Lemma~\ref{truelemma} then implies that $\mathbf V(\mathbb R) = g_\infty M' Lg_\infty^{-1}$. Thus by assumption, $\gamma \notin \mathbf V(E)$.

Consider now the embedding $\rho : \mathbf G \to \mathbf{SL}_d \subset \mathbf A_{\mathbb Q}^{d^2}$, and let $p_1, \ldots, p_k$ be polynomials defined over $E$ whose joint zero locus is $\rho(\mathbf V)$. Expanding their coefficients in a basis for $E/\mathbb Q$, we may find polynomials $q_1, \ldots, q_t$ defined over $\mathbb Q$ whose rational joint zero locus is $\rho(\mathbf V(E) \cap \mathbf G(\mathbb Q))$. There exists $A > 0$ such that for all $q_i$ and all $\gamma \in \mathbf G(\mathbb Q)$ we have a bound of the form
\[ \prod_{p} |p_i(\rho(\gamma))|_{p} \ll \lVert\gamma \rVert_{f}^A\,. \]
Indeed, $A > 0$ can be chosen to be $\max_{i}\deg(q_i)$, simply because a polynomial of degree $s$ can only increase denominators as much as raising them to the power $s$.

Now assume $\gamma \in \mathbf G(\mathbb Q)$ does not lie in $\mathbf V(E)$. Then there exists $q_i$ such that $q_i(\rho(\gamma)) \neq 0$. By the product formula for $\mathbb Q^\times$, this means that
\[ \prod_{v}|q_i(\rho(\gamma))|_v = 1\,, \]
where the product now runs over all absolute values of $\mathbb Q$. Using our bound for the product over the finite primes, this gives
\[ |q_i(\rho(\gamma))|_\infty \gg \lVert\gamma \rVert_{f}^{- A} \,. \]
By smoothness of the action of $\mathbf G(\mathbb R)$ on $h$, the left-hand side is bounded from above by a constant times $d(\gamma, \mathbf V(\mathbb R))$. This shows that
\[ d(\gamma,  g_\infty M' Lg_\infty^{-1}) \gg \lVert\gamma \rVert_{f}^{- A} \,,\]
and the lemma follows.
\end{proof}

When $S$ is a finite set of primes, define $q_S = \prod_{p \in S} p$. Let the integer $D$ be as in \S\ref{secheckealgebra}.

\begin{proposition}\label{smoothamplifiedrelativewitherror}Let $D_{\mathfrak a^*} \subset (\mathfrak a^*)^{\operatorname{gen}}$ be compact. let $\kappa \geq 0$. There exist $A, B, \delta > 0$ such that the following holds. Let $S$ a finite set of primes not dividing $D$ and $k_f \in \mathcal H_S^{\leq \kappa}$. Let $\nu \in D_{\mathfrak a^*}$, $t \geq 0$ and $k_{t\nu} \in \mathcal H_\infty$ as in Lemma~\ref{archimedeantestfunction}. Define $k = k_{t\nu} \otimes k_f$. Then
\begin{align*}
\sum_{j} \widehat k(f_j) \lvert \mathscr P_{\mathbf H}(f_j) \rvert^2&= \operatorname{Vol}([\mathbf H]) \cdot \sum_{\gamma \in N_{\mathbf G}(\mathbf H)(\mathbb Q) / \mathbf H(\mathbb Q)} \int_{\mathbf H(\mathbb A_{\mathbb Q})} k(g_\infty^{-1}\gamma h g_{\infty}) d\mu_{\mathbf H}(h) \\
&\quad+ O\left(\beta( t\nu)(1+t)^{-r}(1+q_S^{- A} t)^{-\delta} q_S^B \lVert k_f \rVert_\infty \right) \,,
\end{align*}
uniformly in $\nu$ and $t$. If $M \geq 1$ and $k_f \in \mathcal H_{S, M}^{\leq \kappa}$ then the same result holds with $q_S$ replaced by $M$.
\end{proposition}

\begin{proof}We start from Lemma~\ref{adelicmainplusoffdiag} and must bound the off-diagonal terms. The fixed function $b$ is a finite sum of factorizable Schwartz-Bruhat functions, so up to introducing a finite sum, we may replace the factors $b(h_1)$ and $b(h_2)$ by factorizable Schwartz-Bruhat functions $b_1(h_1)$ and $b_2(h_2)$. (In fact, we only require that they factorize as an infinite times a finite part, and this is already the case for $b$ as constructed in Lemma~\ref{partitionunityadelic}.) For fixed $\gamma$, the double integral now factorizes as the integral
\[ \int_{\mathbf H(\mathbb R)^2} b_1(h_1) b_2(h_2) k_{t \nu}(g_\infty^{-1} h_1^{-1} \gamma h_2 g_\infty) d\mu_{\mathbf H, \mathbb R}(h_1)d\mu_{\mathbf H, \mathbb R}(h_2) \]
times an integral involving $b_1$, $b_2$ and $k_f$ over a fixed compact subset of $\mathbf H(\mathbb A_f)$. The latter can be bounded trivially by
\[ \ll \lVert k_f \rVert_\infty \cdot \operatorname{Vol}(\operatorname{supp}(b_1)) \cdot \operatorname{Vol}(\operatorname{supp}(b_2)) \cdot \operatorname{Vol}(\operatorname{supp}(k_f)) \,.\]
The first two volumes are bounded, and the third can be bounded by a power of $q_S$, as in \cite[Lemma 4.4]{brumley2020}.

The archimedean integral equals 
\[ \int_{A \times A} b_1(g_\infty a_1 g_\infty^{-1})b_2(g_\infty a_2 g_\infty^{-1}) k_{t \nu}(a_1 g_\infty^{-1} \gamma g_\infty a_2) d a_1 da_2 \,, \]
where we have written $\mathbf H(\mathbb R)$ as a product of $g_\infty A g_{\infty}^{-1}$ and a compact torus, where the compact torus necessarily lies in $g_\infty M g_{\infty}^{-1}$ (see \S\ref{sectorusorbitscorrespondence}) and we may omit it from the integration by our choice of measure on $\mathbf H(\mathbb R)$ (see \S\ref{sectionadelicmeasures}). By our choice of $G$ and $X$, we have that
\[ g_\infty^{-1} \gamma g_\infty \notin \bigcup_{L \in \mathcal L - \{G\}} M' L \]
for $\gamma \in \mathbf G(\mathbb Q)$ as above. Using the inversion formula \eqref{harishchandrainversion}, the rapid decay of $\widehat k_{t \nu}$ and Proposition~\ref{orbitalintegralboundclosetolevi}, this means that the archimedean integral above is bounded by
\[ \ll \beta(t\nu) (1 + t)^{-r} \left(1 + d\left(g_\infty^{-1} \gamma g_\infty, \bigcup_{L \in \mathcal L - \{G\}} M' L \right)^{-N} \cdot t\right)^{-\delta} \]
for certain $\delta, N > 0$. By Lemma~\ref{diophantinedistancebound} we may bound the distance from below by $\lVert \gamma \rVert_f^{-A}$ for some $A > 0$, and as in \cite[Lemma 4.4]{brumley2020} the latter is again bounded from below by $q_S^{-A'}$ for some $A' > 0$. It remains to show that the number of $\gamma$ that contribute to the sum is at most polynomial in $q_S$. Indeed, for such $\gamma$ we have that the rational element $\rho(\gamma) \in \mathbf{SL}_d(\mathbb Q)$ is bounded in the archimedean sense by the support conditions on $b_1$ and $b_2$, and its denominators are at most of size $q_S^{A'}$ by the same argument we just used. This proves the desired bound for the off-diagonal terms.

When $k_f \in \mathcal H_{S, M}^{\leq \kappa}$, it is a sum of at most $M$ elements $k_{f, n} \in \mathcal H_{S_n}^{\leq \kappa}$, with $S_n$ a set of primes with $q_{S_n} \leq M$. Moreover, we may assume that the $k_{f, n}$ have disjoint supports. The result then follows by summing the different bounds for each of the $k_{f, n}$, and replacing $B$ by $B+1$.
\end{proof}

We will require the following estimates for the trace formula.

\begin{proposition}\label{brumarweyllaw}There are constants $\eta, A, B > 0$ and an integer multiple $D'$ of $D$ such that the following holds. Let $\kappa \geq 0$ and $S$ a finite set of primes not dividing $D'$, and $k_f \in \mathcal H_S^{\leq \kappa}$. Let $D_{\mathfrak a^*} \subset \mathfrak a^* - \{0\}$ be compact, $\nu \in D_{\mathfrak a^*}$, $t \geq 1$ and $k_{t\nu} \in \mathcal H_\infty$ as in Lemma~\ref{archimedeantestfunction}. Define $k = k_{t\nu} \otimes k_f$. Then
\[ \sum_{j} \widehat k(f_j) = \operatorname{Vol}(X) \cdot k(1) + O\left(\beta(t \nu) (1 + t)^{-\eta} q_S^{A\kappa + B}  \lVert k_f \rVert_\infty \right) \,, \] 
uniformly in $\nu$ and $t$. If $M \geq 1$ and $k_f \in \mathcal H_{S, M}^{\leq \kappa}$ then the same result holds with $q_S$ replaced by $M$.
\end{proposition}

\begin{proof}
The first statement follows from \cite[Theorem 3.1]{brumley2020}. Note that our Cartan involution is not assumed to be defined over $\mathbb Q$ (see Remark~\ref{remarkcartannotrational}). This is compensated for by our assumption that $\mathbf G$ is almost simple over $\mathbb R$, which we can use to replace the argument in \cite[Lemma 6.3]{brumley2020} needed to show that the centralizer of $G^0$ in $\mathbf G(\mathbb Q)$ is reduced to the center.

The statement about $k_f \in \mathcal H_{S, M}^{\leq \kappa}$ is shown as for Proposition~\ref{smoothamplifiedrelativewitherror}.
\end{proof}

\begin{remark}The estimate \cite[Theorem 3.1]{brumley2020} is stronger than what is needed here, because we do not need uniformity with respect to a variable compact subgroup $K$ (the level aspect) nor do we need asymptotics for spectral parameters that can be singular. But it is the only directly quotable Weyl law-type result that we have found, that covers anisotropic groups and adelic test functions. Note also that the exponents $A$ and $B$ are ineffective, but they will not play any role.\end{remark}

We will apply Propositions~\ref{smoothamplifiedrelativewitherror} and \ref{brumarweyllaw} always with the following positivity assumption on $k_f$:
\begin{equation}\label{conditionkfin}
\text{$\widehat k_f(f_j) \geq 0$ for all $f_j$, $k_f(1) \geq 1$ and $k_f(g) \geq 0$ for $g \in G(\mathbb A_{f})$.}
\end{equation}

The following Proposition combines the two trace formulae and reduces Theorem~\ref{largevaluesslthree} to finding a suitable test function $k_f$.

\begin{proposition}\label{comparisonimplieslargevalues}Let $D_{\mathfrak a^*} \subset (\mathfrak a^*)^{\operatorname{gen}}$ be compact. Let $\kappa \geq 0$. There exist $\delta, C > 0$ and an integer multiple $D'$ of $D$ such that the following holds. Let $\nu \in D_{\mathfrak a^*}$, $S$ be a finite set of primes not dividing $D'$, $t, M \geq 1$ and $k_f \in \mathcal H_{S, M}^{\leq k}$ satisfying \eqref{conditionkfin}. If $M \leq (1+t)^\delta$ and $\lVert k_f \rVert_\infty \leq (1+t)^\delta$, then
\[ \frac{\sum_{\lVert \nu_j - t \nu \rVert \leq C} \widehat k_f(f_j) \lvert\mathscr P_H(f_j) \rvert^2}{\sum_{\lVert \nu_j - t \nu \rVert \leq C} \widehat k_f(f_j)} \gg (1+t)^{-r} \cdot \frac{\int_{H(\mathbb A_f)}k_f}{k_f(1)} \,, \]
uniformly in $\nu, t, S, M, k_f$.
\end{proposition}

\begin{proof}
We first bound the main term in Proposition~\ref{smoothamplifiedrelativewitherror} from below. The integrals factorize. The archimedean component can be bounded from below by $\beta(t\nu) (1+t)^{-r}$. Indeed, using the inversion formula \eqref{harishchandrainversion} it reduces to a similar integral involving the spherical function, which is seen to be independent of $\gamma$ by making a change of variables, and we conclude using Proposition~\ref{mainintegralsphericalfunction}. For the finite component we may use the positivity of $k_f$ and bound the integral trivially from below by $k_f(1) \operatorname{Vol}(\mathbf H(\mathbb A_{f}) \cap K_f) \geq 1$. Given this lower bound for the main term and given the quality of the error term in Proposition~\ref{smoothamplifiedrelativewitherror}, a truncation argument as in \cite[Lemma 4.5]{brumley2020} shows the asymptotic formula
\begin{align*}
&\sum_{\lVert \nu_j - t \nu \rVert \leq C} \widehat k_f(f_j) \left\lvert \mathscr P_H(f_j)\right\rvert^2 \\
& \qquad \asymp \operatorname{Vol}([\mathbf H]) \sum_{\gamma \in N_{\mathbf G}(\mathbf H)(\mathbb Q) / \mathbf H(\mathbb Q)} \int_{\mathbf H(\mathbb A_{\mathbb Q})} k(g_\infty^{-1}\gamma h g_{\infty}) d\mu_{\mathbf H}(h) \,,
\end{align*}
provided that $C$ is large enough and that $M$ and $\lVert k_f \rVert_\infty$ are bounded by a sufficiently small power of $t$ so that the error terms are controlled.
Similarly, the inversion formula \eqref{harishchandrainversion} gives that $k_{t \nu}(1) \asymp \beta(t \nu)$ and a truncation argument applied to Proposition~\ref{brumarweyllaw} gives that
\[\sum_{\lVert \nu_j - t \nu \rVert \leq C} \widehat k_f(f_j) \asymp \operatorname{Vol}(X) k(1) \,, \]
provided that $C$ is large enough and that $M$ and $\lVert k_f \rVert_\infty$ are bounded by a sufficiently small power of $t$.

The statement now follows by taking the quotient of those two asymptotics, using again the archimedean asymptotics, and in the case of the relative trace formula, using positivity of $k_f$ to discard the terms corresponding to $\gamma$ that are not the identity.
\end{proof}

\subsection{Amplification}

\label{largevaluesectionproof}

The construction of amplifiers is our only reason to restrict to groups $\mathbf G$ as in Theorem~\ref{largevaluesslthree}. The theorem follows by applying Proposition~\ref{comparisonimplieslargevalues} with $k_f$ as given by the following proposition, and with $M$ a sufficiently small power of $t$.

\begin{proposition} \label{existencetestfunctionbig} Let $G$ be a $\mathbb Q$-form of $\mathbf{PGL}_{3}$ and let $\delta$ be as stated in Theorem~\ref{largevaluesslthree}. Let $M \geq 2$ and $S$ be the set of primes less than $M$ that do not divide $D'$. There exists $k_f \in \mathcal H_{S, M}^{\leq 4}$ satisfying \eqref{conditionkfin}, such that $\lVert k_f\rVert_\infty \leq M^A$ for some $A > 0$ and 
\[ \frac{\int_{\mathbf H(\mathbb A_f)}k_f}{k_f(1)} \gg (\log  \log M)^{\delta + o(1)} \,. \]
\end{proposition}

We prove Proposition~\ref{existencetestfunctionbig} in \S\ref{sectionamplifiers}. We will also prove a result that states that the lower bound in Proposition~\ref{existencetestfunctionbig} is optimal in a certain sense, although the exact result is not as strong as the optimality statement in \cite{milicevic2010}, and we do not formally exclude the existence of amplifiers of similar quality for, say, forms of $\mathbf{PGL}_n$. See Remark~\ref{remarkglnboundoptimality}.

\section{Construction of amplifiers}

\label{sectionamplifiers}

In this section we prove the existence of the amplifier in Proposition~\ref{existencetestfunctionbig} (in \S\ref{amplifierlowerbounds}) . We also prove an optimality result in Proposition~\ref{optimalityslthreefull} (in \S\ref{amplifierupperbounds}).

\subsection{Preliminary computations}

\label{amplifierheckelattices}

We begin with some preliminary computations of integrals on $p$-adic groups. It is convenient to re-introduce some notation in order to make the key computations more self-contained. Let $n \geq 2$, let $p$ be a prime number and denote $G_p = \operatorname{PGL}_n(\mathbb Q_p)$, $K_p = \operatorname{PGL}_n(\mathbb Z_p)$ and define $H_p \subset G_p$ to be the diagonal torus.
Most of the time we will specialize to $n = 3$, but we allow some generality to be able to make remarks about other $n$, and because certain statements will be proven for arbitrary $n$. We fix Haar measures on $G_p$ and $H_p$ such that the compact open subgroups $K_p$ and $H_p \cap K_p$ have volume $1$. Define $\mathcal H_p = C_c^\infty(K_p \backslash G_p /K_p)$, a commutative algebra with convolution defined by
\[ (k_1 * k_2) (g) = \int_{G_p} k_1(h) k_2(h^{-1}g) dh \,.\]
The adjoint of $k \in \mathcal H_p$ is defined by $k^*(g) = \overline{k(g^{-1})}$. This defines an algebra involution on $\mathcal H_p$.

The question we seek to answer is the following: When $k_p \in \mathcal H_p$ is a function of the form $k_p' * (k_p')^*$ with $k_p' \in \mathcal H_p$, how big can $\int_{H_p} k_p$ be relative to $k_p(1)$?

Such $k_p'$ is a finite linear combination of the basic double coset kernels, which are defined as follows. Let $\mathbf a = (a_1, \ldots, a_n) \in \mathbb Z^n$ and denote $\mu_{\mathbf a}(p) = \operatorname{diag}(p^{a_1}, \ldots, p^{a_n})$. One may view $\mu_{\mathbf a}$ as a cocharacter of $H_p$. Define the function
\[
\tau(\mathbf a, p) = 1_{K_p \mu_{\mathbf a}(p) K_p} \,.
\]
We have $\tau(\mathbf a, p)^* = \tau(-\mathbf a, p)$, and $\tau(\mathbf 0, p)$ is the identity of $C_c^\infty(K_p \backslash G_p / K_p)$.
The function $\tau(\mathbf a, p)$ depends only on the entries of $\mathbf a$ up to permutation and up to translation by a common element in $\mathbb Z$. We may therefore always reduce to the case where $a_1 \geq a_2 \geq \ldots \geq a_n = 0$. The degree $\deg(\mu_{\mathbf a}(p))$ is defined as the cardinality $\#( K_p \mu_{\mathbf a}(p) K_p /K_p)$. With our choice of Haar measures we have
\begin{equation}\label{degreeheckeltwo}
\deg(\mu_{\mathbf a}(p)) = \left\lVert \tau(\mathbf a, p) \right\rVert_2^2 = (\tau(\mathbf a, p) * \tau(-\mathbf a, p))(1) \,.
\end{equation}

We will need bounds for integrals of convolutions of these basic kernels. Specifically, the quantity we will be interested in is the off-diagonal contribution to the integral $\int_{H_p} \tau(\mathbf a_1, p) * \tau(\mathbf a_2, p)$, when $\mathbf a_1, \mathbf a_2 \in \mathbb Z^n$. That is, we are interested in the integral over $H_p - (H_p \cap K_p)$, or what is the same, the integral over $H_p$ minus $(\tau(\mathbf a_1, p) * \tau(\mathbf a_2, p))(1)$.

The following is an example computation of off-diagonal contributions when $n = 2$, which to an extent can be seen geometrically in the Bruhat-Tits tree.

\begin{example}\label{exampleoffdiagonalpadictwo}Let $n = 2$ and define $\mathbf a = (1, 0)$ and $\mathbf b = (0, 0)$. Then $\deg(\mu_{\mathbf a}(p)) = p+1$, $\tau(\mathbf a, p)$ is self-adjoint and
\begin{align*}
&\left(\int_{H_p} \tau(\mathbf a, p) * \tau(\mathbf a, p)\right) -\deg(\mu_{\mathbf a}(p)) = (p+3) - (p+1) = 2 \,, \\
&\int_{H_p} \tau(\mathbf a, p) * \tau(\mathbf b, p) = 2  \,.
\end{align*}
That $\deg(\mu_{\mathbf a}(p)) = p+1$, the cardinality of a radius $1$ ball in the Bruhat-Tits tree, is a classical counting problem of double cosets. That $\tau(\mathbf a, p)$ is self-adjoint holds because $\mathbf a$ and $- \mathbf a$ are $\mathbb Z$-translates. We look at the second integral first. It equals $2$ because it is the intersection of a radius $1$ ball in the Bruhat-Tits tree with an apartment containing the center of the ball. We can also show this more explicitly as follows. We must count how many left cosets of $K_p$ are contained in $K_p \operatorname{diag}(p, 1) K_p$ and intersect $H_p$. We may lift this problem to $\operatorname{GL}_2(\mathbb Q_p)$ and ask how many diagonal matrices lie in $\operatorname{GL}_2(\mathbb Z_p) \operatorname{diag}(p, 1) \operatorname{GL}_2(\mathbb Z_p)$. Multiplication on the left or the right by $\operatorname{GL}_2(\mathbb Z_p)$ does not change the invariant factors of a matrix, so that such a diagonal matrix must have invariant factors $(p, 1)$. It is clear that this gives exactly two matrices, and the statement follows.

The first integral can be computed using a similar argument, after using the Hecke relation $\tau(\mathbf a, p) * \tau(\mathbf a, p) = \tau(2\mathbf a, p) + (p+1) \tau(\mathbf 0, p)$.
\end{example}

\begin{lemma}When $\mathbf a, \mathbf b \in \mathbb Z^n$ are decreasing tuples with $a_n = b_n = 0$, we have $(\tau(\mathbf a, p) * \tau(-\mathbf b, p))(1) = \delta_{\mathbf a, \mathbf b} \deg(\mu_{\mathbf a}(p))$.\end{lemma}

\begin{proof}
When $\mathbf a = \mathbf b$ this is \eqref{degreeheckeltwo}. When $\mathbf a \neq \mathbf b$ the double cosets represented by $\mu_{\mathbf a}(p)$ and $\mu_{\mathbf b}(p)$ are distinct, either by the Cartan decomposition \cite[\S3.3.3]{tits1979} or by an argument using invariant factors. They are therefore disjoint, so that
\[ (\tau(\mathbf a, p) * \tau(-\mathbf b, p))(1) = \langle \tau(\mathbf a, p), \tau(\mathbf b, p) \rangle_{L^2(G_p)} = 0 \,. \qedhere \]
\end{proof}

To do explicit computations, we translate integrations into counting problems involving lattices in $\mathbb Q_p^n$, which are the natural generalization of points in the Bruhat-Tits tree. By a lattice in $\mathbb Q_p^n$ we mean a finitely generated $\mathbb Z_p$-submodule of rank $n$. Denote by $\mathscr R$ the set of lattices in $\mathbb Q_p ^n$, and by $\overline{\mathscr R}$ the set of homothety classes of lattices. The group $\operatorname{GL}_n(\mathbb Q_p)$ acts on $\mathscr R$ and the group $G_p$ acts on $\overline{\mathscr R}$. Denote $L_0 = \mathbb Z_p^n$ and $\overline{L_0}$ its homothety class; they are our base points. The stabilizers of $L_0$ and $\overline{L_0}$ in $\operatorname{GL}_n(\mathbb Q_p)$ and $G_p$ are $\operatorname{GL}_n(\mathbb Z_p)$ and $K_p$, respectively. Acting on the base point yields bijections $\operatorname{GL}_n(\mathbb Q_p) / \operatorname{GL}_n(\mathbb Z_p) \cong \mathscr R$ and $G_p / K_p \cong \overline{\mathscr R}$.

\begin{lemma}\label{doublecosetaslattices}Let $\mathbf a = (a_1, \ldots, a_n) \in \mathbb Z^n$ with $a_1 \geq a_2 \geq \ldots \geq a_n = 0$. There is a natural bijection between the following sets:
\begin{itemize}
\item The lattices $L \subset L_0$ for which $L_0 / L$ has invariant factors $p^{a_1}, \ldots, p^{a_n}$;
\item The left coset space $K_p \mu_{\mathbf a}(p) K_p / K_p$,
\end{itemize}
which is given as follows: To a lattice $L$ one associates the homothety class $\overline L \in \overline{\mathscr R}$, which is identified with a left coset of $K_p$ through the bijection $G_p / K_p \cong \overline{\mathscr R}$.
\end{lemma}

\begin{proof} That the map is well-defined (meaning, that it lands in $K_p \mu_{\mathbf a}(p) K_p / K_p$) is the following fact: When $M$ is a free module over a PID with a submodule $N$, then there exists a basis $(e_i)$ of $M$ that is adapted to $N$, meaning that $N$ has a basis consisting of scalar multiples of the $e_i$. This fact shows that every lattice $L \subset L_0$ such that $L_0/L$ has invariant factors as given, is of the form $k_p \operatorname{diag}(a_1, \ldots, a_n) L_0$ with $k_p \in \operatorname{GL}_n(\mathbb Z_p)$. That the map is surjective is trivial, because the lattice $k_p \operatorname{diag}(a_1, \ldots, a_n) L_0$ is sent to the corresponding left coset in $K_p \mu_{\mathbf a}(p) K_p / K_p$.
\end{proof}

Denote by $(e_1, \ldots, e_n)$ the standard basis of $L_0$. We call $L \in \mathscr R$ an adapted lattice if it has a basis of the form $(b_1 e_1, \ldots, b_n e_n)$ with the $b_i \in \mathbb Q_p^\times$.

\begin{lemma}Let $\mathbf a = (a_1, \ldots, a_n) \in \mathbb Z^n$ with $a_1 \geq a_2 \geq \ldots \geq a_n = 0$. The bijection from Lemma~\ref{doublecosetaslattices} restricts to a bijection between the following sets:
\begin{itemize}
\item The set of adapted lattices $L \subset L_0$ for which $L_0 / L$ has invariant factors $p^{a_1}, \ldots, p^{a_n}$.
\item $(H_p \cap K_p \mu_{\mathbf a}(p) K_p) / (H_p \cap K_p)$.
\end{itemize}
\end{lemma}

\begin{proof}This can be proven as for Lemma~\ref{doublecosetaslattices}.
\end{proof}

\begin{lemma}\label{integraltolattices}Let $\mathbf a, \mathbf b \in \mathbb N^n$ be decreasing tuples with $a_n = b_n = 0$. Then
 $\int_{H_p} \tau(\mathbf a, p) * \tau(-\mathbf b, p)$ counts the pairs of lattices $(L_1, L)$ with the following properties:
\begin{itemize}
\item $L_1$ is adapted and $L \subset L_0 \cap L_1$.
\item The invariant factors of $L_0/L$ are given by $\mathbf a$ and those of $L_1 / L$ by $\mathbf b$.
\end{itemize}
\end{lemma}

\begin{proof}We have
\begin{align*}
\int_{H_p} \tau(\mathbf a, p) * \tau(-\mathbf b, p) &= \int_{H_p/(H_p \cap K_p)} \int_{G_p/K_p} \tau(\mathbf a, p)(g)  \tau(-\mathbf b, p)(g^{-1}h)dg dh \,.
\end{align*}
By our convention on Haar measures, the right hand side is a sum over cosets $h (H_p \cap K_p)$ and $g K_p$. Take such cosets for which the integrand is nonzero (and thus equal to $1$). Then $ g\overline L_0$ is represented by a unique lattice $L \subset L_0$ for which $L/L_0$ has invariant factors given by $\mathbf a$, and this $L$ does not depend on the representative $g$. There is a unique lift $h_1 \in \operatorname{GL}_n(\mathbb Q_p)$ such that $h_1^{-1} L \subset L_0$ has invariant factors given by $\mathbf b$, and this again does not depend on the representative $h$. Define then $L_1 = h_1 L_0$. The pair $(L_1, L)$ satisfies the conditions of the statement. Conversely, to such a pair we can associated unique cosets $h (H_p \cap K_p)$ and $g K_p$ on which the integrand is nonzero.
\end{proof}

\subsection{Lower bounds}

\label{amplifierlowerbounds}

We prove Proposition~\ref{existencetestfunctionbig}. Our construction of an amplifier $k_f$ is inspired by the construction in \cite{milicevic2010}. When translated into the adelic language, the amplifier from \cite{milicevic2010} takes the form $k_f = \omega * \omega^*$, with
\begin{equation}\label{heckeoperatormilishape}
\omega = \sum_{\substack{n \leq M \\ \text{squarefree}}} \prod_{p \mid n} c_p \omega_{p},
\end{equation}
for $\omega_p$ elementary Hecke operators, which in the notation from \S\ref{amplifierheckelattices} equal $\tau((1,0), p)$, and with $c_p > 0$ parameters to be optimized. The equalities in Example~\ref{exampleoffdiagonalpadictwo}, which are about $\tau((1,0), p)$, appear in \cite{milicevic2010}, not explicitly but through global versions thereof that are formulated in terms of divisor functions, in the proof of \cite[Lemma 5]{milicevic2010}.

One can think of the Hecke operator $\omega$ in \eqref{heckeoperatormilishape} as being the formal expansion of the product $\prod ( 1 + c_p\omega_p)$, truncated to only include terms corresponding to sets of primes with product less than $M$.

Our aim is to optimize the choice of $\omega_p$ and of $c_p$ for forms of $\mathbf{PGL}_3$.
We continue to use the notation from \S\ref{amplifierheckelattices} but will soon after switch to a global setup. The following lemma is the key in the construction of an amplifier.

\begin{lemma}\label{winnerkpglthree}Let $n = 3$ and $\mathbf a = (1, 0, 0)$ or $(1, 1, 0)$. Then $\deg(\mu_{\mathbf a}(p)) = p^2 + p + 1$, and we have
\begin{align*}
&\int_{H_p} \tau(\mathbf a, p) * \tau(\mathbf a, p) = 3(p+2) \,, \\
&\left(\int_{H_p} \tau(\mathbf a, p) * \tau(-\mathbf a, p) \right) - \deg(\mu_{\mathfrak a}(p)) = 6  \,, \\
& \int_{H_p} \tau(\mathbf a, p) = 3 \,.
\end{align*}
\end{lemma}

\begin{proof}Using that adjugation is an algebra involution and that $H_p$ is unimodular, it suffices to prove each identity for either $\mathbf a =  (1, 0, 0)$ or $(1, 1, 0)$, because $\tau(\mathbf a, p)$ and $\tau(-\mathbf a, p)$ are adjoints.
When $L_1, L_2 \in \mathscr R$ are lattices, define their generalized index by
\[ [L_1 : L_2] = \frac{[L_1 : L_1 \cap L_2]}{[L_2 : L_1 \cap L_2]} \,. \]
It satisfies the usual transitivity relation. 

For the degree, choose $\mathbf a = (1, 0, 0)$. We must count lattices $L \subset L_0$ of index $p$. Such $L$ contain $p L_0$ and are determined by their quotient modulo $p L_0$. We must then count index $p$ subgroups in $(\mathbb Z / p \mathbb Z) ^3$, or equivalently lines in $\mathbb F_p^3$, of which there are $p^2 + p + 1$.

For the first integral, choose $\mathbf a = (1, 1, 0)$. We must count pairs of lattices $(L_1, L)$ satsifying the conditions in Lemma~\ref{integraltolattices}, with $\mathbf b =(1, 0, 0)$. Thus $L \subset L_0$ is a lattice with $L_0/L \cong (\mathbb Z/p \mathbb Z)^2$ and $[L_1 : L] = p$. Because $L_1$ is adapted, contains $L \supset p \mathbb Z_p$ and $[L_0 : L_1] = [L_0 : L][L : L_1] = p$, it is clear that the only possibilities for $L_1$ are the following:
\begin{itemize}
\item $p^{-1} \mathbb Z_p \oplus p\mathbb Z_p \oplus p \mathbb Z_p$ and
\item $\mathbb Z_p \oplus \mathbb Z_p \oplus p \mathbb Z_p$,
\end{itemize}
up to permutations of the factors. (This list would have been longer if we had switched $\mathbf a$ and $\mathbf b$; what helps us here is that $\mathbf b$ has small entries.) It suffices to count pairs $(L_1, L)$ where $L_1$ is one of these two lattices, and multiply the result by $3$. In the first case, the condition $L \subset L_0 \cap L_1$ forces $L \subset \mathbb Z_p \oplus p \mathbb Z_p \oplus p \mathbb Z_p$, and by considering indices we must have equality. This $L$ is indeed a solution. In the second case, note that $L$ is determined by its quotient modulo $pL_0$, and that $L/p L_0 \subset L_1 / p L_0 \cong (\mathbb Z / p \mathbb Z)^2$ has index $p$. There are $p+1$ such subgroups, and the corresponding $L$ satisfy the conditions. This proves the first statement.

For the second integral, take instead $\mathbf a = (1, 0, 0)$ and apply Lemma~\ref{integraltolattices}. Using as before the observation that $L \supset p \mathbb Z_p$ and using that $[L_0 : L_1] = 1$, the possibilities for $L_1$ are now:
\begin{itemize}
\item $L_0$,
\item $p^{-1}\mathbb Z_p \oplus \mathbb Z_p \oplus p \mathbb Z_p$ and
\item $p^{-2}\mathbb Z_p \oplus p\mathbb Z_p \oplus p \mathbb Z_p$,
\end{itemize}
up to permutation of the factors. When $L_1 = L_0$ we obtain $\deg(\mu_{\mathfrak a}(p))$ pairs $(L_0, L)$. In the second case we use that $L \subset L_0 \cap L_1$ and observe that equality must occur, giving one solution $L$. In the third case the condition $L \subset L_0 \cap L_1$ forces $L \subset \mathbb Z_p \oplus p\mathbb Z_p \oplus p \mathbb Z_p$, and this is impossible. The second statement follows, taking into account permutations.

The computation of the last integral also follows from Lemma~\ref{integraltolattices}, with $\mathbf b = \mathbf 0$. We must count the adapted lattices of index $p$ in $L_0$, and there are $3$ of them.
\end{proof}

We now return to the notation from the previous subsections, and let $\mathbf G$ be a form of $\mathbf{PGL}_{3}$. Let $E$ be the minimal splitting field of $\mathbf H$ and choose an isomorphism $\sigma : \mathbf G_{E} \to \mathbf{PGL}_{3, E}$ that sends $\mathbf H$ to the diagonal torus. Define $\mathcal P_{\operatorname{good}}$ to be the set of primes with the following properties:
\begin{enumerate}
\item $p \nmid D'$, with $D'$ as in Proposition~\ref{comparisonimplieslargevalues};
\item $p$ splits in $E$; equivalently, $E$ embeds in $\mathbb Q_p$;
\item $\sigma(K_p) = \mathbf{PGL}_{3}(\mathbb Z_p)$.
\end{enumerate}

The condition that $p$ splits in $E$ is equivalent to saying that $\mathbf H$ is split over $\mathbb Q_p$. In particular we have for $p \in \mathcal P_{\operatorname{good}}$ that $G_{\mathbb Q_p}$ is split, and therefore isomorphic to $\operatorname{PGL}_{3, \mathbb Q_p}$, by Lemma~\ref{splitimpliesglthree}. This is also apparent from the fact that we may extend $\sigma$ from $E$ to $\mathbb Q_p$.

\begin{lemma}\label{chebotarevcount}We have
\[ \sum_{\substack{p \in \mathcal P_{\operatorname{good}} \\ p \leq x}} \log p = \frac{x}{[E : \mathbb Q]} + o(x) \,,\]
where $E$ is as above.
\end{lemma}

\begin{proof}
The first condition defining $\mathcal P_{\operatorname{good}}$ does not influence the asymptotic, and modulo the second, neither does the third. Meaning, the statement is that the set of primes that split in $E$ has natural density $1/[E : \mathbb Q]$. Because the splitting field $E$ is Galois, this follows from Chebotarev's density theorem with natural density \cite[Theorem 4]{Heilbronn1967}.
\end{proof}

\begin{proof}[Proof of Proposition~\ref{existencetestfunctionbig}]

Let $M \geq 2$.
Fix a real number $c > 0$, which we will optimize later. Let $c_{1} > 0$ be a real number that we will later assume to be sufficiently small. Let $M_{1} =  c_{1} \log M$. When $p$ is a prime, define
\begin{align*}
a_{p} = \begin{cases}
\frac{c}{p} & \text{when} \; p \in \mathcal P_{\operatorname{good}} \; \text{and}\; p \leq M_{1} \,, \\
0 & \text{otherwise.}
\end{cases}
\end{align*}
For $p \in \mathcal P_{\operatorname{good}}$ we may define the elementary Hecke operators $\tau(\mathfrak a, p) \in \mathcal H_p = C_c^\infty(K_p \backslash \mathbf G(\mathbb Q_p) / K_p)$ as in \S\ref{amplifierheckelattices}, through the isomorphism $\sigma$ with $\mathbf{PGL}_3(\mathbb Q_p)$.
Define $\mathbf a = (1, 0, 0)$ and for $p \in \mathcal P_{\operatorname{good}}$ define
\[ \omega_p = a_p \tau(\mathfrak a, p) + a_p \tau(\mathfrak a, p)^* \in \mathcal H_p^{\leq 2} \,. \]
Define
\begin{equation}\label{glthreedefomegaamplifier}
\omega = \sum_{\substack{n \leq M \\ \text{squarefree}}} \prod_{p \mid n} \omega_{p} \in \mathcal{H}_{\mathcal P_{\operatorname{good}}, M}^{\leq 2} \,,
\end{equation}
and finally
\begin{align}
\label{kfconstructionexpansion}
k_{f} = \omega * \omega^* = \sum_{\substack{n, m \leq M \\ \text{squarefree}}} \prod_{\substack{p \mid n \\ q \mid m}} \omega_{p} *\omega_{q}^{*} = \sum_{\substack{n \leq M \\ \text{squarefree}}} \prod_{p \mid n} (\omega_{p} + \omega_{p}^{*} + \omega_{p} *\omega_{p}^{*}) \,,
\end{align}
where the second equality holds by grouping pairs $(n, m)$ with the same least common multiple. (In fact $\omega_p^* = \omega_p$.) Clearly $k_f \in \mathcal H_{S, M}^{\leq 4}$.

It is clear that $k_f$ satisfies \eqref{conditionkfin}. Indeed, it is a self-convolution and therefore has nonnegative eigenvalues. It takes nonnegative values because the $a_p$ are nonnegative, and $k_f(1) \geq 1$ thanks to the term for $n = m = 1$.

We have that $\lVert k_f \rVert_\infty \ll M^A$ for some $A > 0$. This can be seen either by expanding the convolution $\omega_p \omega_p ^*$ in terms of elementary Hecke operators, or using the same arguments as used in \cite[Lemma 4.4]{brumley2020}.

It remains to show the lower bound in Proposition~\ref{existencetestfunctionbig}. Because $\omega_{p}(1) = \omega_{p}^{*}(1) = 0$, we have
\[ k_{f}(1) = \sum_{\substack{n \leq M \\ \text{squarefree}}} \prod_{p \mid n} \lVert \omega_{p} \rVert_{2}^{2} \,. \]
To prove the lower bound, we begin by trivially estimating
\begin{align} \label{maintermnonrelativecrudeupperbound}
k_{f}(1) \leq \prod_{p \leq M_{1}} \left( 1 + \lVert \omega_{p} \rVert_{2}^{2} \right)
\end{align}
by completing the sum over $n$ to all square-free integers. To compute $\int_{H(\mathbb A_f)} k_f$ we integrate \eqref{kfconstructionexpansion} and use the computations from Lemma~\ref{winnerkpglthree}, which give
\begin{align*}
\int_{H(\mathbb A_f)} k_f & = \sum_{\substack{n \leq M \\ \text{squarefree}}} \prod_{p \mid n} \left( 12 a_{p} + (6p+24) a_{p}^{2} + \lVert \omega_{p} \rVert_{2}^{2} \right) \,,
\end{align*}
where we use that $\lVert \omega_{p} \rVert_{2}^{2} = 2 a_p ^2 \deg(\mu_p)$. The term $\lVert \omega_p \rVert_2^2$ corresponds to the diagonal contribution, and we will want the other terms to be large relative to this.
Let $\alpha > 0$. We complete the sum over $n$ to a full product, which introduces an error term that we estimate using Rankin's trick by introducing a factor $(n/M)^\alpha$ in the resulting error terms.
\begin{align} \nonumber
\int_{H(\mathbb A_f)} k_f & = \prod_{p \leq M_{1}} (1 + 12 a_{p}  + (6p+24) a_{p}^{2} + \lVert \omega_{p} \rVert_{2}) \\
\nonumber
& \mathrel{\phantom{=}} + O \left( M^{- \alpha} \sum_{\substack{n > M \\ \text{squarefree}}} \prod_{p \mid n}  p^{\alpha} (12 a_{p} + (6p+24)a_{p}^{2} + \lVert \omega_{p} \rVert_{2}^{2}) \right) \\
\label{relativemaintermerrorandmain}
\begin{split}
& = \prod_{p \leq M_{1}} (1 + 12a_{p}  + (6p+24) a_{p}^{2} + \lVert \omega_{p} \rVert_{2}^{2}) \\
& \mathrel{\phantom{=}} + O \left( M^{- \alpha} \prod_{p \leq M_{1}} \left(1 + p^{\alpha} \left(12 a_{p} + (6p+24)a_{p}^{2} + \lVert \omega_{p} \rVert_{2}^{2}\right)\right) \right)
\end{split}
\end{align}
Applying the inequality $\frac{1 + x}{1 + y} \leq 1 + (x-y) \leq \exp(x-y)$ (for $x \geq y \geq 0$), the ratio of the error term to the main term in \eqref{relativemaintermerrorandmain} is at most
\begin{align} \label{ratioerrortomainrelative}
M^{- \alpha} \exp \left( \sum_{p \leq M_{1}} (p^{\alpha}- 1) ( 12 a_{p} + (6p+24)a_{p}^{2}+ \lVert \omega_{p} \rVert_{2}^{2} ) \right) \,.
\end{align}
We want this to be strictly less than $1$.
By the mean value theorem, $p^{\alpha} - 1 \leq \alpha p^{\alpha} \log p$.
Using the bound $a_{p} \ll 1/p$, we can thus bound \eqref{ratioerrortomainrelative} by
\begin{align*}
\exp \left( - \alpha \log M + O \left( \sum_{p \leq M_{1}} \alpha p^{\alpha} \log p \right) \right) \,.
\end{align*}
Now choose $\alpha = 1 / \log M_{1}$, so that by Chebyshev's estimates this is at most
\begin{align*}
\exp \left( - \alpha (\log M + O ( M_{1})) \right) \,.
\end{align*}
By choosing $c_{1}$ sufficiently small, $M_{1} = c_{1} \log M$ is small enough for the above expression to be at most $1/2$ (say).

It remains to find a lower bound for the ratio of the main term in \eqref{relativemaintermerrorandmain} to the right-hand side of \eqref{maintermnonrelativecrudeupperbound}. Using the bound $1 + x \geq \exp(x + O(x^{2}))$, we have
\begin{align*}
\frac{\int_{H(\mathbb A_f)} k_f }{k_{f}(1)} & \gg \prod_{p \leq M_1} \frac{1  +12 a_{p} + (6p+24) a_{p}^{2} + \lVert \omega_{p} \rVert_{2}^2}{1 + \lVert \omega_{p} \rVert_{2}^2} \\
&\gg \exp \left( \sum_{p \leq M_{1}} \left( \frac{12 a_{p} + (6p+24)a_{p}^{2}}{1 + \lVert \omega_{p} \rVert_{2}^2} + O(1/p^{2}) \right) \right) \\
& \gg \exp \left( \sum_{\substack{p \leq M_{1} \\ p\in \mathcal P_{\operatorname{good}}}} \left(\frac{12 c + 6 c^{2}}{1 + 2 c^{2}} \cdot \frac{1}{p}+O(1/p^{2}) \right) \right) \\
& \gg \exp \left( \frac{1}{[E : \mathbb Q]} \cdot \frac{12 c + 6 c^{2}}{1 + 2 c^{2}} (\log \log M_{1}) (1 + o(1)) \right)  \,.
\end{align*}
In the last step we have used Lemma~\ref{chebotarevcount} together with Abel's summation formula to deduce a Mertens type result from the PNT-type result. The rational function in $c$ that appears is maximal for $c = 1$, where it takes the value $6$. If we recall that $M_1 = c_1 \log M$, this is the lower bound stated in Proposition~\ref{existencetestfunctionbig}.
\end{proof}

\subsection{Upper bounds}

\label{amplifierupperbounds}

We prove that the construction in \S\ref{amplifierlowerbounds} is in a sense optimal, by giving an upper bound for the quotient
\[ \frac{\int_{H(\mathbb A_f)} k_f }{k_{f}(1)} \]
modulo certain restrictions on $k_f$. The main result is Proposition~\ref{optimalityslthreefull}. When $p \in \mathcal P_{\operatorname{good}}$ is a prime, we may identify $\mathbf G_{\mathbb Q_p}$ with $\mathbf{PGL}_{3, \mathbb Q_p}$ using the isomorphism $\sigma$. To simplify the notation, we will not keep track of the set $\mathcal P_{\operatorname{good}}$ here, and state the global bounds instead for the group $\mathbf{PGL}_{3}$, for which we define truncated Hecke algebras in the same way as in \S\ref{sectruncatedhecke}. In fact, we will give statements that are valid more generally for $\mathbf{PGL}_{n}$ with $n \geq 3$. In either case, $\mathbf H$ denotes the diagonal torus.

We will first prove the following key bound. While it is valid for all $n \geq 3$, it is likely not the strongest possible result, see Remark~\ref{remarkglnboundoptimality}. We use the local notation from \S\ref{amplifierheckelattices}.

\begin{proposition}\label{keyboundnegativeresultweak}Let $n \geq 3$ and let $\mathbf a, \mathbf b \in \mathbb N^n$ be decreasing tuples with $a_n = b_n = 0$. Then
\begin{align*}
\left(\int_{H_p} \tau(\mathbf a, p) * \tau(-\mathbf b, p) \right) - \delta_{\mathbf a, \mathbf b} \deg(\mu_{\mathbf a}(p)) \ll \frac 1p \deg(\mu_{\mathbf a}(p))^{1/2} \deg(\mu_{\mathbf b}(p))^{1/2} \,,
\end{align*}
where the implicit constant is allowed to depend on $\mathbf a$ and $\mathbf b$ but not on $p$.
\end{proposition}

The proof of Proposition~\ref{keyboundnegativeresultweak} uses various arguments, one of which is the following trivial bound.

\begin{lemma}\label{keyboundtrivialcase}The conclusion of Proposition~\ref{keyboundnegativeresultweak} holds when $\deg(\mu_{\mathbf b}(p)) \geq p^2 \cdot \deg(\mu_{\mathbf a}(p))$.\end{lemma}

\begin{proof}
We may bound the integral using Lemma~\ref{integraltolattices}, by counting pairs $(L_1, L)$ as in the statement of the lemma. First, because $\mathbf a$ and $\mathbf b$ may be assumed fixed, the number of possibilities for $L_1$ is bounded. To be precise, the inclusions $L_1 \supset L \supset p^{a_1} L_0$ and $p^{b_1 }L_1 \subset L \subset L_0$ imply that there are at most $(a_1 + b_1 + 1)^n$ possibilities for $L_1$. Second, the number of possibilities for $L$ is trivially bounded by $\deg(\mu_{\mathbf a}(p))$. Therefore the left-hand side in Proposition~\ref{keyboundnegativeresultweak} (even without the subtraction) is bounded up to a constant by
\[ \deg(\mu_{\mathbf a}(p)) \leq \frac 1p \deg(\mu_{\mathbf a}(p))^{1/2} \deg(\mu_{\mathbf b}(p))^{1/2} \,. \qedhere\]
\end{proof}

To prove Proposition~\ref{keyboundnegativeresultweak}, it suffices to prove an upper bound of the form
\[ \left(\int_{H_p} \tau(\mathbf a, p) * \tau(-\mathbf b, p) \right) - \delta_{\mathbf a, \mathbf b} \deg(\mu_{\mathbf a}(p))\ll \frac1{p^2} \cdot \max \left(  \deg(\mu_{\mathbf a}(p)),  \deg(\mu_{\mathbf b}(p)) \right) \,. \]
Indeed, by symmetry and by Lemma~\ref{keyboundtrivialcase} we must only consider the situation where
\[\deg(\mu_{\mathbf a}(p)) \leq \deg(\mu_{\mathbf b}(p)) \leq p^2 \deg(\mu_{\mathbf a}(p)) \,, \]
in which case the above bound is at least as strong as what is needed.

The other type of argument we will use is the following. Let $G$ be any (abstract) group and $f : X \to Y$ a $G$-equivariant map between finite transitive $G$-sets. Then the preimages $f^{-1}(y)$ have the same cardinality. In particular, when $S \subset X$, a bound of the form $|f(S)| \leq \delta |Y|$ implies $|S| \leq \delta |X|$. We will apply this principle with $G = \operatorname{GL}_n(\mathbb Z / p^a \mathbb Z)$, $X$ a set of subgroups of $(\mathbb Z / p^a \mathbb Z)^n$, and $f$ reduction mod $p$.

\begin{lemma}\label{invariantfactorswithrestrictionbound}
Let $\mathbf a\in \mathbb N^n$ be a decreasing tuple with $a_n = 0$. Let $1 \leq m \leq n$ and consider the subgroups $L \subset (\mathbb Z / p^{a_1} \mathbb Z)^n$ with the following properties:
\begin{itemize}
\item The quotient $(\mathbb Z / p^{a_1} \mathbb Z)^n / L$ has invariant factors given by $\mathbf a$.
\item $L \subset (p\mathbb Z / p^{a_1} \mathbb Z)^m \oplus (\mathbb Z / p^{a_1} \mathbb Z)^{n-m}$.
\end{itemize}
Let $d = \#\{i :  a_i = 0\}$. Then the number of such subgroups is bounded up to a constant by $p^{ -md} \deg(\mu_{\mathbf a}(p))$, where the constant does not depend on $\mathbf a$ nor $p$.\end{lemma}

\begin{proof}The group $G = \operatorname{GL}_n(\mathbb Z / p^{a_1} \mathbb Z)$ acts transitively on the subgroups $L$ whose quotient has invariant factors given by $\mathbf a$. There are precisely $\deg(\mu_{\mathbf a}(p))$ of those. The reduction mod $p$ of all such $L$ is a subspace of $\mathbb F_p^n$ of dimension $d$. The group $G$ acts transitively on subspaces of given dimension in a way compatible with reduction mod $p$. There are $\asymp p^{d(n-d)}$ such subspaces. On the other hand, the subgroups $L$ as in the statement have reduction mod $p$ lying in a fixed $(n-m)$-dimensional subspace. When $n-m < d$ there is nothing to prove. When $n-m \geq d$, the number of such subspaces is $\asymp p^{d(n-m-d)}$. We conclude that the number of subgroups we want to count, is bounded up to a constant by
\[ \frac{p^{d(n-m-d)}}{p^{d(n-d)}} \deg(\mu_{\mathbf a}(p)) = p^{-md}\deg(\mu_{\mathbf a}(p)) \,. \qedhere\]
\end{proof}

Finally, we will use the perfect pairing on $(\mathbb Z / p^a \mathbb Z)^n$, which provides a notion of duality. Specifically, denote by $\langle \cdot, \cdot \rangle$ the component-wise pairing in the standard basis. When $L \subset (\mathbb Z / p^a \mathbb Z)^n$, define $L^* = \{x : \left\langle x, L \right\rangle = 0 \}$. Then $(L^*)^* = L$, duality reverses inclusions and if $(\mathbb Z / p^a \mathbb Z)^n / L$ has invariant factors $(p^{a_i})$ then $(\mathbb Z / p^a \mathbb Z)^n / L$ has invariant factors $(p^{a - a_i})$. For an explicit example, the dual of $\bigoplus p^{a - a_i} \mathbb Z / p^{a} \mathbb Z$ is $\bigoplus p^{a_i} \mathbb Z / p^{a} \mathbb Z$.

We have the following dual version of Lemma~\ref{invariantfactorswithrestrictionbound}.

\begin{lemma}\label{invariantfactorswithrestrictionbounddual}Let $\mathbf a\in \mathbb N^n$ be a decreasing tuple with $a_n = 0$. Let $1 \leq m \leq n$ and consider the subgroups $L \subset (\mathbb Z / p^{a_1} \mathbb Z)^n$ with the following properties:
\begin{itemize}
\item The quotient $(\mathbb Z / p^{a_1} \mathbb Z)^n / L$ has invariant factors given by $\mathbf a$.
\item $L \supset (p^{a_1 - 1}\mathbb Z / p^{a_1} \mathbb Z)^m \oplus \{0\}^{n-m}$.
\end{itemize}
Let $d = \#\{i :  a_i = a_1\}$. Then the number of such subgroups is bounded up to a constant by $p^{ -md} \deg(\mu_{\mathbf a}(p))$, where the constant does not depend on $\mathbf a$ nor $p$.\end{lemma}

\begin{proof}
For $L$ as in the statement, we have that $L^*$ satisfies the conditions in Lemma~\ref{invariantfactorswithrestrictionbound}. The only observation we have to make is that the tuple $\mathbf a^* := (a_1 - a_n, \ldots, a_1 - a_2, a_1 - a_1)$ satisfies $\deg(\mu_{\mathbf a^*}(p)) = \deg(\mu_{\mathbf a}(p))$; this is just the statement that $\deg(\mu_{-\mathbf a}(p)) = \deg(\mu_{\mathbf a}(p))$.
\end{proof}

With these three ingredients we are ready to prove the key bound.

\begin{proof}[Proof of Proposition~\ref{keyboundnegativeresultweak}] By symmetry we may assume that $\deg(\mu_{\mathbf a}(p)) \leq \deg(\mu_{\mathbf b}(p))$. To bound the integral we use Lemma~\ref{integraltolattices}. We must bound the number of pairs $(L_1, L)$ with $L_1$ adapted and different from $L_0$, and $L \subset L_0 \cap L_1$ for which $L_0 / L$ has invariant factors given by $\mathbf a$ and those of $L_1 / L$ are given by $\mathbf b$. Write $L_1 = \bigoplus p^{t_i} \mathbb Z_p$ with the $t_i \in \mathbb Z$. As in the proof of Lemma~\ref{keyboundtrivialcase} there are only finitely many possibilities for $L_1$, so we may assume $L_1$ is fixed.

Suppose there exists $t_i > 0$. Then up to permutation of factors, $L$ lies in $p \mathbb Z_p \bigoplus \mathbb Z_p^{n-1}$, so that the subgroup $L / p^{a_1} L_0 \subset L_0 / p^{a_1} L_0 \cong (\mathbb Z/p^{a_1} \mathbb Z)^n$ satisfies the conditions of Lemma~\ref{invariantfactorswithrestrictionbound} with $d, m \geq 1$. Therefore there are at most (up to a constant) $p^{-1} \deg(\mu_{\mathbf a}(p))$ possibilities for $L$ in this case, and this bound is good enough because $\deg(\mu_{\mathbf a}(p)) \leq \deg(\mu_{\mathbf b}(p))$.

We may now assume that all $t_i \leq 0$, so that $L_0 \subset L_1$. If two distinct $t_i  < 0$, then we may apply Lemma~\ref{invariantfactorswithrestrictionbound} again, this time to $L / p^{b_1} L_1 \subset L_1 / p^{b_1}L_1$ and with $m \geq 2$, $d \geq 1$. We find that there are at most $p^{-2} \deg(\mu_{\mathbf b}(p))$ possibilities for $L$ in this case, and this bound is good enough as remarked below Lemma~\ref{keyboundtrivialcase}.

We may now assume that a single $t_i < 0$ and all others are $0$. If $b_1 \leq a_1$, then $L$ contains $p^{a_1 + t_i} e_i$ and in particular $p^{a_1 - 1} e_i$. We may then apply Lemma~\ref{invariantfactorswithrestrictionbounddual} to $L / p^{a_1} L_0 \subset L_0 / p^{a_1} L_0$, with $m, d \geq 1$ and conclude that there are at most $p^{-1} \deg(\mu_{\mathbf a}(p))$ possibilities for $L$. If $b_1 > a_1$, then $L_0$ contains all the elements $p^{b_1 - 1} e_j$ with $j \neq i$. Because $n \geq 3$, there are at least two of these. We may now apply Lemma~\ref{invariantfactorswithrestrictionbounddual} to $L/p^{b_1} L_1 \subset L_1 / p^{b_1} L_1$ with $m \geq 2$ and $d \geq 1$, and conclude that there are at most (up to a constant) $p^{-2} \deg(\mu_{\mathbf b}(p))$ possibilities for $L$ in this case. Again, this is sufficient by the comment below Lemma~\ref{keyboundtrivialcase}.
\end{proof}

Let $\mathcal P$ denote the set of all primes.

\begin{proposition}\label{optimalityslthreefull}Let $n \geq 3$, $\kappa > 0$ and $M \geq 3$. For every nonzero $\omega \in \mathcal{H}_{\mathcal P, M}^{\leq\kappa}$ the convolution operator $k_{f} = \omega * \omega^{*}$ satisfies
\[ \frac{\int_{\mathbf H(\mathbb A_f)} k_f}{k_{f}(1)} \ll (\log \log M)^{C} \,, \]
for a constant $C$ that is allowed to depend on $n$ and $\kappa$.
\end{proposition}

\begin{proof}We may find a finite family of squarefree integers $(n_{i})_{i \in I} \in [0, M]$ and for every prime $p \mid n_i$ a scalar multiple of elementary Hecke operator $\omega_{i, p} \in \mathcal{H}_{p}^{\leq \kappa}$, such that $\omega = \sum_{i \in I} \prod_{p \mid n_{i}} \omega_{i, p}$. Moreover, we may assume that the $\prod_{p \mid n_{i}} \omega_{i, p}$ have disjoint supports. If $X_\kappa$ denotes the set of cocharacters of $\mathbf H_{\overline {\mathbb Q}}$ of norm at most $\kappa$, then $|X_\kappa|$ is bounded. If $\omega(m)$ denotes the number of prime divisors of $m$, then every $m \leq M$ occurs at most $|X_\kappa|^{\omega(m)}$ times as an integer $n_i$ in the family.

Let $C$ be the largest implicit constant in Proposition~\ref{keyboundnegativeresultweak} when $\mathbf a$ and $\mathbf b$ run trough the tuples with $\mu_{\mathbf a}, \mu_{\mathbf b} \in X_\kappa$.

When $p \nmid n_{i}$, define $\omega_{i, p} = 1_{K_{p}}$. We then have
\begin{align*}
k_{f} = \sum_{i, j \in I} \prod_{\substack{p \mid n_{i} \\ q \mid n_{j}}} \omega_{i, p}* \omega_{j, q}^{*} = \sum_{i, j \in I} \prod_{p \mid n_{i} n_{j}} \omega_{i, p}* \omega_{j, p}^{*} \,.
\end{align*}
The disjointness of the supports of the $\prod_{p \mid n_{i}} \omega_{i, p}$ implies
\begin{align*}
k_{f}(1) & = \sum_{i \in I} \prod_{p \mid n_{i}} \lVert \omega_{i, p} \rVert_{2}^{2} \,.
\end{align*}
Using Proposition~\ref{keyboundnegativeresultweak} we have
\begin{align*}
\int_{\mathbf H(\mathbb A_f)} k_f & = \sum_{i, j \in I} \prod_{p \mid n_{i} n_{j}} \int_{H_p }(\omega_{i, p} * \omega_{j, p}^{*}) \\
& \leq \sum_{i, j \in I} \prod_{p \mid n_{i} n_{j}} C p^{-1} \lVert \omega_{i, p} \rVert_{2} \lVert \omega_{j, p} \rVert_{2} \\
& = \sum_{i, j \in I} \operatorname{lcm}(n_{i}, n_{j})^{-1} C^{\omega(\operatorname{lcm}(n_{i}n_{j}))} \prod_{p \mid n_{i} n_{j}} \lVert \omega_{i, p} \rVert_{2} \lVert \omega_{j, p} \rVert_{2} \\
& \leq \sum_{d \leq M} \frac{d}{C^{\omega(d)}} \left( \sum_{d \mid n_{i}} \frac{C^{\omega(n_{i})}}{n_{i} } \prod_{p \mid n_{i}} \lVert \omega_{i, p} \rVert_{2} \right)^{2} \,,
\end{align*}
where in the last equality we have written $\operatorname{lcm}(n_{i} n_{j}) = n_{i}n_{j} / d$ with $d = \gcd(n_{i}, n_{j})$, and then extended the sum to run over all $d \mid n_{i}, n_{j}$.
By Cauchy--Schwarz, this is bounded by
\begin{align*}
\sum_{d \leq M} \frac{d}{C^{\omega(d)}} \left( \sum_{d \mid n_{i}} \frac{C^{2 \omega(n_{i})}}{n_{i}^{2}} \right) \left( \sum_{d \mid n_{i}} \prod_{p \mid n_{i}} \lVert \omega_{i, p} \rVert_{2}^{2} \right) \,.
\end{align*}
It follows that
\begin{align*}
\frac{\int_{\mathbf H(\mathbb A_f)} k_f}{k_{f}(1)} & \leq \sup_{\substack{n \leq M \\ \square \text{-free}}} \sum_{\substack{d, n_{i} \\ d \mid n, n_{i}}} \frac{d}{C^{\omega(d)}} \frac{C^{2 \omega(n_{i})}}{n_{i}^{2}} \,.
\end{align*}
Because every $m$ occurs at most $|X_\kappa|^{\omega(m)}$ times in the family $(n_{i})_{i \in I}$, this is at most
\begin{align*}
& \ll_{\kappa, \epsilon} \sup_{\substack{n \leq M \\ \square \text{-free}}} \sum_{\substack{d \mid n, m \\m \; \square \text{-free}}} \frac{d}{C^{\omega(d)}} \frac{(C^2|X_\kappa|)^{\omega(m)}}{m^{2}} \\
& = \sup_{\substack{n \leq M \\ \square \text{-free}}} \sum_{\substack{d \mid n}} \frac{(C|X_\kappa|)^{\omega(d)}}{d} \sum_{m \; \square \text{-free}} \frac{(C^2|X_\kappa|)^{\omega(m)}}{m^{2}} \\
& \ll \sup_{\substack{n \leq M \\ \square \text{-free}}} \sum_{\substack{d \mid n}} \frac{(C|X_\kappa|)^{\omega(d)}}{d} \,,
\end{align*}
because the sum over $m$ is convergent.
The latter expression is largest when $n$ has the smallest possible prime factors. So take $n = \prod_{p \leq x} p$ for some $x > 1$. Then $n \leq M$ implies $x \ll \log M$ by Chebyshev's estimate, and we have
\begin{align*}
\sum_{\substack{d \mid n}} \frac{(C\lvert X_\kappa \rvert)^{\omega(d)}}{d} & = \prod_{p \leq x} \left( 1 + \frac{C\lvert X_\kappa \rvert}{p} \right) \\
& \ll (\log x)^{C\lvert X_\kappa \rvert} \\
& \ll (\log \log M)^{C\lvert X_\kappa \rvert} \,,
\end{align*}
where we have used the Mertens' theorem in the first estimate.
\end{proof}

\begin{remark}\label{remarkglnboundoptimality} It is likely that a stronger version of Proposition~\ref{keyboundnegativeresultweak} is still true. Namely, when $n \geq 4$ we expect that a a similar bound holds with the power of $p$ in the right-hand side replaced by $p^{-3/2}$. Moreover, when $n = 3$ we expect that this can also be shown, except in the situation of Lemma~\ref{winnerkpglthree}. We have partial proofs of these statements, which use mostly the same arguments as in the proof of Proposition~\ref{keyboundnegativeresultweak}, and contains lots of casework. However a few cases remain, where we notably require bounds for specific Hall polynomials. We hope to settle this stronger version in the near future.

The stronger version would imply that the upper bound in Proposition~\ref{optimalityslthreefull} can be replaced by $1$ (no growth at all) for $n \geq 4$. Indeed, the exponent $3/2$ is then propagated throughout the proof until the very last lines, where we may then use that $\prod_{p \leq X} \left( 1 + p^{-3/2} \right)$ is bounded.
\end{remark}

\begin{remark}When $n = 3$, Proposition~\ref{keyboundnegativeresultweak} is not as strong as we would like in different ways. It would be desirable to remove or make explicit the dependence on $\kappa$ and to make the power of $\log \log M$ match with the exponent in Proposition~\ref{existencetestfunctionbig}. Such bounds should follow from the stronger version of Proposition~\ref{keyboundnegativeresultweak} that we expect to hold.\end{remark}

\begin{acknowledgements} I am very grateful to Farrell Brumley for introducing me to the circle of questions that led us to study toric periods for $\mathbf{PGL}_n$, and for informing me about the period relation that put the result into context. I also thank Simon Marshall for spotting an error in an earlier version of the manuscript, and thank him together with Jasmin Matz for their comments after a careful reading.
\end{acknowledgements}

\bibliographystyle{plain}
\bibliography{../../../library/_bib/bibliography}

\end{document}